\documentclass[a4paper,fleqn]{cas-dc2}

\usepackage[numbers]{natbib}

\usepackage{multirow}
\usepackage{array}
\usepackage{derivative}
\usepackage{hyperref}

\usepackage{amsmath}
\usepackage{amssymb}
\usepackage{eurosym}
\usepackage{amsthm}

\usepackage[acronym, nonumberlist]{glossaries}

\newtheorem{theorem}{Proposition}
\newtheorem{remark}{Remark}

\newglossary[slg]{symbolslist}{sym}{sbl}{List of symbols}

\makeglossaries
\setacronymstyle{long-short}

\newacronym{ecms}{ECMS}{equivalent consumption minimization strategy}
\newacronym{ghg}{GHG}{greenhouse gas}
\newacronym{imo}{\textit{IMO}}{\textit{International Maritime Organization}}
\newacronym{kt}{K\textsubscript{T}}{thrust coefficient}
\newacronym{kq}{K\textsubscript{Q}}{torque coefficient}
\newacronym{lhs}{LHS}{left-hand side}
\newacronym{milp}{MILP}{mixed integer linear program}
\newacronym{minlp}{MINLP}{mixed integer non-linear program}
\newacronym{mpc}{MPC}{model predictive control}
\newacronym{nlp}{NLP}{nonlinear program}
\newacronym{nox}{NOx}{nitrogen oxide}
\newacronym{ocv}{OCV}{open circuit voltage}
\newacronym{owd}{OWD}{open water diagram}
\newacronym{pm2.5}{PM\textsubscript{2.5}}{fine particulate matter}
\newacronym{poly2}{poly2}{second-order polynomial fit}
\newacronym{poly3}{poly3}{third-order polynomial fit}
\newacronym{soc}{SOC}{state of charge}
\newacronym{sox}{SOx}{sulphur oxide}
\newacronym{1dof}{1-DOF}{one-degree-of-freedom}
\newacronym{1q}{1Q}{one-quadrant}
\newacronym{4q}{4Q}{four-quadrant}

\renewcommand{\vec}[1]{\mathrm{\underline{#1}}}

\newglossaryentry{a_c}{
	name={\ensuremath{a_{\mathrm{c}}}},
    user1={\ensuremath{a_{\mathrm{c,0}}}},
    user2={\ensuremath{a_{\mathrm{c,1}}}},
	description={coefficient of the converter loss function, $i \in {0, 1}$},
	sort={ag}, 
	type={symbolslist}
}

\newglossaryentry{a_Q}{
	name={\ensuremath{\tilde{a}_{\mathrm{Q,}i}}},
    user1={\ensuremath{\tilde{a}_{\mathrm{Q,0}}}},
    user2={\ensuremath{\tilde{a}_{\mathrm{Q,1}}}},
    user3={\ensuremath{\tilde{a}_{\mathrm{Q,2}}}},
	description={coefficient of the torque coefficient approximation function, $i \in {0, 1, 2}$},
	sort={aq}, 
	type={symbolslist}
}

\newglossaryentry{a_T}{
	name={\ensuremath{\tilde{a}_{\mathrm{T,}i}}},
    user1={\ensuremath{\tilde{a}_{\mathrm{T,0}}}},
    user2={\ensuremath{\tilde{a}_{\mathrm{T,1}}}},
    user3={\ensuremath{\tilde{a}_{\mathrm{T,2}}}},
	description={coefficient of the thrust coefficient approximation function, $i \in {0, 1, 2}$},
	sort={at}, 
	type={symbolslist}
}

\newglossaryentry{b}{
	name={\ensuremath{b}},
    user1={\ensuremath{b'}},
    user2={\ensuremath{b''}},
	description={auxiliary variable for squared vessel speed},
	sort={b}, 
	type={symbolslist}
}

\newglossaryentry{D_p}{
	name={\ensuremath{D_{\mathrm{p}}}},
	description={diameter of propeller},
	sort={Dp}, 
	type={symbolslist}
}

\newglossaryentry{dE_bat}{
	name={\ensuremath{\Delta E_{\mathrm{batt}}}},
	description={change in battery energy},
	sort={debat}, 
	type={symbolslist}
}

\newglossaryentry{dE_batmax}{
	name={\ensuremath{\Delta E_{\mathrm{batt,min/max}}}},
    user1={\ensuremath{\Delta E_{\mathrm{batt,min}}}},
    user2={\ensuremath{\Delta E_{\mathrm{batt,max}}}},
	description={limits for the change in battery energy},
	sort={debatmax}, 
	type={symbolslist}
}

\newglossaryentry{dE_em}{
	name={\ensuremath{\Delta E_{\mathrm{EM}}}},
	description={change in electric machine input energy},
	sort={deem}, 
	type={symbolslist}
}

\newglossaryentry{dE_p}{
	name={\ensuremath{\Delta E_{\mathrm{p}}}},
	description={change in propulsive energy},
	sort={dep}, 
	type={symbolslist}
}

\newglossaryentry{E_bat}{
	name={\ensuremath{E_{\mathrm{batt}}}},
	description={battery energy},
	sort={Ebat}, 
	type={symbolslist}
}

\newglossaryentry{E_batmax}{
	name={\ensuremath{E_{\mathrm{batt,max}}}},
	description={maximum battery energy},
	sort={Ebatmax}, 
	type={symbolslist}
}

\newglossaryentry{F_T}{
	name={\ensuremath{F_{\mathrm{T}}}},
	description={force of propeller thrust},
	sort={FT}, 
	type={symbolslist}
}

\newglossaryentry{F_D}{
	name={\ensuremath{F_{\mathrm{D}}}},
	description={total drag},
	sort={FD}, 
	type={symbolslist}
}

\newglossaryentry{F_H}{
	name={\ensuremath{F_{\mathrm{H}}}},
	description={hydrodynamic hull force from drift},
	sort={FH}, 
	type={symbolslist}
}

\newglossaryentry{F_P}{
	name={\ensuremath{F_{\mathrm{P}}}},
	description={rotational damping force},
	sort={FP}, 
	type={symbolslist}
}

\newglossaryentry{F_R}{
	name={\ensuremath{F_{\mathrm{R}}}},
	description={rudder steering force},
	sort={FR}, 
	type={symbolslist}
}

\newglossaryentry{D_R}{
	name={\ensuremath{D_{\mathrm{R}}}},
	description={rudder drag force},
	sort={DR}, 
	type={symbolslist}
}

\newglossaryentry{F_bat}{
	name={\ensuremath{F_{\mathrm{batt}}}},
	description={purely fictive "force" representing derivative of  battery energy with regard to path parameter},
	sort={Fbat}, 
	type={symbolslist}
}

\newglossaryentry{F_batd}{
	name={\ensuremath{F_{\mathrm{batt,d}}}},
	description={purely fictive "force" representing derivative of battery losses with regard to path parameter},
	sort={Fbatd}, 
	type={symbolslist}
}

\newglossaryentry{F_bati}{
	name={\ensuremath{F_{\mathrm{batt,i}}}},
	description={purely fictive "force" representing derivative of internal battery energy with regard to path parameter},
	sort={Fbati}, 
	type={symbolslist}
}

\newglossaryentry{F_em}{
	name={\ensuremath{F_{\mathrm{EM}}}},
	description={purely fictive "force" representing derivative of electric machine energy with regard to path parameter},
	sort={Fem}, 
	type={symbolslist}
}

\newglossaryentry{F_fc}{
	name={\ensuremath{F_{\mathrm{fc}}}},
	description={purely fictive "force" representing derivative of fuel cell energy with regard to path parameter},
	sort={Ffc}, 
	type={symbolslist}
}

\newglossaryentry{F_fci}{
	name={\ensuremath{F_{\mathrm{fc,i}}}},
	description={purely fictive "force" representing derivative of internal fuel cell energy with regard to path parameter},
	sort={Ffci}, 
	type={symbolslist}
}

\newglossaryentry{F_c}{
	name={\ensuremath{F_{\mathrm{c}}}},
	description={purely fictive "force" representing derivative of converter energy with regard to path parameter},
	sort={Fg}, 
	type={symbolslist}
}

\newglossaryentry{F_ci}{
	name={\ensuremath{F_{\mathrm{c,i}}}},
	description={purely fictive "force" representing derivative of internal converter energy with regard to path parameter},
	sort={Fgi}, 
	type={symbolslist}
}

\newglossaryentry{F_dE_p}{
	name={\ensuremath{F_{\Delta \mathrm{E}_{\mathrm{p}}}}},
	description={purely fictive "force" representing derivative of propeller input energy with regard to path parameter},
	sort={Fp}, 
	type={symbolslist}
}

\newglossaryentry{f_t}{
	name={\ensuremath{f_{\mathrm{t}}}},
	description={thrust deduction factor},
	sort={ft}, 
	type={symbolslist}
}

\newglossaryentry{f_w}{
	name={\ensuremath{f_{\mathrm{w}}}},
	description={wake fraction coefficient},
	sort={fw}, 
	type={symbolslist}
}

\newglossaryentry{I}{
	name={\ensuremath{I_{\mathrm{0}}}},
	description={moment of inertia},
	sort={I}, 
	type={symbolslist}
}
\newglossaryentry{I_add}{
	name={\ensuremath{\tilde{I}}},
	description={added moments of inertia},
	sort={Iadd}, 
	type={symbolslist}
}

\newglossaryentry{I_w}{
	name={\ensuremath{I_{\mathrm{w}}}},
	description={moment of inertia of displaced water},
	sort={Iw}, 
	type={symbolslist}
}

\newglossaryentry{J}{
	name={\ensuremath{J}},
	description={advance coefficient},
	sort={J}, 
	type={symbolslist}
}

\newglossaryentry{k_c}{
	name={\ensuremath{k_{\mathrm{c}}}},
	description={number of general energy converters turned on as a function of the path variable},
	sort={kc}, 
	type={symbolslist}
}

\newglossaryentry{k_dEp}{
	name={\ensuremath{k_{\mathrm{dE}_\mathrm{p},i}}},
    user1={\ensuremath{k_{\mathrm{dE}_\mathrm{p},0}}},
    user2={\ensuremath{k_{\mathrm{dE}_\mathrm{p},1}}},
    user3={\ensuremath{k_{\mathrm{dE}_\mathrm{p},2}}},
	description={coefficient of the thrust coefficient approximation function, $i \in {0, 1, 2}$},
	sort={at}, 
	type={symbolslist}
}

\newglossaryentry{L_H}{
	name={\ensuremath{L_{\mathrm{H}}}},
	description={arm for hydrodynamic hull force},
	sort={LH}, 
	type={symbolslist}
}

\newglossaryentry{L_P}{
	name={\ensuremath{L_{\mathrm{P}}}},
	description={arm for damping force},
	sort={LP}, 
	type={symbolslist}
}

\newglossaryentry{L_R}{
	name={\ensuremath{L_{\mathrm{R}}}},
	description={arm for rudder steering force},
	sort={LR}, 
	type={symbolslist}
}

\newglossaryentry{M}{
	name={\ensuremath{\vec{M}}},
	description={mass and inertia matrix},
	sort={M}, 
	type={symbolslist}
}

\newglossaryentry{m_add}{
	name={\ensuremath{\tilde{m}}},
	description={added mass of the vessel, m $\cdot k_1$},
	sort={madd}, 
	type={symbolslist}
}

\newglossaryentry{m_vessel}{
	name={\ensuremath{m}},
	description={mass of the vessel},
	sort={m}, 
	type={symbolslist}
}

\newglossaryentry{n_p}{
	name={\ensuremath{\tilde{n}_{\mathrm{p}}}},
	description={auxiliary variable for propeller shaft speed},
	sort={np}, 
	type={symbolslist}
}

\newglossaryentry{P_aux}{
	name={\ensuremath{P_{\mathrm{aux}}}},
	description={auxiliary loads during vessel operation},
	sort={Paux}, 
	type={symbolslist}
}

\newglossaryentry{P_batmax}{
	name={\ensuremath{P_{\mathrm{batt,dis/cha,max}}}},
    user1={\ensuremath{P_{\mathrm{batt,dis,max}}}},
    user2={\ensuremath{P_{\mathrm{batt,cha,max}}}},
	description={maximum battery dis-/charge power},
	sort={Pbatmax}, 
	type={symbolslist}
}

\newglossaryentry{P_fcmax}{
	name={\ensuremath{P_{\mathrm{fc,min/max}}}},
    user1={\ensuremath{P_{\mathrm{fc,min}}}},
    user2={\ensuremath{P_{\mathrm{fc,max}}}},
	description={minimum/maximum output power of the fuel cell},
	sort={Pfcmax}, 
	type={symbolslist}
}

\newglossaryentry{P_gi}{
	name={\ensuremath{P_{\mathrm{g,i},k}}},
	description={internal generator power},
	sort={Pgi}, 
	type={symbolslist}
}

\newglossaryentry{q}{
	name={\ensuremath{\vec{q}}},
    user1={\ensuremath{q}},
	description={state vector of the vessel model},
	sort={q}, 
	type={symbolslist}
}

\newglossaryentry{Q_p}{
	name={\ensuremath{Q_{\mathrm{p}}}},
	description={torque of propeller},
	sort={Qp}, 
	type={symbolslist}
}

\newglossaryentry{R}{
	name={\ensuremath{\vec{R}}},
	description={rotation matrix},
	sort={R}, 
	type={symbolslist}
}

\newglossaryentry{rho_sw}{
	name={\ensuremath{\rho_{\mathrm{sw}}}},
	description={density of seawater},
	sort={rhosw}, 
	type={symbolslist}
}

\newglossaryentry{s}{
	name={\ensuremath{\vec{s}}},
    user1={\ensuremath{\vec{s}'}},
    user2={\ensuremath{\vec{s}''}},
    user3={\ensuremath{s}},
    user4={\ensuremath{s'}},
    user5={\ensuremath{s''}},
	description={path of the vessel as a function of the path coordinate},
	sort={s}, 
	type={symbolslist}
}

\newglossaryentry{sigma}{
	name={\ensuremath{\sigma}},
	description={path coordinate},
	sort={rhosw}, 
	type={symbolslist}
}

\newglossaryentry{soc_max}{
	name={\ensuremath{\mathrm{SOC}_{\mathrm{min/max}}}},
    user1={\ensuremath{\mathrm{SOC}_{\mathrm{min}}}},
    user2={\ensuremath{\mathrm{SOC}_{\mathrm{max}}}},
	description={minimum/maximum state of charge},
	sort={socmax}, 
	type={symbolslist}
}

\newglossaryentry{T_p}{
	name={\ensuremath{T_{\mathrm{p}}}},
	description={thrust of propeller},
	sort={Tp}, 
	type={symbolslist}
}

\newglossaryentry{u}{
	name={\ensuremath{\vec{u}}},
    user1={\ensuremath{\vec{\tilde{u}}}},
	description={control input vector},
	sort={u}, 
	type={symbolslist}
}

\newglossaryentry{v_a}{
	name={\ensuremath{v_{\mathrm{a}}}},
	description={advance speed},
	sort={va}, 
	type={symbolslist}
}

\newglossaryentry{v_s}{
	name={\ensuremath{v_{\mathrm{s}}}},
	description={vessel speed},
	sort={vs}, 
	type={symbolslist}
}

\newglossaryentry{w_t}{
	name={\ensuremath{\omega_{\mathrm{T}}}},
	description={weighting factor for time in objective function},
	sort={wt}, 
	type={symbolslist}
}

\newglossaryentry{y_t}{
	name={\ensuremath{y_{\mathrm{t}}}},
	description={auxiliary variable for the derivative of the time with regard to the path variable},
	sort={yt}, 
	type={symbolslist}
}

\newglossaryentry{z}{
	name={\ensuremath{z}},
	description={auxiliary variable for the product of speed and shaft speed},
	sort={z}, 
	type={symbolslist}
}

\newcommand{\reals}{\mathbb{R}}
\newcommand{\naturals}{\mathbb{N}}
\newcommand{\figref}[1]{Figure~\ref{#1}}

\newcommand{\appref}[1]{Appendix~\ref{#1}}

\newcommand{\lMultIneq}[1]{\lambda_{\mathrm{#1}}}
\newcommand{\lMultDiff}[1]{\psi_{\mathrm{#1}}}

\newcommand{\rowVecTrans}[1]{\begin{bmatrix}#1\end{bmatrix}^\top}

\def\tsc#1{\csdef{#1}{\textsc{\lowercase{#1}}\xspace}}
\tsc{WGM}
\tsc{QE}

\begin{document}
\let\WriteBookmarks\relax
\def\floatpagepagefraction{1}
\def\textpagefraction{.001}

\shorttitle{Integrated supervisory control and fixed path speed trajectory generation}    

\shortauthors{Ritari, Katzenburg, Oliveira, Tammi}  

\title [mode = title]{Integrated supervisory control and fixed path speed trajectory generation for hybrid electric ships via convex optimization}  




\author[1]{Antti Ritari}[orcid=0000-0002-6883-1447]
\cormark[1]
\fnmark[1]
\ead{antti.ritari@aalto.fi}
\address[1]{School of Engineering, Department of Mechanical Engineering, Aalto University, Otakaari 4, 02150 Espoo, Finland}

\author[1]{Niklas Katzenburg}[orcid=0009-0001-7878-0866]
\fnmark[1]
\ead{niklas.katzenburg@alumni.aalto.fi}

\author[2]{Fabricio Oliveira}[orcid=0000-0003-0300-9337]
\ead{fabricio.oliveira@aalto.fi}
\author[1]{Kari Tammi}[orcid=0000-0001-9376-2386]
\ead{kari.tammi@aalto.fi}

\address[2]{School of Science, Department of Mathematics and Systems Analysis, Aalto University, Otakaari 1, 02150 Espoo, Finland}

\cortext[1]{Corresponding author}

\fntext[1]{Equal contribution}


\begin{abstract}
Battery-hybrid power source architectures can reduce fuel consumption and emissions for ships with diverse operation profiles. 
However, conventional control strategies may fail to improve performance if the future operation profile is unknown to the controller. 
This paper proposes a guidance, navigation, and control (GNC) function that integrates trajectory generation and hybrid power source supervisory control. 
We focus on time and fuel optimal path-constrained trajectory planning. This problem is a nonlinear and nonconvex optimal control problem, which means that it is not readily amenable to efficient and reliable solution onboard. 
We propose a nonlinear change of variables and constraint relaxations that transform the nonconvex planning problem into a convex optimal control problem. 
The nonconvex three-degree-of-freedom dynamics, hydrodynamic forces, fixed pitch propeller, battery, and general energy converter (e.g., fuel cell or generating set) dissipation constraints are expressed in convex functional form.
A condition derived from Pontryagin's Minimum Principle guarantees that, when satisfied, the solution of the relaxed problem provides the solution to the original problem.
The validity and effectiveness of this approach are numerically illustrated for a battery-hybrid ship in model scale. 
First, the convex hydrodynamic hull and rudder force models are validated with towing tank test data. 
Second, optimal trajectories and supervisory control schemes are evaluated under varying mission requirements. 
The convexification scheme in this work lays the path for the employment of mature, computationally robust convex optimization methods and creates a novel possibility for real-time optimization onboard future smart and unmanned surface vehicles. 
\end{abstract}


\maketitle

\section{Introduction}
\textbf{Emission abatement drivers.} Rising environmental awareness and international regulations such as the Paris Agreement, which aims to limit “global average temperature to well below 2~°C above pre-industrial levels” \cite{unite_ParisAgreement_2016}, require all industries to take serious efforts to decrease their \gls{ghg} emissions.
Thus, the \gls{imo} adopted the \textit{Initial \gls{imo} Strategy on reduction of \gls{ghg} emissions from ships} in April 2018.
Within this resolution, the \gls{imo} confirms the contribution of international shipping to the goals set in the Paris Agreement.
They strive for a \gls{ghg} emission reduction between 50\% and 70\% by 2050 compared to 2008 \cite{imo_NoteInternationalMaritime_2018}.

Vessels do not only emit \gls{ghg} though but also other pollutants such as \gls{pm2.5}, \glspl{sox} and \glspl{nox}, which impact climate and human health \cite{sofie_CleanerFuelsShips_2018}.
According to \cite{smith_ThirdIMOGHG_2015}, international shipping causes approximately 12\% of annual, global \gls{sox} emissions based on the average values between 2007 and 2012.
Without taking any actions, the authors in \cite{sofie_CleanerFuelsShips_2018} suggest that shipping emissions will lead to approximately 14 million childhood asthma cases – about 16\% of the total estimated 86 million childhood asthma cases – and a total of 403,300 annual premature adult deaths.
Densely populated coastal regions – especially in developing and least developed countries – are the most affected \cite{sofie_CleanerFuelsShips_2018}.

\textbf{Hybrid architectures.} Conventional direct-driven mechanical propulsion exhibits poor efficiency when the vessel operating point lies outside the optimized design point, which increases fuel consumption and emissions. Advanced powertrain configurations based on electrical propulsion, hybrid combination of mechanical and electrical propulsion, and power generation combined with electrical energy storage can deliver the needed adaptability and efficiency.

Electrical propulsion and azimuth thrusters are now the standard architecture in cruise ships that frequently maneuver in ports; a hybrid power source with li-ion batteries as an energy buffer can provide spinning reserve and zero-emission sailing for coastal ferries under strict safety requirements and sensitivity to emissions due to the route's proximity to habitation \cite{ritar_HybridElectricTopology_2020}. Similarly, offshore vessels alternating between transit and dynamic positioning operations benefit from battery spinning reserve; mechanical propulsion with power take-in and take-out function from the parallel electric machine has been shown to reduce emissions and increase the performance of naval vessels performing a combination of patrol and littoral operations \cite{geert_PitchControlShips_2017}.

Fuel cell hybrid power source is a promising solution for countering vessel \gls{ghg} emissions and pollutants, at least locally. Fuel cells convert the chemical energy of fuels to electricity by oxidation and reduction reactions. As fuel cells are not heat engines, they are not bounded by the thermodynamic constraints expressed by Carnot’s law \cite{BMT+20}. The theoretical maximum conversion efficiency is higher compared to internal combustion engines, which lowers fuel consumption and emissions. Proton exchange membrane and solid oxide fuel cells have been identified as the most promising for shipping applications \cite{emsa17}. Load leveling and peak shaving functions provided by an energy buffer are essential in fuel cell power generation configuration for extending the useful lifetime and mitigating the limited power output ramp rate of the stacks.

\textbf{Traditional energy management strategies.} The hybrid configurations introduce a control task that consists of determining the setpoints of the various power converters that constituting the powertrain. The control task is called energy management, supervisory control, or tertiary control in relation to the hierarchical control stack, which also includes primary and secondary control \cite{geert_DesignControlHybrid_2017}. The notion of optimal energy management strategy refers to the exploitation of the degrees of freedom in control to minimize a set of criteria that typically represents fuel consumption or pollutant emissions. An important and challenging characteristic in the charge-sustaining battery-hybrid energy management problem is the constraint that requires the State of Charge (SoC) at the end of the voyage to take a value close to its initial value.

All high-performing energy management strategies follow the principle that the internal combustion engine should be operated at favourable conditions at relatively high loads. The traditional strategies aim to achieve this goal by calculating the power converter set points from rules as a function of various measured vessel quantities. Recently, following the success in the automotive sector, optimal control theory and, in particular, Pontryagin's Minimum Principle (PMP) has been investigated in hybrid vessel control as well \cite{kalik_ShipEnergyManagement_2018}.

\textbf{Motion planning.} Both rule-based and PMP energy management strategies aim to fulfill the operator's request that is unknown in advance to the controller. In unmanned surface vehicles the propulsion and power generation are considered distinct subsystems that receive actuator commands from the guidance, navigation, and control (GNC) system \cite{ZYX+16}. Regardless of manned or unmanned guidance, this function must continuously generate smooth and feasible trajectories for references to the power and propulsion controllers. Information provided by the navigation system, voyage plan, vessel capability, and environmental conditions defines the set of feasible trajectories. The role of path planning and trajectory generation in energy-efficient vessel operation is highlighted by \gls{imo} \cite{imo_2016GuidelinesDevelopment_2016}. The vessel dynamics are emphasized for short sea vessels due to the short overall voyage length and the acceleration and deceleration phases. These phases are especially important in areas with speed limits like archipelagos and close proximity to harbors.

Engine efficiency curve, propeller characteristics, and other power generation-related factors have been recognized to influence the fuel and voyage time optimal trajectory \cite{imo_2016GuidelinesDevelopment_2016, dnvg_MARITIMEFORECAST2050_2019}. Nevertheless, the synthesis of GNC with hybrid power and propulsion architecture energy management has hardly been investigated yet despite the promise of improved performance. The reason behind this fact is associated with the challenging scale and computational complexity of the nonlinear optimal control problem that arises from the vessel dynamics and propulsor characteristics. Although the optimal solution can be approximated by the dynamic programming method \cite{frang_DevelopmentsTrendsChallenges_2020}, it is unsuitable for real-time applications due to the ``curse of dimensionality'' – i.e., the computational effort rises exponentially with the number of states \cite{bohme_HybridSystemsOptimal_2017}. Thus, the requirement for real-time control and autonomous decision-making rule out dynamic programming from consideration. 

The objective of this study is to develop a framework for synthesizing the trajectory simultaneously with the energy management strategy. The performance of this framework concerns not only optimality in terms of voyage time, fuel consumption, or pollutant emission, but also implementability. The latter criterion refers to the reliability, the memory footprint of the computational implementation, and the processor runtime. 

We focus on a trajectory generation problem that takes a previously computed path as input, including spatial constraints (e.g., speed limits). 
A higher-level guidance path planning algorithm may have generated a set of feasible alternative paths between waypoints. These alternatives must be evaluated by generating an optimal trajectory and energy management strategy for each alternative.

\textbf{Convex optimization.} To achieve the required efficiency and reliability for real-time application, we propose formulating the integrated control problem as a convex optimization problem. Convex optimization has become a mature technology, being applied in several fields of engineering and science during the last two decades \cite{hobur_GeometricProgrammingAircraft_2014}.
Examples include – but are not limited to - aircraft design \cite{hobur_GeometricProgrammingAircraft_2014}, hybrid electric vehicles \cite{murgo_ConvexRelaxationsOptimal_2015}, \textit{Formula 1} cars \cite{ebbes_TimeoptimalControlStrategies_2018} and planetary soft landing \cite{acikm_LosslessConvexificationNonconvex_2013}.

Convex optimization problems exhibit a number of favourable properties. Any locally optimal solution is also globally optimal. Infeasibility can be detected unambiguously, i.e., the problem cannot be solved with the given constraints.
Iterative algorithms for solving convex optimization problems self-initialize, which means that they require neither any initial values nor parameter tuning. There also exists a deterministic bound on the number of arithmetic operations needed to solve a convex optimization problem to within any desired accuracy \cite{boyd_ConvexOptimization_2004}. These special properties motivate the use of convex optimization in applications that require fast and reliable autonomous decision making capability.

However, the benefits of convex optimization come at a price: the mathematical model that describes the physical relations must be expressed within the restricted functional forms that yield a feasible convex set. 
The challenge arises from formulating the prevailing physics-based models of hull forces, propulsor thrust and torque, power converters and battery energy storage in this restricted functional form without sacrificing consistency with the first-principles high-fidelity nonconvex models.

In this work, the challenge of ``convexifying'' the original nonconvex problem is met and overcome by formulating a convex three-degree-of-freedom vessel dynamics model. The dynamics are convexified exactly, meaning that we provide a novel equivalent convex formulation. Convex expressions are introduced for hull forces, the fixed pitch propeller, the battery system and converters (internal combustion engine or fuel cell modules).

\textbf{Contribution.} To the authors' knowledge, this model formulation includes three major novelties for the maritime sector.
First, the model is formulated in the spatial domain to convexify the vessel dynamics and take the distance-dependent speed limits and zero-emission legs into account.
Second, a convex fixed-pitch propeller model is introduced, improving the drivetrain modeling.
Third, hydrodynamic forces, acting on the hull and rudder, are modeled in the convex optimization framework and compared against captive test measurements. 
Therefore, the main contribution of this work is the convexification of a model for the simultaneous optimization of the trajectory and the energy management strategy for vessels equipped with a hybrid power source and electrical propulsion. This convexification lays the path for the employing mature, computationally robust convex optimization methods and creates a novel possibility for real-time optimization of these systems. In particular, for the application of autonomous vessels, the onboard computers need to make the path-planning decisions rapidly and reliably without a human being in the loop. In this context, convex optimization is the ideal approach.

\section{Related work} \label{sec:literature_review}
\subsection{Advanced control strategies in the maritime sector}
Due to increasingly demanding reduction goals of \gls{ghg} emissions, the interest of the maritime industry in battery systems combined with a diesel engine as a hybrid energy supply system has increased over the past years.
This technique promises high fuel savings and \gls{ghg} reductions on the one hand.
On the other hand, the costs for propulsion systems increase as the complexity rises. \cite{geert_DesignControlHybrid_2017}

Additionally, these hybrid systems are commonly run with conventional heuristic rule-based control algorithms leading to marginal fuel savings only.
Higher savings could be achieved with more advanced controls such as \gls{ecms} or \gls{mpc}. \cite{geert_DesignControlHybrid_2017}

\Gls{mpc} in the maritime context is for example studied in \cite{park_RealTimeModelPredictive_2015, makin_EnergySavingsShip_2017, hou_MitigatingPowerFluctuations_2018, hasel_ModelPredictiveManeuvering_2019, huota_HybridShipUnit_2020}.
The studies in \cite{makin_EnergySavingsShip_2017, hou_MitigatingPowerFluctuations_2018} are especially interesting in the context of this work since they include the longitudinal vessel dynamics as well as propeller models for electric propulsion.
However, they do not focus on trajectory planning but on the fluctuations due to waves on the propeller and the load sharing, respectively.

The model in \cite{makin_EnergySavingsShip_2017} targets the control of the propulsion system of a vessel with a controllable pitch propeller.
It takes the vessel design speed as an input and the proposed non-linear \gls{mpc} is designed to follow that trajectory by adjusting the control inputs of the propeller - drive frequency, propeller pitch ratio and rate of change in propeller pitch ratio - to minimize the reference speed tracking error and the energy consumed by the electric motor.
Additionally, the control inputs and the tracking errors for reference drive frequency and reference propeller pitch ratio are minimized.
Each term in the cost function includes a weighting factor.
The prediction horizon of the \gls{mpc} is limited to 90~s resolved in 0.1~s steps resulting in only 900 discretization steps in total.

The main goals of the authors in \cite{hou_MitigatingPowerFluctuations_2018} are "to minimize the power tracking error [...] and reduce [hybrid energy storage system] losses to improve energy efficiency".
Their hybrid system consists of a battery system and an ultracapacitor.
Two different \gls{mpc} variants with the vessel speed and motor shaft speed as inputs are developed and compared to achieve optimal load sharing for the aforementioned goals.
The prediction horizon for the receding \gls{mpc} is limited between 10 and 20 steps.
However, the authors of \cite{hou_MitigatingPowerFluctuations_2018} argue that the results are comparable to an offline \gls{mpc} with a prediction horizon of 100 steps and show exemplary results from a real-time platform.

None of the mentioned \gls{mpc} studies focuses on the simultaneous, fuel-minimal optimization of propeller shaft speed, load sharing and speed trajectory.
This simultaneous approach combined with a long prediction horizon could achieve even higher savings.
Additionally, the generation of a feasible speed trajectory considering vessel dynamics as well as propeller and energy supply system behavior is required for autonomous shipping.

A similar conclusion regarding the simultaneous approach is drawn in \cite{frang_DevelopmentsTrendsChallenges_2020}, where an extensive literature review covering 57 journal articles published between 2008 and 2020 dealing with the synthesis, design and operational optimization of maritime energy systems is conducted.
As this work focuses on operational optimization, the findings in \cite{frang_DevelopmentsTrendsChallenges_2020} regarding this field are discussed further.

Out of the 57 articles, eleven deal solely with the operational optimization of maritime energy systems.
Nine other articles consider the design as well, and another ten articles also take the synthesis into account.
However, only six of these 30 articles cover the dynamic operation of the vessel.
Two of the six, \cite{wang_DynamicOptimizationShip_2018} and \cite{tzort_DynamicOptimizationSynthesis_2019}, are of particular interest since they optimize the speed as a control input variable.
They both neglect vessel dynamics such as acceleration and deceleration though and consider the speed to be constant during each leg of the journey.
This is also common practice in maritime logistic problems as seen in \cite{norst_TrampShipRouting_2011, hvatt_AnalysisExactAlgorithm_2013, zis_EnvironmentalBalanceShipping_2015}.
The assumption of constant speed per journey leg might be acceptable for ocean-going vessels because they travel at a constant speed – preferably the respective design speed – for most of the voyage.

Smaller vessels like ro-pax ferries traverse in waters close to coastal areas where speed limits might apply.
Thus, they cannot travel at a single constant speed for most of the journey.
As an additional aspect, the duration of respective acceleration and deceleration phases becomes more significant compared to the overall shorter voyage duration.
Therefore, vessel dynamics must be considered to achieve minimum energy consumption.
To the authors' knowledge, no study exists dealing with the combined approach of optimizing speed profile and energy management for a short sea vessel when taking the longitudinal dynamics into account.

\subsection{Convex optimization in optimal control}
Starting about a decade ago, convex optimization has become more popular in engineering applications dealing with optimal control problems due to the development of computationally efficient and reliable implementations of interior point methods for cone programming \cite{Gon12}.
Especially for integrated design and control of hybrid electric vehicles, the number of publications is rising \cite{SHM+16}.
Other applications include the operation of battery systems for the frequency regulation market in power grids \cite{shi_OptimalBatteryControl_2018}, the energy management in a microgrid for a sustainable community \cite{cai_ConvexOptimizationbasedControl_2017} and the trajectory planning for robots \cite{versc_PracticalTimeOptimalTrajectory_2008}.
The latter, together with \cite{ebbes_TimeoptimalControlStrategies_2018}, which deals with the speed optimization for a hybrid electric race car, are of particular interest because their problems are formulated in the spatial domain.
Modeling in the spatial domain is also applied in this work due to the speed limits. 

In maritime applications, convex optimization is scarcely employed.
Some studies employ its subclass of linear programs \cite{zis_EnvironmentalBalanceShipping_2015} but, rather, more often \glspl{milp} or \glspl{nlp} are used \cite{frang_DevelopmentsTrendsChallenges_2020}.
One example of the latter is the fuel cell hybrid electric vessel in \cite{sun_OptimalEnergyManagement_2020}.
There, the energy management optimization problem is formulated as a convex \gls{minlp} to include some piecewise functions.
It is optimized in the time domain and the total voyage duration of ten hours is discretized in steps of 1~h each.

In \cite{kalik_ShipEnergyManagement_2018}, the minimization of the energy consumption of a tugboat with a hybrid system consisting of a diesel engine and a battery is presented.
The authors formulate a \gls{minlp} to optimize the load sharing between the different power sources.
According to them, the optimization problem is still efficiently solvable because only three integer variables – limiting the problem size – exist and the sub-routines solve convex problems.
Seven different profiles for vessel speed and bollard pull are evaluated through simulation runs in the time domain.
Again, no holistic approach is implemented.
Instead, only the load sharing is optimized.

In \cite{hasel_ModelPredictiveManeuvering_2019} a linearized model for \gls{mpc} of an autonomous ship is developed.
As before, the vessel studied is a tugboat with a hybrid system consisting of a diesel engine and a battery, and several operational profiles are simulated.
It is considered relevant for this literature review because the predicted power demand in the \gls{mpc} is based on a propeller model and the dynamic vessel speed.
The vessel speed is determined by following a reference trajectory while respecting the vessel dynamics.
However, it is not precisely described how the hydrodynamic resistance is obtained.

\subsection{Maneuvering models}
The predominant mathematical models of conventional surface vessel maneuvering consider the motion in a horizontal plane in the surge (forward), sway (crossbody) and yaw (heading) directions.
The equations of motion are expressed by applying Newton's second law in a vessel fixed coordinate system.
The hydrodynamic forces and moments are assumed to be functions of the velocities and accelerations of all three degrees.
The unknown force and moment functions are approximated by terms in a Taylor series expansion, which is a concept originally introduced by Abkowitz \cite{Abk64}.
The coefficients in the expansion are called hydrodynamic derivatives, which are determined from captive tests for a given hull form.

Numerical integration of the nonlinear differential equations describing vessel motions is used in the design phase to predict trajectories and assess the maneuvering performance of a vessel with given hull form.
Maneuverability is described by a number of parameters, such as overshoot angle in zig-zag maneuver, which are required to receive certain numerical values as imposed by classification societies.
The predictive accuracy of Abkowitz-type models is considered sufficient for certification \cite{Abs17}. However, the terms with second and higher order render the differential equations nonlinear and non-convex. Thus, the Abkowitz-type models are incompatible with the convex optimization framework.

The surge force is controlled via the propeller, sway force via bow thrusters and yaw moment via the rudder.
The system is considered underactuated if bow thrusters are unavailable or not used, because only two inputs are available for controlling three degrees of freedom \cite{LPN03}.
In steering and tracking control, the equations of motions are typically linearized by dropping all nonlinear terms and assuming that the surge velocity deviates only slightly from the initial velocity.
The linearized equations for steering were introduced by Nomoto in 1957 \cite{NTH+57}, and are still commonly used.

\subsection{Propeller models}
Independent from the specific design of a propeller, its performance can be described using its general open water characteristics.
These are based on the forces and momenta of the propeller when operating in open water without any disturbances and are usually expressed using the non-dimensional variables \glsentrylong{kt} \glsunset{kt} 
\begin{equation} \label{eq:K_T}
    \gls{kt} = \frac{\gls{T_p}}{\gls{rho_sw}  \gls{D_p}[^4]  n_\text{p}^2},
\end{equation}
\glsentrylong{kq} \glsunset{kq} 
\begin{equation} \label{eq:K_Q}
    \gls{kq} = \frac{\gls{Q_p}}{\gls{rho_sw}  \gls{D_p}[^5]  n_\text{p}^2},
\end{equation}
open water efficiency
\begin{equation} \label{eq:eta_o}
    \eta_o = \frac{P_\text{p,out}}{P_\text{p,in}} = \frac{\gls{T_p}  \gls{v_a}}{2\pi  \gls{Q_p}  n_\text{p}} = \frac{\gls{kt}}{\gls{kq}}  \frac{\gls{J}}{2\pi},
\end{equation}
and advance coefficient
\begin{equation} \label{eq:J}
    \gls{J} = \frac{\gls{v_a}}{n_\text{p}  \gls{D_p}}
    = \frac{(1 - \gls{f_w})  v_\mathrm{s}}{n_\text{p}  \gls{D_p}},
\end{equation}
where \gls{rho_sw} is the density of seawater, $\gls{D_p}$ the diameter of the propeller, $n_\text{p}$ the propeller shaft speed, \gls{Q_p} the propeller torque, $P_\text{p,out}$ thrust power and $P_\text{p,in}$ the shaft input power \cite{carlt_MarinePropellersPropulsion_2007}.

The advance speed $v_\text{a}$ describes the average speed of the water across the surface of the propeller.
Due to the interaction of the hull with the surrounding water, the advance speed is lower than the vessel speed.
This is mathematically expressed through the wake fraction coefficient \gls{f_w} and the relation between advance and vessel speed is given in equation \eqref{eq:J}.

The open water efficiency as well as \gls{kt} and \gls{kq} are often plotted over the advance coefficient in a so-called open water diagram.
An exemplary diagram is shown in Figure~\ref{fig:owd_example}.

\begin{figure}[h]
\centering
\includegraphics[width=\columnwidth]{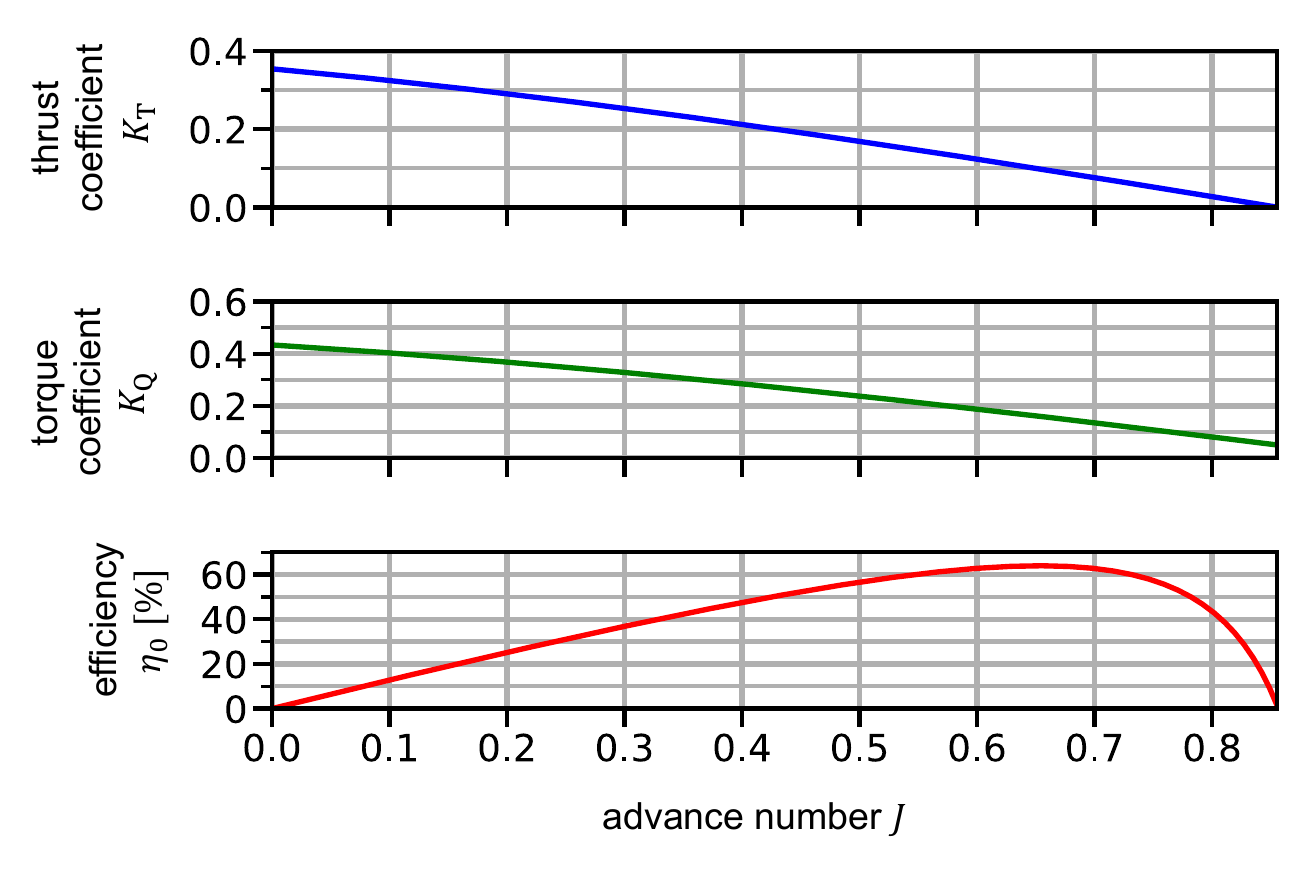}
\caption{Exemplary open water diagram.}
\label{fig:owd_example}
\end{figure}

\textbf{Wageningen B-Series.}
One of the most fundamental and commonly cited propeller models is developed and published in \cite{vanl_WageningenBScrewSeries_1969, ooste_RecentDevelopmentsMarine_1972, ooste_FurtherComputeranalyzedData_1975}.
The model is based on a regression analysis of the Wageningen B-Series.
\Gls{kt} and \gls{kq} are represented as polynomials of the advance coefficient,
\begin{equation} \label{eq:K_T_poly}
    \hat{K}_\mathrm{T} = \sum_{n=1}^{N_\mathrm{T}} C_\mathrm{T,n}  J^{S_\mathrm{T,n}}  \left( \frac{P_\mathrm{p}}{\gls{D_p}} \right)^{t_\mathrm{T,n}}  \left( \frac{A_\mathrm{E}}{A_0} \right)^{u_\mathrm{T,n}}  Z^{v_\mathrm{T,n}}
\end{equation}
and
\begin{equation} \label{eq:K_Q_poly}
    \hat{K}_\mathrm{Q} = \sum_{n=1}^{N_\mathrm{Q}} C_\mathrm{Q,n}  \gls{J}^{S_\mathrm{Q,n}}  \left( \frac{P_\mathrm{p}}{\gls{D_p}} \right)^{t_\mathrm{Q,n}}  \left( \frac{A_\mathrm{E}}{A_0} \right)^{u_\mathrm{Q,n}}  Z^{v_\mathrm{Q,n}}
\end{equation}
depending on several propeller parameters.
These are the number of blades $Z$, the extended blade area ratio ${A_\mathrm{E}}/{A_0}$ - the ratio of the expanded blade area to the disc area - and the pitch diameter ratio ${P_\mathrm{p}}/{\gls{D_p}}$.
They describe the specific design of the propeller and are explained in the related literature \cite{bertr_PracticalShipHydrodynamics_2012, carlt_MarinePropellersPropulsion_2007}.
The analysis in \cite{vanl_WageningenBScrewSeries_1969} is not only limited to the \gls{1q} open water diagram, but \gls{4q} data is also included, and the respective coefficients for thrust and torque are each represented as a Fourier series with 21 coefficients.

\textbf{Models based on lift and drag.}
Since the model in \cite{vanl_WageningenBScrewSeries_1969} is limited to the Wageningen B-Series, the authors in \cite{heale_ImprovedUnderstandingThruster_1994} introduce another \gls{4q} model based on the lift and drag forces of the propeller to achieve a more versatile model, which was used for an underwater vehicle.
They calculate lift and drag based on sinusoidal functions in order to derive thrust and torque with a rotational transformation.
The required inputs are the effective angle of attack and the total relative velocity.
These are calculated based on an additional fluid model.

The propeller model in \cite{heale_ImprovedUnderstandingThruster_1994} requires the axial flow velocity to be known and is estimated with a fluid model.
The authors in \cite{kim_AccuratePracticalThruster_2006} simplify this estimation by calculating the axial flow velocity as the linear combination of the vehicle velocity and the propeller shaft speed for their steady-state model of an underwater vehicle.

In \cite{smoge_ControlMarinePropellers_2006}, where a low-level thruster controller for a fixed pitch propeller is developed, the equations for thrust and torque are modified to include losses concerning mainly the propeller loading.
These losses are summarized in two different loss factors for thrust and torque, which is a function of thrust.
The thrust and torque equations are simply multiplied by the respective coefficient.
The same model is further investigated in the context of thruster control in \cite{soren_TorquePowerControl_2009}.
It is also employed in \cite{skuls_HybridApproachMotion_2021} in the context of motion prediction with machine learning for ship docking.

In \cite{blank_DynamicModelThrust_2000}, the linear approximation of thrust and torque coefficients and the resulting quadratic models for thrust and torque are modified to account for differences in the axial water flow into the propeller.
The authors neglect the drag in the thrust equation, but they note it could be easily included in case the linear approximation does not yield sufficient accuracy.
They add that the thrust coefficient as a function of the advance coefficient would become a second-order polynomial when including the drag.

\textbf{\Gls{4q} model for underwater vehicles.}
The authors in \cite{hausl_FourQuadrantPropellerModeling_2013} derive a \gls{4q} propeller model for the use of underwater vehicles.
They are interested in ``developing algorithms for the computation of energy-optimal trajectories for multiple vehicles acting in cooperation'' and have found the Wageningen B-Series \gls{4q} model to be insufficient due to numerical reasons.
They consider using the model from \cite{heale_ImprovedUnderstandingThruster_1994}, but discard it due to physical inconsistencies regarding the behaviour of thrust, torque and open water efficiency.
Therefore, they introduce another sinusoidal model based on lift and drag to receive a low-order approximation of the model in \cite{vanl_WageningenBScrewSeries_1969}.

\textbf{Propeller models in optimization.}
Most of the previously mentioned articles either deal with low-level thruster control or underwater vehicles and are included to give an overview over different propeller models.
Articles dealing with problems more closely related to this work include \cite{geert_PitchControlShips_2017, kalik_ShipEnergyManagement_2018, wang_DynamicOptimizationShip_2018, hou_AdaptiveModelPredictive_2018, hasel_ModelPredictiveManeuvering_2019}.

\textbf{\Gls{4q} models for controllable and fixed pitch propellers.}
In \cite{geert_PitchControlShips_2017} and \cite{kalik_ShipEnergyManagement_2018}, models based on the \gls{4q} open water diagram are employed.
The first article focuses on fuel savings for diesel-mechanical propulsion and considers controllable pitch propellers of the Wageningen C- and D-series.
The influence of waves on the propeller performance and the engine loading is considered by including the wave speed in the calculation of the advance speed.
In \cite{kalik_ShipEnergyManagement_2018}, the energy management for hybrid propulsion with a fixed-pitch propeller of the Wageningen B-series is examined.

\textbf{\Gls{1q} models for fixed pitch propellers.}
The authors in \cite{hasel_ModelPredictiveManeuvering_2019} use a \gls{1q} model based on the open water diagram of the Wageningen B-Series for maneuvering control and energy management of a hybrid vessel with electric propulsion.
Thrust and torque are expressed as quadratic functions of the propeller shaft speed and utilize \gls{kt} and \gls{kq}.
The coefficients are given as functions of propeller geometry and advance coefficient.

A similar approach is used in \cite{wang_DynamicOptimizationShip_2018} for a \gls{mpc} to optimize the energy efficiency of a vessel.
\gls{kt} and \gls{kq} are based on the \gls{1q} open water diagram of the Wageningen B-series, but neither thrust nor torque is explicitly calculated.
Instead, the coefficients are used to determine the open water efficiency.

The authors in \cite{hou_AdaptiveModelPredictive_2018} also present a \gls{1q} model based on the open water diagram of the Wageningen B-Series.
Their goal is to minimize the wear of the propulsion system and maximize the efficiency of the used hybrid energy supply system when considering propeller load fluctuations.
Therefore, they modify the presented model and suggest using a torque prediction model with online parameter estimation.
The thrust is not estimated.

\textbf{\Gls{1q} model for controllable pitch propeller.}
In \cite{makin_EnergySavingsShip_2017} a propeller is modeled in the context of non-linear \gls{mpc} for an optimal control problem to minimize the energy consumption of a bulk carrier equipped with electric propulsion.
The propeller torque is modeled as a function of shaft speed and \gls{kq}. 
\gls{kq} is a second-order polynomial of the advance coefficient - the same is true for \gls{kt}.
The coefficients for both, \gls{kt} and \gls{kq}, depend on the pitch of the propeller.

\textbf{Comparison of propeller models.}
The presented models show how wide is the range of approaches to model the complex behavior of propellers.
Whereas more detailed models rely on high order polynomials or sinusoidal functions, simplified models sufficiently accurate for the optimization of vessel operation are typically based on low order polynomials due to computational efficiency.
Throughout all kinds of studies about fixed pitch propeller models it is common to use the Wageningen B-Series.

\section{Minimum time and fuel optimal control problem}
This section first lays out the elementary minimum energy and time optimal control problem along a fixed path.
The nonconvexity of this problem is resolved via spatial domain reformulation and variable changes, which are discussed next. 
Convex reformulations of hydrodynamic forces, propeller and energy system components are discussed immediately following the derivation of the equations governing the respective subsystem. 
The final part of this section derives a condition that, when satisfied, ensures that the reformulated convex problem provides the same solution as the original nonconvex problem.

\subsection{Minimum time and energy problem along a fixed path}

\textbf{Dynamics.}
We consider a vessel with $p$ degrees of freedom represented by the configuration vector $\gls{q}(t)\in \reals^p$ and control input vector $\gls{u}(t) \in \reals^r$.
The equations of motion arising from Newton’s second law are of the second-order form in the global frame
\begin{equation} \label{eq:time_dynamics}
    \gls{R}(\gls{q}) \, \gls{u} = \gls{M} \, \Ddot{\gls{q}} + \tau(\gls{q}),
\end{equation}
where $\Ddot{\gls{q}}$ is the elementwise second time derivative of the state vector, $ \gls{R}(\gls{q}):\reals^p \rightarrow \reals^{p \times r}$ is the state dependent rotation matrix, $ \gls{M} \in \reals^{p \times p}$ is the symmetric, positive definite mass matrix and $\tau(\gls{q})\in\reals^p$ is the disturbance force.

\textbf{Path.}
Let the function $\gls{sigma} : [0,T] \rightarrow [0,1]$ denote a path coordinate, such that $\gls{sigma}(0) = 0, \gls{sigma}(T) = 1$ and speed along the path is $\dot{\gls{sigma}} > 0$, where $T$ is the terminal time \cite{versc_PracticalTimeOptimalTrajectory_2008}.
The vector-valued function $\gls{s} : [0,1] \rightarrow \reals^p$ defines the path of the vessel as a mapping from the path coordinates to the state vector of the vessel \gls{q} at every point along the path.
The vessel moves along the path when
\begin{equation}\label{eq:path_definition}
    \gls{s}(\gls{sigma}(t)) = \gls{q}(t)
\end{equation}
is fulfilled for each $t \in [0,T]$.

\textbf{Optimal control problem.}
We can represent the dynamics \eqref{eq:time_dynamics} in terms of the path coordinate $\gls{sigma}$ using the relation (see Appendix \ref{app:dynamics_spatial})
\begin{equation*}
    \ddot{\gls{q}}(t) = \gls{s}''(\gls{sigma}(t))\dot{\gls{sigma}}(t)^2 + \gls{s}'(\gls{sigma}(t))\ddot{\gls{sigma}}(t).
\end{equation*}
By introducing the new function
\begin{equation} \label{eq:b}
    \gls{b} = \dot{\gls{sigma}}^2
\end{equation}
the minimum time and energy motion planning problem (P-CVX) along a fixed path is formulated as a convex optimization problem in the spatial domain:
\begin{align*}
     \text{(P-CVX):} & \\
     \underset{b(\gls{sigma}),\gls{u}(\gls{sigma})}{\text{min.}} \quad & \int_{0}^{1} \left( \frac{1}{\sqrt{b(\gls{sigma})}} + \phi(\gls{u}(\gls{sigma})) \right) \,d\gls{sigma} \\
    \text{s.t. } \quad & \gls{R} \left(\gls{s}(\gls{sigma})\right) \, \gls{u}(\gls{sigma}) = \gls{M}\, \gls{s}'(\gls{sigma}) \, \frac{\gls{b}'(\gls{sigma})}{2}, \\ 
     & + \gls{M} \, \gls{s}''(\gls{sigma}) \, b(\gls{sigma}) + \tau\left(\gls{s}(\gls{sigma})\right), \quad \gls{sigma} \in [0,1], \\
     & \left( \gls{s}'(\gls{sigma})^2 \gls{b}(\gls{sigma}), \gls{u}(\gls{sigma}) \right) \in \Pi\left(\gls{s}(\gls{sigma})\right), \quad \gls{sigma} \in [0,1],
\end{align*}
where $\phi : \reals^r \rightarrow \reals$ is a convex fuel use function and $\Pi(\gls{q}) \subseteq \reals^{p \times r}$ is a set valued mapping of convex sets.
The integral of the first term in the objective is $T$, the terminal time:
\begin{equation*}
    \int_{0}^{1}\frac{d\gls{sigma}}{\sqrt{\gls{b}(\gls{sigma})}}  =
    \int_{0}^{1}\frac{d\gls{sigma}}{\dot{\gls{sigma}}}  = \int_{\gls{sigma}(0)}^{\gls{sigma}(T)}\frac{d\gls{sigma}}{\dot{\gls{sigma}}} = \int_{0}^{T}1dt = T.
\end{equation*}
See \cite{versc_PracticalTimeOptimalTrajectory_2008, lipp_MinimumtimeSpeedOptimisation_2014} for further discussion of the convexity of problem P-CVX, and its relation to the time domain formulation.

\textbf{Relaxation and change of variables.}
In the following sections, vessel subsystem models that are compatible with the problem formulation P-CVX are discussed.
To support this discussion, auxiliary variables and constraints that will be applied throughout the following sections are introduced.
The explicit dependence of functions on \gls{sigma} will be omitted in most cases.

\subsection{Vessel dynamics and hydrodynamic forces} \label{sec:vessel_dynamics}
The vessel is regarded as a rigid body, and the maneuvering model restricts the motion to three degrees of freedom: surge, sway and yaw.
The heel is discarded, meaning the vessel always maintains an upright orientation.
Only still water condition is considered.
This in-plane motion approximation is commonly used for surface vessels \cite{Mat13}.

\subsubsection{Three-degree-of-freedom vessel model}
The navigational position of the vessel is given in an $xy$-plane fixed to the earth.
The origin (0,0) is located at the initial location of the vessel’s center of gravity.
Let $\gls{q} = (x,y,\theta)$ denote the vessel configuration state vector, where elements $(x,y)$ are the in-plane position and $\theta$ is the orientation with respect to $x$-axis.

\begin{figure}[h]
    \centering
    \includegraphics[width=\columnwidth]{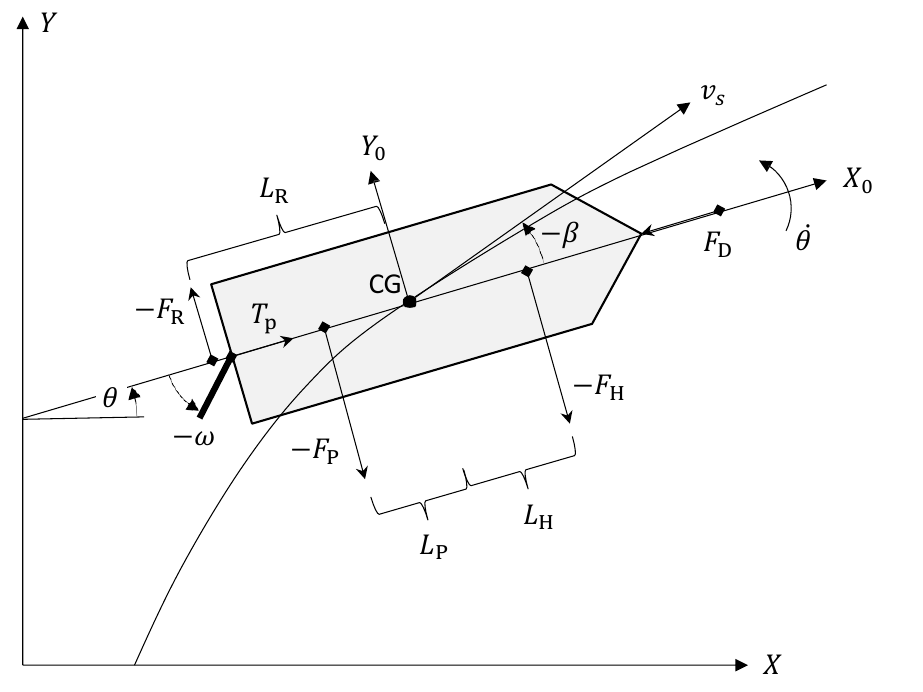}
    \caption{Forces acting on the hull.}
    \label{fig:three-dof}
\end{figure}

The hydrodynamic forces consist mainly of resistance from moving the hull entrained in fluid, and forces generated by the propellers and rudders.
Let $\vec{u}=(\gls{T_p}, \gls{F_D}, \gls{F_H}, \gls{F_P}, \gls{F_R})^\top$ denote the force input vector, where \gls{T_p} is propeller thrust, \gls{F_D} is the total drag, \gls{F_H} is the hydrodynamic hull force from drift, \gls{F_P} is the hydrodynamic damping force and \gls{F_R} is rudder steering force.
The force inputs are represented in the frame of the vessel (Figure \ref{fig:three-dof}).
The rotation matrix maps the force inputs in the vessel frame into the global frame:
\begin{equation*}
    \gls{R} =
    \begin{bmatrix}
        \cos(\theta) & -\cos(\theta) & -\sin(\theta) & -\sin(\theta) & \sin(\theta)\\
        \sin(\theta) & -\sin(\theta) & \cos(\theta) & \cos(\theta) & -\cos(\theta)\\
        0 & 0 & \gls{L_H} & -\gls{L_P} & \gls{L_R}
    \end{bmatrix},
\end{equation*}
where $\gls{L_H}, \gls{L_P}$ and $\gls{L_R}$ denote the lever arms of drift, damping and rudder forces, respectively.
The mass and inertia matrix is written as
\begin{equation*}
    \gls{M} = \mathop{\mathrm{diag}}(\gls{m_vessel}  (1+k_1),\gls{m_vessel}  (1+k_1),\gls{I}+k_2  \gls{I_w}),
\end{equation*}
where \gls{m_vessel} is vessel mass, \gls{I} is the moment of inertia of the vessel, \gls{I_w} is the moment of inertia of displaced water and $k_1$ and $k_2$ are dimensionless coefficients of the added inertia \cite{DS46}.
The added inertias account for the fact that accelerating the vessel body at a given rate in fluid requires larger force than in vacuum \cite{Mat13}.

\subsubsection{Drift}
The nominal orientation of the vessel is assumed to be known, dependent on the position $(x,y)$, and in the direction of traversal.
We allow the true orientation to vary by the angle $\beta$ from the nominal orientation.
Any deviation from the nominal orientation implies that the vessel has nonzero speed in the sway direction, which is called drifting. 

Drift generates a lift force that acts perpendicular to the hull.
The force acts through a point located at distance \gls{L_H} from the center of gravity towards the bow.
The yawing moment at the center is the moment arm \gls{L_H} multiplied by the force \gls{F_H}.

In the presence of drift, the inflow to the hull resembles the flow over a low aspect ratio airfoil at an angle of attack $\beta$, which is the drift angle \cite{Mat13}.
The lift produced by the oblique flow is
\begin{equation} \label{eq:drift_lift}
    \gls{F_H} = \frac{1}{2}  \gls{rho_sw}  S   C_{\mathrm{L}}  \gls{v_s}^2,
\end{equation}
where $S$ is the area of the lifting surface and \gls{v_s} is the vessel speed in surge direction: 
\begin{equation*}
\begin{split}
    \gls{v_s} &= \lVert (\dot{x},\dot{y})  \rVert_2\\
        &= \sqrt{\dot{x}^2 + \dot{y}^2}\\
        &= \sqrt{\dot{\glsuseri{q}}_1^2 + \dot{\glsuseri{q}}_2^2}.
\end{split}
\end{equation*}

\textbf{Maximum drift.}
The coefficient $C_{\mathrm{L}}$ can be approximated for small $\beta$ as
\begin{equation*}
    C_{\mathrm{L}} = a_{\mathrm{L,0}} + a_{\mathrm{L,1}}  \beta,
\end{equation*}
where $a_{\mathrm{L,0}}, a_{\mathrm{L,1}} \in \reals_{++}$ are constants.
Then, the maximum lift that can be generated by drifting,
\begin{equation} \label{eq:max_drift_force}
    |\gls{F_H}| \leq \frac{\gls{rho_sw}}{2}  S  C_{\mathrm{L,max}}  \gls{v_s}^2\,,
\end{equation}
is limited by the largest permitted drift angle $\beta_\text{max}$ or the largest permitted lift coefficient $C_{\mathrm{L,max}} = a_{\mathrm{L,0}} + a_{\mathrm{L,1}}  \beta_\text{max}$, respectively.

\textbf{Convexification of drift.}
The inequality \eqref{eq:max_drift_force} defines a nonconvex set because the function on the right-hand side is convex instead of affine or strictly concave \cite{boyd_ConvexOptimization_2004}.
Thus, with \eqref{eq:q_dot_aux} (see Appendix \ref{app:dynamics_spatial}) and the auxiliary variable \gls{b} the vessel speed in surge direction is rewritten as
\begin{equation} \label{eq:v2spatial}
\begin{split} 
    \gls{v_s} &= \sqrt{(\gls{s}'_1  \dot{\sigma})^2 + (\gls{s}'_2  \dot{\sigma})^2}\\
    &= \sqrt{(\gls{s}_1^{\prime 2} + \gls{s}_2^{\prime 2})}  \dot{\gls{sigma}}\\
    &= \sqrt{(\gls{s}_1^{\prime 2} + \gls{s}_2^{\prime 2})}  \sqrt{\gls{b}}.
\end{split}
\end{equation}
The first term is hereafter abbreviated as
\begin{equation*} 
    \gls{s}'_{12} = (\gls{s}_1^{\prime 2} + \gls{s}_2^{\prime 2})
\end{equation*}
and can be directly obtained from the definition of the path, and thus calculated prior to solving the optimization problem.

By inserting \eqref{eq:v2spatial} into \eqref{eq:max_drift_force} the right-hand side of \eqref{eq:max_drift_force} becomes affine in \gls{b} and can be used to limit the drift in the convex optimization problem with the constraint
\begin{equation} \label{eq:constraint_drift_force}
    |\gls{F_H}| - \frac{\gls{rho_sw}}{2}  S  C_{\mathrm{L,max}}  \gls{s}'_{12}  \gls{b} \leq 0,
\end{equation}
which is given in the standard form of inequality constraints in optimization problems.

\subsubsection{Drag}
Induced drag from drift is incorporated to the total resistance via the total drag coefficient $C_{\mathrm{D}}$.
The drag is given by
\begin{equation} \label{eq:F_D}
    \gls{F_D} = \frac{1}{2}  \gls{rho_sw}  C_{\mathrm{D}}  A_{\mathrm{s}}  \gls{v_s}^2,
\end{equation}
where $A_{\mathrm{s}}$ is the combined wetted surface area of the hull and the rudder. The drag coefficient is
\begin{equation} \label{eq:drag_coef}
    C_{\mathrm{D}} = \frac{C_{\mathrm{L}}^2}{\pi  \Omega} + C_{\mathrm{F}} + C_{\mathrm{R}},
\end{equation}
where $\Omega$ is the aspect ratio of the lifting surface and $C_{\mathrm{F}}$ is the frictional resistance coefficient, which can be evaluated according to the ITTC-57 empirical formula
\begin{equation*}
    C_{\mathrm{F}} =  \frac{0.075}{(\log_{10}(\mathrm{Rn})-2)^2}\,, \quad \mathrm{Rn}=\frac{v_s  L}{\nu},
\end{equation*}
where $L$ is vessel length and $\nu$ is kinematic viscosity of water \cite{Itt11}.
In \eqref{eq:drag_coef}, the residual coefficient $C_{\mathrm{R}}$ accounts for wave-making and viscous pressure resistance.
Estimates based on the Froude number $F_n=v_s/\sqrt{g  L}$ are given in various sources, while a more accurate model is obtained by fitting a convex function to data generated by CFD simulation. 

By expressing the coefficient of lift $C_{\mathrm{L}}$ as a function of the drift force from \eqref{eq:drift_lift} and inserting the result into \eqref{eq:F_D}, the drag
\begin{equation} \label{eq:drag}
    \gls{F_D} = \frac{\gls{rho_sw}}{2}  (C_{\mathrm{F}} + C_{\mathrm{R}})  A_{\mathrm{s}}  \gls{v_s}^2 + \frac{2  \gls{F_H}^2  A_{\mathrm{s}}}{\gls{rho_sw}  \pi  \Omega  S^2  \gls{v_s}^2}
\end{equation}
is obtained.
The added drag due to rudder deflection is discussed in Section \ref{sec:rudder_lift}.

\textbf{Convexification of the drag.}
The right-hand side of this formulation is not affine and cannot be included in the convex optimization model.
However, by inserting \eqref{eq:v2spatial}, the right-hand side becomes convex in \gls{b} and \gls{F_H} and the relaxed form of \eqref{eq:drag},
\begin{equation} \label{eq:drag_convex}
    \frac{\gls{rho_sw}}{2}  (C_{\mathrm{F}} + C_{\mathrm{R}})  A_{\mathrm{s}}  \gls{s}'_{12}  \gls{b} +
        \frac{2  \gls{F_H}^2  A_{\mathrm{s}}}{\gls{rho_sw}  \pi  \Omega  S^2  \gls{s}'_{12}  \gls{b}} - \gls{F_D} \leq 0
\end{equation}
becomes a lower bound for the drag.

\subsubsection{Yaw damping}
The rotation of the vessel gives rise to a hydrodynamic force and an associated moment which depend on the rate of turn and the surge speed.
The hydrodynamic force dampens turning because the moment acts opposite to the drift moment.
Slender-body theory has been observed to provide a good approximation of force and moment involved in turning \cite{New08, TH03}.
Slender-body theory models the hull as a stack of thin sections that are approximated by simple shapes whose added mass is given by an analytic expression.

The net force prediction from slender-body theory is calculated by integrating strip force at coordinate $x$ over the entire hull \cite{TH03}.
For a vessel with a pointed bow, the added mass at the bow is zero, resulting in net force
\begin{equation} \label{eq:damp_force}
    \gls{F_P} = x_\mathrm{T} m_\mathrm{a}(x_\mathrm{T})  \gls{v_s}  \dot{\theta}. 
\end{equation} 
Here, $x_\mathrm{T}$ is the coordinate for stern.
Assuming draft $T$ is constant along the hull length, and added mass approximated by its value for an ellipse, the strip added mass $m_\mathrm{a}(x)$, can be estimated as
\begin{equation*}
    m_\mathrm{a}(x)=\frac{\pi}{2}  \gls{rho_sw}  T^2,
\end{equation*}
where the factor $1/2$ is applied to include only the lower half of the ellipse \cite{New08}. 

\textbf{Convexification of damping.}
The right-hand side of \eqref{eq:damp_force} relate to the product $\gls{v_s}  \dot{\theta}$, which can be rewritten as
\begin{equation} \label{eq:theta_dot_spatial}
    \gls{v_s}  \dot{\theta} = \glsuseriv{s}_3  \sqrt{\gls{b}}  \sqrt{\gls{s}'_{12} \gls{b}}
    = \glsuseriv{s}_3  \sqrt{\gls{s}'_{12}}  \gls{b},
\end{equation}
which is linear in $b$.
Thus, \eqref{eq:damp_force} obtains the form of a linear equation.

\subsubsection{Rudder lift} \label{sec:rudder_lift}
A rudder set at a nonzero angle $\omega$ develops a lift force that acts sideways (sway direction) and creates a yaw moment at the center of gravity that turns the vessel.
An angled rudder also creates a drag force that acts in the surge direction.

\textbf{Downstream velocity.}
The velocity of the water flowing to the rudder is substantially higher than the speed of the vessel because the rudder is located downstream from the propeller.
The downstream velocity can be evaluated by means of the thrust loading coefficient,
\begin{equation*}
    C_\text{th} = \frac{\gls{T_p}}{\frac{\gls{rho_sw}}{2}  \gls{v_a}[^2]  \frac{\pi}{4}  \gls{D_p}[^2]}
\end{equation*}
describing the loading degree of the propeller.
The flow velocity downstream, 
\begin{equation*}
    v_\text{ds} = \gls{v_a}  \sqrt{1+C_\text{th}} = \sqrt{\gls{v_a}[^2] + \frac{\gls{T_p}}{\frac{\gls{rho_sw}}{2}  \frac{\pi}{4}  \gls{D_p}[^2]}}\,
\end{equation*}
is calculated according to potential flow theory \cite{Mat13}.
Due to the close proximity of the propeller and rudder and turbulent mixing of the water jet and surrounding flow, the velocity of the water entering the rudder is slightly lower than $v_\text{ds}$.
The corrected velocity is expressed as as $v_\text{R} = k_\text{tm}  v_\text{ds}$, where $k_\text{tm} \approx 0.9$.

\textbf{Boundaries of rudder lift force.}
The lift force generated by the rudder is given as
\begin{equation*}
    L_\text{rudder}=\frac{\gls{rho_sw}}{2}  C_\text{K}  A_\text{R}  v_\text{R}^2,
\end{equation*}
where $A_\text{R}$ is the projected area of the side view of the rudder \cite{Mat13}.
The lift coefficient is 
\begin{equation*}
    C_\text{K}=\frac{2 \pi  \Lambda  \left( \Lambda + 1 \right)}{\left( \Lambda + 2 \right)^2}  \sin(\omega),
\end{equation*}
where $\Lambda=b_\text{R}^2/A_\text{R}$ denotes aspect ratio with the average height $b_\text{R}$ and $\omega$ is the angle of attack.
The generated rudder force
\begin{multline} \label{eq:boundaries_F_R}
    |\gls{F_R}| \leq \frac{\gls{rho_sw}}{2}  C_\text{K} \left( \omega_\text{max} \right)  A_\text{R}  v_\text{R}^2 \\
    = \frac{\gls{rho_sw}}{2}  C_\text{K} \left( \omega_\text{max} \right)  A_\text{R}  k_\text{tm}^2  \left( \gls{v_a}[^2]+\frac{\gls{T_p}}{\frac{\gls{rho_sw}}{2}  \frac{\pi}{4}  \gls{D_p}[^2]} \right)
\end{multline}
is limited by the total angle of attack, set as $\omega_\text{max}$.

\textbf{Convexification of boundaries of rudder lift force.}
By inserting \eqref{eq:b} into \eqref{eq:boundaries_F_R}, an affine formulation of the boundaries of the rudder lift force is obtained 
\begin{multline} \label{eq:boundaries_F_R_convex} 
    |\gls{F_R}| - \frac{\gls{rho_sw}}{2}  C_\text{K}\left(\omega_\text{max} \right)  A_\text{R}  k_\text{tm}^2 \\
    -\left( (1 - \gls{f_w})^2  \gls{s}'_{12}  \gls{b} + \frac{\gls{T_p}}{\frac{\gls{rho_sw}}{2}  \frac{\pi}{4}  \gls{D_p}[^2]} \right) \leq 0.
\end{multline}
Here, \gls{T_p} is given in \eqref{eq:T_p_approx}.

\textbf{Rudder drag.}
We only model the drag due to deflection, because the rudder area and thus the viscous drag is already included in the hull surface area. The drag force \cite{Mat13} 
\begin{equation*}
    \gls{D_R} = \frac{\gls{rho_sw}}{2}  \frac{1.1  C_\text{K}^2}{\pi  \Lambda}  A_\text{R}  v_\text{R}^2
\end{equation*}
yields a convex inequality when the term $C_\text{K}$ is expressed in terms of \gls{F_R}
\begin{equation} \label{eq:rudder_drag}
    \gls{D_R} \geq \frac{2.2  \gls{F_R}^2}{A_\text{R}  \pi  \Lambda  \gls{rho_sw}^2  k_\text{tm}^2 \\
    \left( (1 - \gls{f_w})^2  \gls{s}'_{12}  \gls{b} + \frac{\gls{T_p}}{\frac{\gls{rho_sw}}{2}  \frac{\pi}{4}  \gls{D_p}[^2]} \right)}.
\end{equation}

\subsection{Propeller model} \label{sec:propeller_model}
Fixed-pitch propellers are considered for propulsive thrust generation.
The respective model needs to replicate thrust, torque and efficiency with sufficient accuracy, but it is not required to capture all hydrodynamic effects.
The latter can be achieved with more computationally expensive models during the propeller design process, which is not part of this work.

\textbf{Thrust and torque.}
The first step is to rearrange equations~\eqref{eq:K_T} and \eqref{eq:K_Q} as
\begin{equation} \label{eq:T_p}
    \gls{T_p} = \gls{rho_sw}  \gls{D_p}[^4]  \gls{kt} \left( \gls{J} \right)  n_\text{p}^2
\end{equation}
and
\begin{equation} \label{eq:Q_p}
    \gls{Q_p} = \gls{rho_sw}  \gls{D_p}[^5]  \gls{kq} \left( \gls{J} \right)  n_\text{p}^2\,.
\end{equation}
Neither of the equations forms a convex set.
The expressions for \Gls{kt} and \gls{kq} are polynomials - see \eqref{eq:K_T_poly} and \eqref{eq:K_Q_poly} - and hence not easily convexified.

\textbf{Convexification of thrust and torque coefficients.}
The curves in the open water diagram in \figref{fig:owd_example} suggest the convex, second-order polynomial approximations
\begin{equation} \label{eq:K_T_approx}
    \gls{kt}\left(\gls{J}\right) = - a_\text{T,2}  \gls{J}^2 - a_\text{T,1}  \gls{J} + a_\text{T,0}
\end{equation}
and
\begin{equation} \label{eq:K_Q_approx}
    \gls{kq}\left(\gls{J}\right) = - a_\text{Q,2}  \gls{J}^2 - a_\text{Q,1}  \gls{J} + a_\text{Q,0}
\end{equation}
depending on the advance coefficient.
This assumption is supported by the findings in \cite{kim_AccuratePracticalThruster_2006, blank_DynamicModelThrust_2000}.

In equation \eqref{eq:K_T_approx}, $a_\text{T,2} \in \reals_+$, $a_\text{T,1} \in \reals_+$ and $a_\text{T,0} \in \reals_{++}$ are the coefficients for the approximation function of \gls{kt}.
They are determined by fitting the second-order polynomial to the original open water diagram.
The same is true for the coefficients of the torque coefficient function, $a_\text{Q,2} \in \reals_+$, $a_\text{Q,1} \in \reals_+$ and $a_\text{Q,0} \in \reals_{++}$, in ~\eqref{eq:K_Q_approx}.

\textbf{Convexification of the thrust.}
By inserting the thrust coefficient approximation from equation \eqref{eq:K_T_approx} in equation \eqref{eq:T_p} and by replacing the advance coefficient with equation \eqref{eq:J}, the thrust in equation \eqref{eq:K_T_approx} is rewritten as
\begin{equation} \label{eq:T_p_prelim}
    \gls{T_p} = \gls{rho_sw}  \gls{D_p}[^4] 
        \left( - \frac{a_\text{T,2}}{\gls{D_p}[^2]}  v_\text{a}^2 -
        \frac{a_\text{T,1}}{\gls{D_p}}  v_\text{a}  n_\text{p} + a_\text{T,0}  n_\text{p}^2 \right).
\end{equation}
As this equation must be represented as an equality constraint, the right-hand side needs to be formulated as an affine expression.

Thus, a new auxiliary variable
\begin{equation} \label{eq:n_tilde}
    \gls{n_p} = n_\text{p}^2
\end{equation}
substituting the squared shaft speed is introduced.

Since equation \eqref{eq:T_p_prelim} is still not convex due to the bilinear term of advance speed and shaft speed, another auxiliary variable
\begin{equation} \label{eq:z}
    \gls{z} = \gls{v_s}  n_\text{p} \leq \sqrt{\gls{s}'_{12}  \gls{b}  \gls{n_p}} \Leftrightarrow \gls{z} - \sqrt{\gls{s}'_{12}  \gls{b}  \gls{n_p}} \leq 0
\end{equation}
is proposed.
It expresses the aforementioned bilinear term of advance and shaft speed as the square root of the auxiliary variables from equations \eqref{eq:b} and \eqref{eq:n_tilde}.

With the auxiliary variables in \eqref{eq:n_tilde}, \eqref{eq:z}, and \eqref{eq:b}, the thrust is rewritten as
\begin{equation} \label{eq:T_p_approx}
    \gls{T_p} = - \glsuseriii{a_T}  \gls{s}'_{12}  \gls{b}
                   - \glsuserii{a_T}  z
                   + \glsuseri{a_T}  \gls{n_p}.
\end{equation}
For the sake of readability, the auxiliary coefficients $\glsuseriii{a_T}$, $\glsuserii{a_T}$ and $\glsuseri{a_T}$ are introduced.
They are given in \appref{app:T_p_coeffs}.
The formulation \eqref{eq:T_p_approx} is affine and can be included as an equality constraint in the convex optimization model.

\textbf{Convexification of the torque.}
An analogous approach leads to the affine formulation of the propeller torque
\begin{equation} \label{eq:Q_p_approx}
    \gls{Q_p} = - \glsuseriii{a_Q}  \gls{s}'_{12}  \gls{b}
                   - \glsuserii{a_Q}  z
                   + \glsuseri{a_Q}  \gls{n_p}.
\end{equation}
Once again, auxiliary coefficients $\glsuseriii{a_Q}$, $\glsuserii{a_Q}$ and $\glsuseri{a_Q}$, given in \appref{app:Q_p_coeffs}, are used.

\textbf{Convexification of the propulsive power.}
To obtain the last missing value for the propeller model, namely the change in propeller energy input, the equation for the propulsive power is formulated as
\begin{equation} \label{eq:P_p_spat}
    P_{\mathrm{p,in}} = 2\pi  \gls{Q_p}  n_\text{p}.
\end{equation}
Since the power is the derivative of the energy with regard to time, it cannot be included in the convex optimization model.
Therefore, a fictive propeller input force - the index ``in'' is omitted for the sake of readability -
\begin{equation} \label{eq:intro_FdEp}
    \gls{F_dE_p} = \frac{dE_\text{p}}{d\gls{sigma}} = \frac{dE_\text{p}}{dt}  \frac{dt}{d\gls{sigma}} = P_{\mathrm{p,in}}  \frac{dt}{d\gls{sigma}} = \frac{P_{\mathrm{p,in}}}{\sqrt{\gls{b}}}
\end{equation}
is introduced.

By replacing the propulsive power with \eqref{eq:intro_FdEp}, \eqref{eq:P_p_spat} is rewritten and the change in propeller energy input
\begin{equation} \label{eq:dE_p_prelim}
    \gls{F_dE_p} = \frac{2\pi  \gls{Q_p}  n_\text{p}}{\sqrt{\gls{b}}}
\end{equation}
is derived.

The resulting function cannot be included in a convex optimization model.
Therefore, the propulsive torque is replaced with equation \eqref{eq:Q_p_approx}.
Additionally, equation \eqref{eq:dE_p_prelim} is relaxed with the auxiliary variables from equations \eqref{eq:b}, \eqref{eq:n_tilde} and \eqref{eq:z} and the right-hand side of
\begin{equation} \label{eq:dE_p}
    \gls{F_dE_p} \geq - \glsuseriii{k_dEp}  \gls{z} - \glsuserii{k_dEp}  \gls{n_p} + \glsuseri{k_dEp}  \frac{\gls{n_p}^2}{z}
\end{equation}
is now convex.
Again, for the sake of readability the auxiliary coefficients $\glsuseriii{k_dEp} \in \reals_+$, $\glsuserii{k_dEp} \in \reals_+$ and $\glsuseri{k_dEp} \in \reals_{++}$ are introduced.
They are given in \appref{app:coeff_dE_p}.

\textbf{Convexification of propulsive power for reduced model.}
The auxiliary variable \gls{z} can be eliminated from the propeller model by assuming that $a_\text{T,1}=0=a_\text{Q,1}=0$, at the cost of larger fitting error.
In this case, the thrust, torque and energy input expressions retain their convexity.
Appendix \ref{app:power_convexity} provides a proof of the convexity of the energy input expression
\begin{equation}
    \gls{F_dE_p} \geq 2\pi  \left( \frac{a_\text{Q,0}  \gls{n_p}^2  \gls{D_p}[^5]}{\sqrt{\gls{n_p}  \gls{s}'_{12}  \gls{b}}} - a_\text{Q,2}  \sqrt{\gls{n_p}  \gls{s}'_{12}  \gls{b}}  \gls{D_p}[^3] \right)
\end{equation}
for the reduced model.

\subsection{Subsystem models} \label{sec:subsystem_models}
\subsubsection{Drivetrain model} \label{sec:drvtr_model}
The energy system under consideration is shown in \figref{fig:es_layout}.
It consists of a general energy converter, e.g., a fuel cell or a generating set, a battery system, power electronic components and two electric motors with one gearbox each.
Depending on the actual type of the general energy converter, the respective DC/DC converter might need to be replaced by a three-phase rectifier circuit.
This would be the case for a generating set or a gas turbine.

\begin{figure}[h]
\centering
\includegraphics[width=\columnwidth]{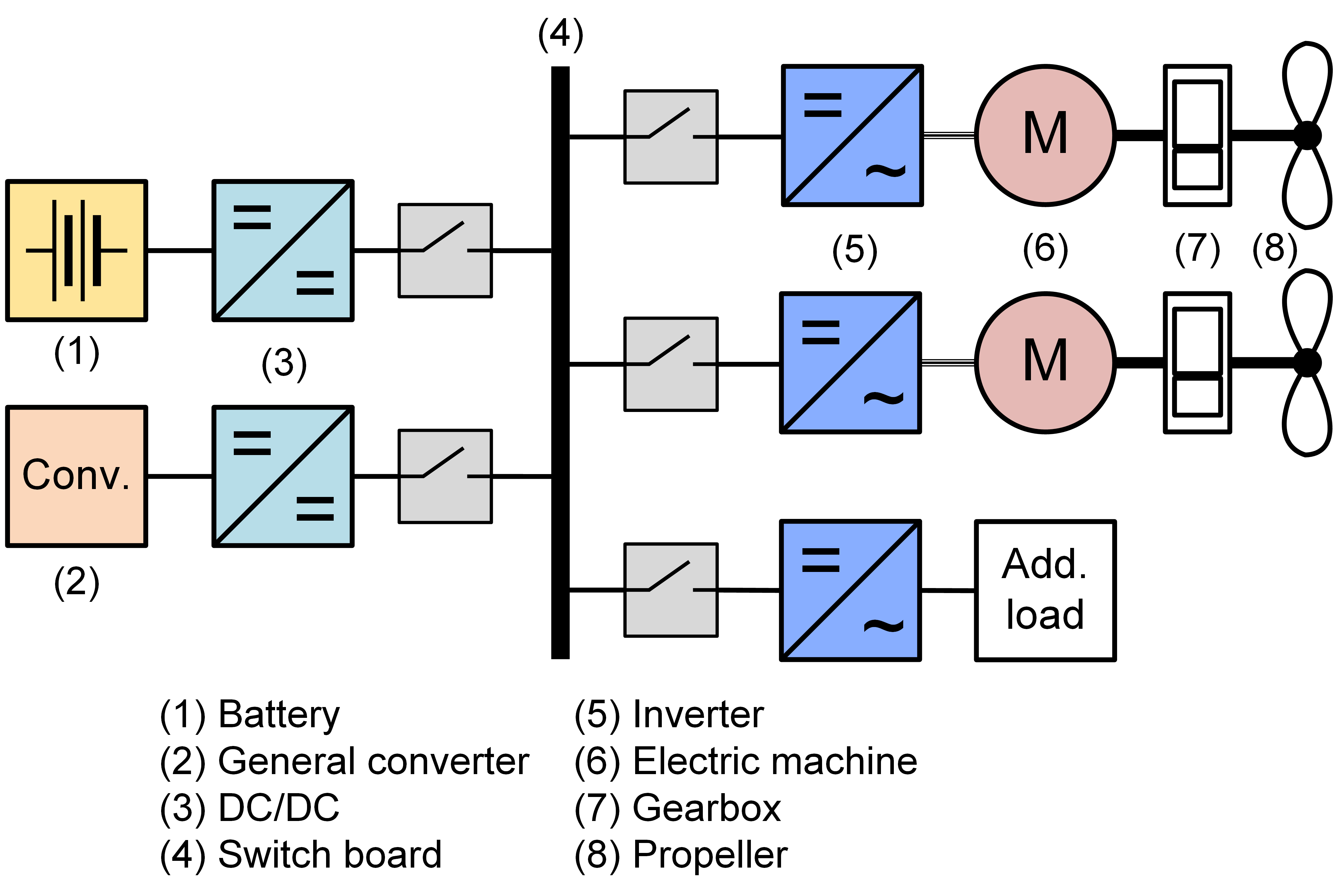}
\caption{Energy system layout.}
\label{fig:es_layout}
\end{figure}

\textbf{Electric machine input power.}
As seen in \figref{fig:es_layout}, the propeller is drive by the electric machine via a gearbox.
Both components, the gearbox and the electric machine, are lossy.
Based on the assumption of constant efficiencies, the electrical input power per electric machine is
\begin{equation*} 
    P_{\mathrm{EM,in}} = \frac{P_{\mathrm{EM,out}}}{\eta_{\mathrm{EM}}} = \frac{P_{\mathrm{p,in}}}{\eta_{\mathrm{EM}}  \eta_{\mathrm{g}}},
\end{equation*}
where $\eta_{\mathrm{EM}}$ and $\eta_{\mathrm{g}}$ are the efficiencies of electric machine and gearbox respectively, which are summarized in $\tilde{\eta}_{\mathrm{EM}} = \eta_{\mathrm{EM}}  \eta_{\mathrm{g}}$.
If the electric machine and its inverter are treated as one unit, the efficiency of the inverter $\eta_\text{inv}$ can be included in this term $\tilde{\eta}_{\mathrm{EM}} = \eta_{\mathrm{EM}}  \eta_{\mathrm{g}}  \eta_\text{inv}$.

To be compatible with equation \eqref{eq:dE_p} the power needs to be transformed to a purely mathematical force representing the change in energy over distance.
This leads to the electric machine input force
\begin{equation} \label{eq:dE_EM}
    \gls{F_em} = \frac{\gls{F_dE_p}}{\tilde{\eta}_{\mathrm{EM}}}.
\end{equation}
The index $in$ is once again omitted for the sake of readability.

\textbf{Electric machine speed limit.}
The electric machine model incorporates torque, power and speed limitations.
The shaft speed of the electric machine
\begin{equation} \label{eq:n_EM}
    n_\text{EM} = i_\text{g}  n_\text{p}    
\end{equation}
and its torque
\begin{equation*}
    Q_\text{EM} = \frac{\gls{Q_p}}{\eta_{\mathrm{g}}  i_\text{g}}
\end{equation*}
are derived using the transmission ratio of the gearbox $i_\text{g}$.

As the propeller shaft speed is not included in the model as a variable, the auxiliary variable from equation \eqref{eq:n_tilde} is used in equation \eqref{eq:n_EM} to formulate the upper limit for the shaft speed
\begin{equation} \label{eq:n_EM_lim}
    \gls{n_p} - \left( \frac{n_\text{EM,max}}{i_\text{g}} \right)^2 \leq 0,
\end{equation}
where $n_\text{EM,max}$ is the maximum electric machine shaft speed.

\textbf{Electric machine torque limit.}
With the maximum electric machine torque $Q_\text{EM,max}$ the torque limitation
\begin{equation} \label{eq:Q_EM_lim}
    \gls{Q_p} - Q_\text{EM,max}  \eta_{\mathrm{g}}  i_\text{g} \leq 0
\end{equation}
is given.
This is only accurate for speeds below the nominal speed $n_{\mathrm{EM,n}}$ because for higher speeds in the field-weakening region, the maximum torque decreases proportionally to the inverse shaft speed.
This behavior is depicted in \figref{fig:EM_char}, and it is common to both three-phase induction machines and synchronous permanent magnet machines, if they are controlled appropriately.

\begin{figure}
	\centering	
    \includegraphics[width=\columnwidth]{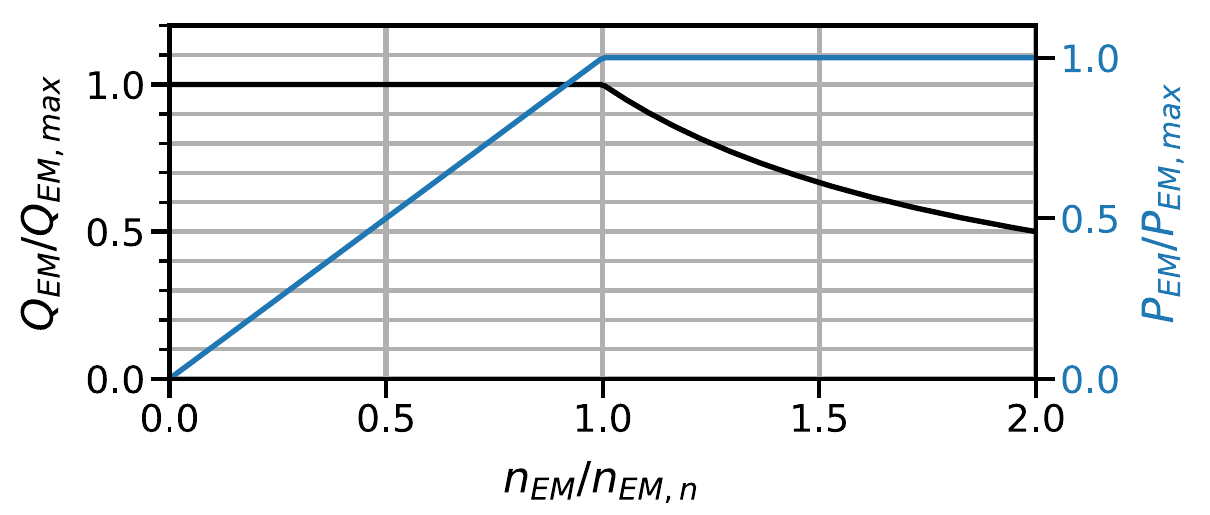}
	  \caption{Exemplary electric machine operating characteristics.}
	  \label{fig:EM_char}
\end{figure}

\textbf{Electric machine power limit.}
\figref{fig:EM_char} and the mechanical power equation
\begin{equation*}
    P_\text{EM,out} = 2\pi  Q_\text{EM}  n_\text{EM}
\end{equation*}
show that limiting the electric machine output power to a constant value in the field weakening region ensures a proper torque limitation for speeds above nominal speed.
With the inclusion of a constant torque limit and a constant electric machine output power limitation, the electric machine limits can be modeled over the whole relevant speed range.
Recalling the relationship between fictive force, power and inverse speed squared, the limitation for the maximum electric machine power $P_\text{EM,max}$ is approximated with a first-order Taylor polynomial
\begin{equation}
\begin{split} \label{eq:F_EM_max_approx}
    F_\text{EM,max} &= \frac{P_\text{EM,max}}{\sqrt{\gls{b}_\text{r}}} - \frac{P_\text{EM,max}}{2  \gls{b}_\text{r}^{1.5}}  \left( \gls{b} - \gls{b}_\text{r} \right)\\
    &= \frac{P_\text{EM,max}}{2  \sqrt{\gls{b}_\text{r}}}  \left( 3 - \frac{\gls{b}}{\gls{b}_\text{r}}\right)\\
    &= \frac{P_\text{EM,max}}{2}  \frac{\sqrt{\gls{s}'_{12}}}{\gls{v_s}_\text{,r}}  \left( 3 - \frac{\gls{s}'_{12}  \gls{b}}{\gls{v_s}[^2]_\text{,r}}\right)
\end{split}
\end{equation}
around the reference speed squared $\gls{b}_\text{r} = \gls{v_s}[^2]_\text{,r} / \gls{s}'_{12}$.
With this, the upper boundary is formulated as
\begin{equation} \label{eq:P_EM_max}
    \gls{F_em} - F_\text{EM,max} \leq 0.
\end{equation}

\subsubsection{Energy supply system model} \label{ess_model}
The energy supply system does not only need to provide the propulsive energy but also the hotel load \gls{P_aux}, which is considered constant during the whole trip.
Using this value, as well as the electric machine input force from equation \eqref{eq:dE_EM}, the overall fictive force balance is
\begin{equation} \label{eq:overall_energy_balance}
    k_\text{p}  \frac{\gls{F_dE_p}}{\tilde{\eta}_{\mathrm{EM}}} + \gls{P_aux}  \gls{y_t} + \gls{F_batd} =
    \gls{k_c}  \gls{F_c} + \gls{F_bat}  \eta_{\mathrm{DC/DC,batt}}
\end{equation}
including the energy converter output force \gls{F_c}, the battery output force \gls{F_bat} and the force of battery losses \gls{F_batd}. In \eqref{eq:overall_energy_balance}, a new auxiliary variable
\begin{equation} \label{eq:intro_yt}
    \gls{y_t} = \frac{1}{\sqrt{\gls{b}}}
\end{equation}
is introduced.
Since $1/\sqrt{\gls{b}}$ is a convex expression, the inequality
\begin{equation}
    \frac{1}{\sqrt{\gls{b}}} - \gls{y_t} \leq 0\
\end{equation}
can be included in the convex optimization problem formulation.

Since the electric machine input force is given per machine respectively per propeller, the number of propellers $k_p$ must be included to retrieve the overall demand for propulsion.
In addition, the efficiency for the DC/DC converter of the battery $\eta_{\mathrm{DC/DC,batt}}$ needs to be considered.

Due to the convex formulation of the fictive propulsive force in \eqref{eq:dE_p}, the fictive force balance \eqref{eq:overall_energy_balance} must be relaxed to
\begin{equation} \label{eq:force_balance_convex}
    k_\text{p}  \frac{\gls{F_dE_p}}{\tilde{\eta}_{\mathrm{EM}}} + \gls{P_aux}  \gls{y_t} + \gls{F_batd} -
    \left( \gls{k_c}  \gls{F_c} + \gls{F_bat}  \eta_{\mathrm{DC/DC,batt}}\right) \leq 0.
\end{equation}
In Section \ref{sec:battery_model}, it is shown in \eqref{eq:F_batd} that the expression for battery losses is also convex.

As the forces in equation \eqref{eq:overall_energy_balance} correspond to the output power, the respective efficiency of each energy supply component needs to be taken into account to determine the change in the actual energy stored within the respective energy storage.
The battery combines energy supply and storage in one component whereas the general energy converter only converts and supplies the energy stored in some kind of tank.

\subsubsection{General energy converter model}
In addition to the battery, the energy supply system includes $K$ general energy converters, e.g., fuel cells or generating sets comprising a generator and a diesel engine.
All $K$ converters are assumed to be of the same making.
The required power input $P_{\mathrm{c,in}}$, i.e., the power provided by the fuel, of one converter is calculated as a linear function of the output power $P_{\mathrm{c}}$,
\begin{equation*}
    P_{\mathrm{c,in}} = \glsuseri{a_c} + (\glsuserii{a_c} - 1)  P_{\mathrm{c}},
\end{equation*}
which can be directly transferred into the fictive internal force
\begin{equation*}
    \gls{F_ci} = \glsuseri{a_c}  \gls{y_t} + (\glsuserii{a_c} - 1)  \gls{F_c}.
\end{equation*}
The coefficients $\gls{a_c}_{,i}$ are determined by the part-load efficiency characteristics of the respective converter.

\textbf{Operation of multiple converters.}
With multiple converters of the same making, it is possible to switch off some of the converters for legs of the trip with low power demand to increase the efficiency of the converters running.
Including this decision in the optimization model would require integer decision variables, thereby destroying the beneficial structure of the convex optimization problem.
However, since zero-emissions must be modeled, the number of converters turned on $\gls{k_c} \in \naturals$ is included in \eqref{eq:overall_energy_balance}.
Additionally, for the legs of the trip with a speed limit, it is possible to limit the number of converters turned on based on the sum of the minimum power required to propel the vessel at the speed limit. 
As the location of both, zero-emission zones and legs of the trip with a speed limit, is fully determined by the path, $\gls{k_c}(\gls{sigma})$ can be calculated prior to the optimization. 

\textbf{Energy converter power limits.}
As for the electric machine, the power limit of the energy converter translates to a force limit. 
The affine expression for the upper boundary on the converter force
\begin{equation} \label{eq:converter_limits_approx}
     \gls{F_c} - \frac{P_\text{c,max}  \sqrt{\gls{s}'_{12}}}{2  \gls{v_s}_\text{,r}}  \left( 3 - \frac{\gls{s}'_{12}  \gls{b}}{\gls{v_s}[^2]_\text{,r}}\right) \leq 0
\end{equation}
is obtained by a first-order Taylor approximation around the reference speed $\gls{v_s}_\text{,r}$.
The lower boundary is simply given by
\begin{equation} \label{eq:converter_lower_limit}
     -\gls{F_c} \leq 0.
 \end{equation}

\subsubsection{Battery system model} \label{sec:battery_model}
The operational behavior of a battery is determined by various electrochemical processes.
Different modeling approaches exist, ranging from detailed electrochemical to “black box” ones.
A typical approach in engineering applications is the equivalent circuit model because it combines sufficient accuracy with low computational cost \cite{zhang_BatteryModellingMethods_2014}.

The equivalent circuit model consists of a voltage source representing the \gls{ocv} of the battery, an internal resistance modelling the linear part of the losses and one or more RC networks for modeling the dynamic characteristics.
As this work focuses on energy consumption and its resulting costs, the battery dynamics are not as relevant, and the RC networks are not implemented to simplify the model.
The resulting equivalent circuit is depicted in Figure~\ref{fig:bat_model}.
Another assumption regarding the battery model is the temperature independence of \gls{ocv} and internal resistance.
This is justified if proper cooling is implemented, which we assume to be the case.

\begin{figure}
\centering
\includegraphics[width=0.6\columnwidth]{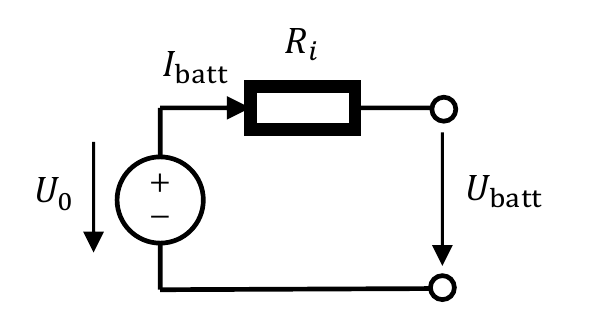}
\caption{Battery equivalent circuit model.}
\label{fig:bat_model}
\end{figure}

\textbf{Losses.}
The equivalent circuit is used to derive the load-dependent losses
\begin{equation} \label{eq:P_batd}
    P_\text{batt,d} = R_\text{i}  I_\text{batt}^2 = \frac{R_\text{i}}{U_0^2}  P_\text{batt}^2,
\end{equation}
which can be expressed as a function of the internal resistance $R_\text{i}$ and the battery current $I_\text{batt}$ as well as a function of the battery output power $P_\text{batt}$ and the \gls{ocv} $U_0$.
The right-hand side is derived by replacing the battery current with the quotient of battery output power and \gls{ocv}.
This is only possible if the voltage drop across the internal resistance is neglected.
This is a common approach when deriving battery models for convex optimization problems \cite{ma_ConvexModelingOptimal_2019, murgo_ConvexRelaxationsOptimal_2015}.
Other simplifications that are typical for convex battery modeling are the assumption of \gls{soc} -independent internal resistance and \gls{ocv} \cite{ma_ConvexModelingOptimal_2019, murgo_ConvexRelaxationsOptimal_2015, ebbes_TimeoptimalControlStrategies_2018, bordi_IncludingPowerManagement_2019}.

\textbf{Convexification of the losses.}
To add the battery model to the optimization model, equation \eqref{eq:P_batd} has to be transformed into the respective fictive force
\begin{equation} \label{eq:F_batd}
    \gls{F_batd} = \frac{R_\mathrm{i}}{U_0^2}  \gls{F_bat}  \sqrt{\gls{b}} = \frac{R_\mathrm{i}}{U_0^2}  \frac{\gls{F_bat}[^2]}{\gls{y_t}}.
\end{equation}

\textbf{\Gls{soc} update and capacity limits.}
In addition to the constraint regarding battery efficiency, several other constraints are required to model the battery operation accurately.
The first one is the battery energy \gls{E_bat} at an arbitrary point $l$ during the trip.
Using the initial battery energy at the beginning of the trip $E_{\mathrm{bat,0}}$, the change in battery energy content after travelling a certain distance is expressed as
\begin{equation*} 
    \gls{dE_bat}(l) = \gls{E_bat}(l) - E_{\mathrm{batt,0}}
\end{equation*}
resulting in the dynamics
\begin{equation*}
    \odv{\gls{dE_bat}}{\sigma} = - \gls{F_bat}\sqrt{\gls{s}'_{12}}.
\end{equation*}

The energy content of the battery, or the change in battery energy, respectively, is limited due to physical boundaries. 
Stricter boundaries,
\begin{subequations} \label{eq:E_bat_lim}
\begin{align}
    \gls{dE_bat}_{\mathrm{,min}} - \gls{dE_bat} &\leq 0,\\
    \gls{dE_bat} - \gls{dE_bat}_{\mathrm{,max}} &\leq 0
\end{align}
\end{subequations}
can be applied based on the minimum and maximum \gls{soc} of the battery, $\mathrm{SOC}_{\mathrm{min}}$ and $\mathrm{SOC}_{\mathrm{max}}$, respectively.
These are used to improve the battery lifetime and justify the constant voltage assumption in equation \eqref{eq:F_batd}.

\textbf{Battery power limits.}
Besides restricting the total battery energy, the battery output force also has to be limited.
This represents the power limit in the time domain.
With the maximum charge power $P_\text{cha, max}$ and the maximum discharge power $P_\text{dis, max}$ of the battery – both values are assumed to be non-negative – the force boundaries are approximated analogous to the one of the electric machine in \eqref{eq:F_EM_max_approx},
\begin{align*}
    F_\text{cha,max} &= \frac{P_\text{cha,max}}{2}  \frac{\sqrt{\gls{s}'_{12}}}{\gls{v_s}_\text{,r}}  \left( 3 - \frac{\gls{s}'_{12}  \gls{b}}{\gls{v_s}_\text{,r}^2}\right),\\
    F_\text{dis,max} &= \frac{P_\text{dis,max}}{2}  \frac{\sqrt{\gls{s}'_{12}}}{\gls{v_s}_\text{,r}}  \left( 3 - \frac{\gls{s}'_{12}  \gls{b}}{\gls{v_s}_\text{,r}^2}\right)
\end{align*}
resulting in the constraints
\begin{subequations} \label{eq:P_bat_lim}
\begin{align}
    -F_\text{cha,max} - \gls{F_bat} &\leq 0,\\ 
    \gls{F_bat} - F_\text{dis,max} &\leq 0
\end{align}
\end{subequations}
for limiting the output power of the battery in the convex optimization model.

\textbf{\gls{soc}-sustainability.}
Another constraint might be implemented depending on the desired type of operation regarding \gls{soc}-sustainability.
If the \gls{soc} of the battery should be the same at the end and the start of the trip the battery is operated \gls{soc}-sustaining.
In this case, the constraint
\begin{equation*} 
    \gls{dE_bat}(1) = 0
\end{equation*}
has to be added.

\subsection{Minimum time and energy problem in spatial domain}
Let the vector-valued function with the left-hand side of the given equations
\begin{equation*}
\begin{split}
    \Vec{g}(\gls{b},\glsuseri{u}) &=
    \left[ \begin{matrix} \eqref{eq:constraint_drift_force} & \eqref{eq:drag_convex} & \eqref{eq:boundaries_F_R_convex} & 
    \eqref{eq:rudder_drag} &
    \eqref{eq:z} & \eqref{eq:n_EM_lim} & \eqref{eq:Q_EM_lim} 
    \end{matrix} \right.\\
        &\qquad \quad \left. \begin{matrix} 
        \eqref{eq:P_EM_max} &
        \eqref{eq:force_balance_convex} & \eqref{eq:converter_limits_approx} & \eqref{eq:converter_lower_limit} & \eqref{eq:intro_yt} & \eqref{eq:E_bat_lim} & \eqref{eq:P_bat_lim}\end{matrix} \right]^\top
\end{split}
\end{equation*}
represent all the inequality functions in standard form.
With the state variables $\Vec{x} = \rowVecTrans{\gls{b} & \gls{dE_bat}}$ and the extended input vector
\begin{equation*}
\begin{split}
    \glsuseri{u} &= 
    \left[ \begin{matrix} \gls{F_D} & \gls{F_H} & \gls{F_R} & \gls{y_t} \end{matrix}\right.\\
        &\qquad \quad \left. \begin{matrix} \gls{D_R} & \gls{z}& \gls{n_p} & \gls{F_c} & \gls{F_bat} \end{matrix} \right]^\top
\end{split}
\end{equation*}
the minimum time and fuel motion planning problem is
\begin{align*}
     {\text{min.}} \quad & \int_{0}^{1} \left[ \left(\gls{k_c} \glsuseri{a_c} + \gls{w_t} \right) \gls{y_t} + \gls{k_c} \glsuserii{a_c} \gls{F_c} \,  \right] \text{d}\gls{sigma} \\
    \text{s.t. } \quad & \gls{R} \, \begin{bmatrix}
    \gls{T_p} - \gls{F_D} - \gls{D_R} \\
    \gls{F_H} + \gls{F_P} - \gls{F_R} \\
    \gls{F_H} + \gls{F_P} - \gls{F_R} \\
    \end{bmatrix} 
     = \gls{M} \, \left[ \gls{s}'\frac{\gls{b}'}{2} + \gls{s}'' \, \gls{b} \right] + \Vec{\tau}, \\
     & \gls{dE_bat}' = -\gls{F_bat}\sqrt{\gls{s}'_{12}},  \\
     & \Vec{g} \leq \vec{0}, \\
     & \gls{b}(0) = \gls{s}'_{12} v_\text{init}^2 \quad \gls{b}(1) =\gls{s}'_{12} v_\text{final}^2, \\
     & \gls{dE_bat}(0) = \gls{dE_bat}(1),
\end{align*}
where \gls{T_p} and \gls{F_P} are given in \eqref{eq:T_p_approx} and \eqref{eq:damp_force}, respectively. The first three  constraints are enforced for each $\sigma \in [0,1]$.

\subsection{Optimality conditions}
It is not obvious that the relaxations of \eqref{eq:z}, \eqref{eq:intro_yt} and \eqref{eq:overall_energy_balance} are valid, i.e., that they hold with equality at the optimum and the solution of the relaxed convex problem is the same as the solution of the original nonconvex problem.
\begin{theorem}
The relaxed inequalities \eqref{eq:z}, \eqref{eq:intro_yt} and \eqref{eq:overall_energy_balance} hold with equality at the optimum, i.e., 
\begin{align*}
    \gls{y_t}[^*] &= \frac{1}{\sqrt{\gls{b}[^*]}},\\
    \gls{z}[^*] &= \sqrt{\gls{s}'_{12}  \gls{b}[^*]  \gls{n_p}[^*]}, \\
    \gls{k_c}  \gls{F_c}[^*] &+ \gls{F_bat}[^*]  \eta_\mathrm{DC/DC,batt} =
        \gls{F_dE_p}[^*]  \frac{k_p}{\tilde{\eta}_{\mathrm{EM}}} + \gls{P_aux}  \gls{y_t}[^*]
\end{align*}
for $\sigma \in [0,1]$, if the propeller parameters satisfy the condition
    \begin{equation*}
        1 < \frac{a_{Q1}a_{T0}}{a_{Q0}a_{T1}}-\frac{a_{Q2}a_{T0}^2}{a_{Q0}a_{T1}^2}
    \end{equation*}
and the battery dissipation power satisfies the condition 
\begin{equation*}
     P_\mathrm{batt,d}(P_\mathrm{batt}) \leq \gls{P_aux}
\end{equation*}
and the electric machine operates below maximum speed limit 
\begin{equation*}
n_\text{EM} < i_\text{g}  n_p    
\end{equation*}
and below maximum torque limit
\begin{equation*}
    Q_\mathrm{EM} < \frac{\gls{Q_p}}{\eta_{\mathrm{g}}  i_\mathrm{g}}
\end{equation*}
and the converters deliver positive output power, which translates to the condition on force
\begin{equation*}
    \gls{F_c} > 0
\end{equation*}
for each $\sigma \in [0,1]$.
\end{theorem}
\begin{proof}
Let the scalar function $f_\text{obj}$ represent the objective function of the problem and the vector function $ \Vec{f}_\text{dyn}$  the dynamics $\gls{b}'$ and $\gls{dE_bat}'$, respectively.
The augmented Hamiltonian of the problem is given by
\begin{equation*}
    H = f_\text{obj} + \Vec{\Psi}^\top \, \Vec{f}_\text{dyn} + \Vec{\lambda}^\top \, \Vec{g},
\end{equation*}
where $\Vec{\Psi}$ is a vector of costates and $\Vec{\lambda}$ is a vector of nonnegative Lagrange multipliers.

The necessary conditions for optimality given in \cite{bryso_AppliedOptimalControl_2018} are stated.
\begin{enumerate}[(i)]
\item Adjoint system:
\begin{equation*}
    \Vec{\Psi}' = - \frac{\partial H}{\partial \Vec{x}}.
\end{equation*}
\item Minimum Principle:
\begin{equation*}
    \frac{\partial H}{\partial \glsuseri{u}}= \Vec{0}.
\end{equation*}
\item Complementary slackness:
\begin{equation*}
    \Vec{\lambda} \, \Vec{g} = \Vec{0}.
\end{equation*}
\end{enumerate}

\textbf{Force balance.} 
The Minimum Principle states that
\begin{equation*}
    \pdv{H}{\gls{F_c}} = \gls{k_c} \glsuserii{a_c} + \lambda_{F_\text{c,max}} - \gls{k_c} \lMultIneq{ESS} = 0,
\end{equation*}
where $\lMultIneq{ESS}$ and $\lambda_{F_\text{c,max}}$ are the elements of the vector $\lambda$ corresponding to the force balance and converter maximum force limit. Since $\lMultIneq{ESS} = \glsuserii{a_c} + \lambda_{F_\text{c,max}}/\gls{k_c} > 0$, complementary slackness implies that the inequality \eqref{eq:force_balance_convex} holds with equality at the optimum.

\textbf{Inverse speed squared.}
Similarly expanding the Minimum Principle gives
\begin{equation*}
    \pdv{H}{\gls{y_t}} = \gls{k_c} \glsuseri{a_c} + \gls{w_t} - \lMultIneq{\gls{y_t}} + \lMultIneq{ESS} \left( \gls{P_aux} - \frac{R_{\mathrm{i}}}{U_0^2} \frac{\gls{F_bat}[^2]}{\gls{y_t}[^2]} \right) = 0,
\end{equation*}
where $\lMultIneq{\gls{y_t}}$ is the Lagrange multiplier corresponding to the inequality \eqref{eq:z}. It follows that $\lMultIneq{\gls{y_t}} > 0$, if
\begin{equation*}
    \gls{k_c} \glsuseri{a_c} + \gls{w_t} + \left( \glsuserii{a_c} + \frac{\lambda_{F_\text{c,max}}}{ \gls{k_c}} \right) \left( \gls{P_aux} - \frac{R_i}{U_0^2} \frac{\gls{F_bat}[^2]}{\gls{y_t}[^2]} \right) > 0
\end{equation*}
holds. The condition reduces to
\begin{equation*}
    P_{\mathrm{batt,d}}(P_{\mathrm{batt}}) \leq P_{\mathrm{aux}} < P_{\mathrm{aux}} + \frac{ \gls{w_t} + \gls{k_c} \glsuseri{a_c}}{\glsuserii{a_c} +\lambda_{F_\text{c,max}}/\gls{k_c}}
\end{equation*}
since the last term on the right-hand side is positive.

\textbf{Propeller.}
Expanding $\partial H/\partial \gls{z} $ and $\partial H/\partial \gls{n_p}$ yields

\begin{align*}
     \dfrac{\partial H}{\partial z } & = - k_\text{xy}  \glsuserii{a_T} + \lMultIneq{\gls{z}} \\
     &\quad + \lMultIneq{ESS}  \frac{k_\text{p}}{\tilde{\eta}_{\mathrm{EM}}}  \left( - \glsuseriii{k_dEp}  -\frac{\gls{n_p}[^2]}{\gls{z}^2} \right) =0, \\
    \dfrac{\partial H}{\partial \gls{n_p} } & = k_\text{xy}  \glsuseri{a_T} - \lMultIneq{\gls{z}}  \frac{\gls{s}_{12}^{'2}  \gls{b}}{2\sqrt{ \gls{s}_{12}^{'2}  \gls{b}}  \gls{n_p}}\\
    &\quad + \lMultIneq{ESS}  \frac{k_\text{p}}{\tilde{\eta}_{\mathrm{EM}}}  \left( - \glsuserii{k_dEp} + \glsuseri{k_dEp} \frac{\gls{n_p}}{z} \right) =0,
\intertext{where}
    k_\text{xy} & = \left( \Psi_\text{x}  \frac{2}{\gls{s}'_1  \gls{m_vessel}}  \cos{\theta} + \Psi_\text{y}  \frac{2}{\gls{s}'_2  \gls{m_vessel}}  \sin{\theta} \right).
\end{align*}
Let $\lMultIneq{\gls{z}}^*=0$, which implies $z^* < \sqrt{(\gls{s}_{12}^{'2}\gls{b}[^*]\gls{n_p}[^*])}$ according to complementary slackness. By solving the expressions above with respect to $k_\text{xy}  \tilde{\eta}_{\mathrm{EM}}/\lMultIneq{ESS}  k_\text{p} $ and equating them, we obtain
\begin{equation*}
        \frac{\glsuseri{k_dEp}}{\glsuserii{a_T}}  \frac{\gls{n_p}^2}{\gls{z}^2} -
        2  \frac{\glsuseri{k_dEp}}{\glsuseri{a_T}}  \frac{\gls{n_p}}{\gls{z}}
        + \frac{\glsuseriii{k_dEp}}{\glsuserii{a_T}} + \frac{\glsuserii{k_dEp}}{\glsuseri{a_T}} = 0,
\end{equation*}
which is a quadratic equation in $\gls{n_p}/\gls{z}$.
Real valued roots only exist, if
\begin{equation*}
    4  \left( \frac{\glsuseri{k_dEp}}{\glsuseri{a_T}} \right)^2 - 4  \frac{\glsuseri{k_dEp}}{\glsuserii{a_T}}  \left(\frac{\glsuseriii{k_dEp}}{\glsuserii{a_T}} + \frac{\glsuserii{k_dEp}}{\glsuseri{a_T}} \right) \geq 0
\end{equation*}
holds, which simplifies to
\begin{equation*}
    \frac{\glsuseri{k_dEp}  \glsuserii{a_T}}{\glsuseri{a_T}^2} - \frac{\glsuseriii{k_dEp}}{\glsuserii{a_T}} - \frac{\glsuserii{k_dEp}}{\glsuseri{a_T}} \geq 0,
\end{equation*}
since $\glsuseri{k_dEp} > 0$.
If this condition does not hold, $\lMultIneq{\gls{z}} = 0$ is infeasible, and \eqref{eq:z} holds with equality at the optimum. 
\end{proof}
\begin{remark}
    In instances where the converter is off or idling, $\gls{F_c}=0$ and the Lagrange multiplier $\lambda_{F_\text{c,min}}$, corresponding to the converter minimum force limit, enters into the term $ \partial{H}/\partial\gls{F_c}$
    \begin{equation*}
        \pdv{H}{\gls{F_c}} = \gls{k_c}  \glsuserii{a_c} - \lambda_{F_\mathrm{c,min}} - \gls{k_c}  \lMultIneq{ESS} = 0
    \end{equation*}
    which does not necessitate $\lMultIneq{ESS} > 0$. This case can be treated by the approach introduced in \cite{murgo_ConvexRelaxationsOptimal_2015}, which splits the objective function integral to instances under electrical operation and instances with positive converter power.
\end{remark}

\section{Path}
\subsection{Path parametrization as a parametric polynomial curves}
We employ B\'ezier curves, a type of parametric polynomials, to parametrize the path function $s$ through a finite number of decision variables. B\'ezier curves are commonly used in higher-level algorithms for planning obstacle-avoiding trajectories. The curves exhibit favourable properties \cite{MPW+22,LSB09}:
\begin{enumerate}
    \item every point on the curve is a convex combination of the control points,
    \item every point on the curve is contained within the convex hull of the control points.
\end{enumerate}
The first property reduces the collision avoidance problem to a tractable convex optimization problem for placing the control points inside convex safe regions, that is, outside of obstacles. The second property ensures that the path is always feasible given that the control points are contained in safe regions.

 The curve $B$ is defined as a convex combination of points $P_0,...,P_n$, called \textit{control points}
\begin{equation*}
    B(\sigma) = \sum_{i=1}^n \Theta_{i,n}(\sigma)P_i,
\end{equation*}
where $\sigma \in [0,1]$. The coefficients of the convex combination are polynomials of degree $n$: 
\begin{equation*}
    \Theta_{i,n}(\sigma)= \begin{pmatrix} n \\ i  \end{pmatrix} \sigma^i (1-\sigma)^{n-i}, \ i=0,...,n.
\end{equation*}

The first and second derivatives of the path are needed for computing the magnitudes of forces due to translational motion, and for computing the orientation angle and angular speed. Moreover, the third derivative of the path is required for computing angular acceleration. To this end, we will make use of the property that the derivatives of $B$ are themselves B\'ezier curves, which allows us to compute the derivatives exactly at the discretization points $\sigma_0,...,\sigma_N$ in the problem implementation.

The $k$th derivative is 
\begin{equation} \label{eq:bezier_derivative}
    B^{[k]}(\sigma)=n(n-1)(n-2)\ldots (n-k+1) \sum_{i=0}^{n-k} \Theta_{n-k,i}(\sigma)D_i^k,
\end{equation}
where $D_i^k$ is the finite difference
\begin{align*}
    D_i^k & = D_{i+1}^{k-1} - D_{i}^{k-1}, \ i=0,...,n-k, \\
    D^1_{n-1} & = P_n - P_{n-1}.
\end{align*}
We will require that the curve has at least degree four, which ensures that it is at least three times differentiable. 

\subsection{Orientation and its derivatives}
The orientation of the vessel is expressed with respect to the positive $x$-axis as a function of the path coordinate $\sigma$:
\begin{equation*}
    s_3(\sigma) = \theta(\sigma) = \text{atan2}(s'_2, s'_1).
\end{equation*}
The derivative of $\theta$ with respect to $\sigma$ depends on $\sigma$ via both $s'_2$ and $s'_1$. Using the formula for total derivative gives
\begin{align*}
     s'_3  = \frac{d\theta}{d\sigma} & =\left( \frac{\partial}{\partial s'_2} \text{atan2}(s'_2, s'_1) \right) \frac{d s'_2}{d\sigma} \\
     & + \left( \frac{\partial}{\partial s'_1} \text{atan2}(s'_2, s'_1) \right) \frac{d s'_1}{d\sigma} \\
      & = \frac{s'_1}{s_{1}^{\prime 2} + s_{2}^{\prime 2}} s_2'' - \frac{s'_2}{s_{1}^{\prime 2} + s_{2}^{\prime 2}} s_1''.
\end{align*}
The second derivative is obtained by computing the total derivative of the above expression with respect to $\sigma$:
\begin{equation*}
    s''_3  = \frac{ s_2's_1'' k_1}{k_2^2} - \frac{ s_1's_2'' k_1}{k_2^2} - \frac{s_2' s_1'''}{k_2} + \frac{s_1' s_2'''}{k_2},
\end{equation*}
where
\begin{align*}
    k_1 & = 2s_1's_1'' + 2s_2's_2'', \\
    k_2 & = s_1^{\prime 2} + s_2^{\prime 2}.
\end{align*}

\section{Validation}
\subsection{Comparison of propeller models}
In this section, \gls{poly2} fitting of \gls{poly2} of \gls{kt} is evaluated for all 180 propellers of the Wageningen B-Series.
The configuration combinations in terms of blade number and expanded area ratio are given in \cite{vanl_WageningenBScrewSeries_1969}.
The pitch diameter ratio is varied between 0.6 and 1.4 in steps of 0.1. 

The \gls{poly2} of \gls{kt} and \gls{kq} is compared with a linear fit and a \gls{poly3}, which are typically used in marine vessel simulation models \cite{smoge_ControlMarinePropellers_2006, hausl_NovelFourQuadrantPropeller_2015}.
The average relative errors of \gls{kt}, \gls{kq} and the open water efficiency compared to the original curves based on equations \eqref{eq:K_T_poly} and \eqref{eq:K_Q_poly} are calculated for all propellers and all approximation variants.

\begin{figure}[h]
	\centering
	    \includegraphics[width=\columnwidth]{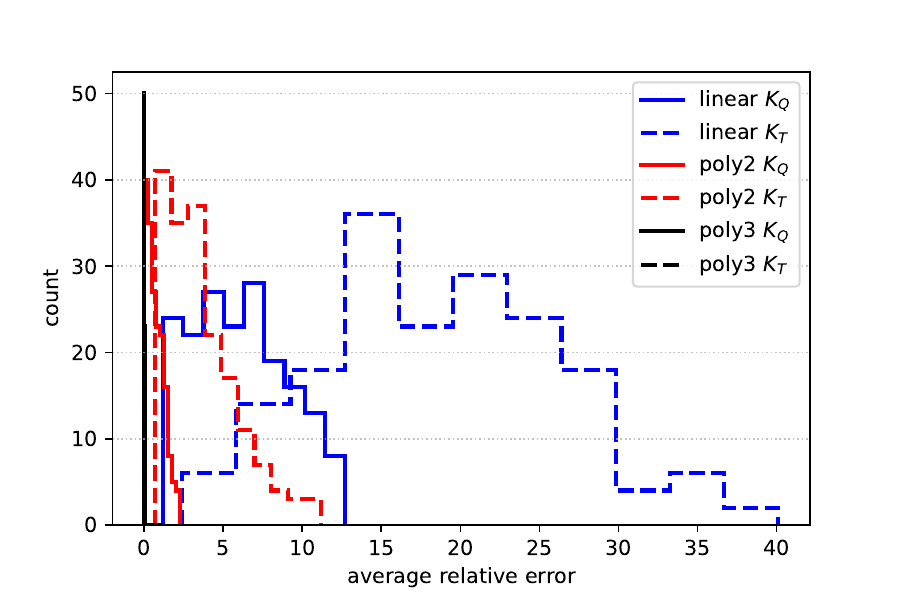}
	    \caption{Average relative error of thrust and torque coefficients for different approximation methods.}
	    \label{fig:hist_kt_err_avg}
\end{figure}

The histogram of average errors (Figures~\ref{fig:hist_kt_err_avg}) indicates that the \gls{poly2} achieves a lower average relative error for most propellers of the Wageningen B-Series than the linear fit.
The best results are attained with the \gls{poly3} which achieves an average relative error of less than 1\% in all three categories.
However, the \gls{poly3} cannot be integrated into the presented convex optimization model, and the \gls{poly2} delivers acceptable results.

Approximately 77\% of the \glspl{poly2} achieve an average relative error of less than 5\% in all three categories.
This is deemed tolerable since the motion planning model includes more prominent sources of error originating from the approximation of the hydrodynamic forces.
In addition, the average relative error for the \gls{poly2} is driven to these comparably high values due to the large deviation for high advance coefficients.
The operating points of a properly selected propeller are concentrated at the apex of the open water efficiency curve. Thus, the \gls{poly2} and the linear fit do not deviate as much from the \gls{poly3} for most advance coefficients as the average relative error suggests.
The relative error of \gls{poly2} around the most common operating point could be decreased further by approximating the \gls{kt} and \gls{kq} coefficients at the point of the most frequent operation since the resulting larger relative errors in other regions are acceptable.

\subsection{Hull form}
The 5415 modern surface vessel hull form is considered as our test case (Figure \ref{fig:hull_form}). The 5415 hull form is an open-source design established for benchmarking maneuvering simulation methods. Two open-water propellers, driven by strut-supported shafts, provide propulsion. The hull geometry and relevant loading conditions and speeds are described in Table \ref{tbl:ship_params}.
\begin{figure}
    \centering
    \includegraphics[width=0.5\textwidth]{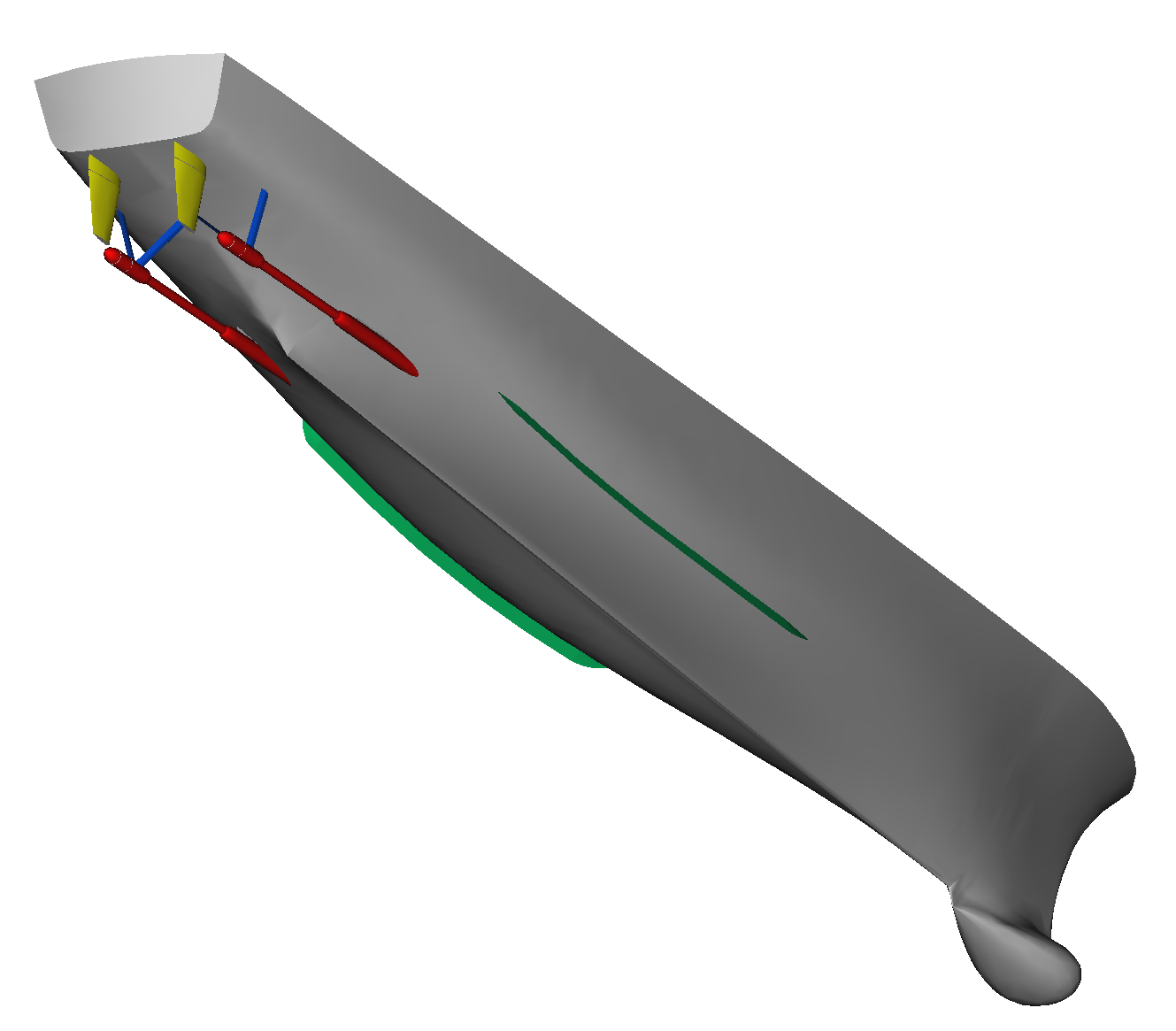}
    \caption{5415 hull form appended with bilge keels. The hull geometry includes a sonar dome in the bow.}
    \label{fig:hull_form}
\end{figure}
Captive and free model tests have been conducted for the 5154 hull in the towing tank of Maritime Research Institute Netherlands (MARIN). Captive tests target accurate measurement of hull forces and moments, while the IMO standard zig-zag maneuvers intended for assessing vessel maneuverability, are conducted with free sailing model. Test specifications are reported in \cite{Mar14a,Mar14b}. We will employ the physical model test data to benchmark against the convexified equations for yaw, drift and rudder forces and moments.
\begin{table}[h!]
\caption{Values of selected fixed parameters in the numerical examples.}\label{tbl:ship_params}
\begin{tabular*}{\tblwidth}{@{}lll@{}}
\toprule
    $L_\text{pp}$ = 4.002 m                     & $D$ = 0.173 m                          & $i_\text{g}$ = 0.1 \\
    $B_\text{wl}$ = 0.537 m                     & $f_\text{w}$ = 0.2                     & $\Omega$ = 20.2 \\
    $T$ = 0.173 m                               & $a_\text{T0}$ = 0.3                    &  $A_\text{R}$ = 0.012 m\textsuperscript{2}  \\
    $L_\text{R}$ = 1.85 m                        &  $a_\text{T2}$ = 0.35                 & $b_\text{R}$ = 0.124 m \\
    $L_\text{H}$ = 1 m                          &  $a_\text{Q0}$ = 0.041                 & $U_0$ = 48 V  \\
    $L_\text{P}$ = 9.5 m                        &  $a_\text{Q2}$ = 0.041                 & $R_\text{i}$ = 53 m$\Omega$ \\
    $\nabla$ = 0.189 m\textsuperscript{3}       &  $n_\mathrm{EM,max}$ = 1566 rad/s   & $E_{\mathrm{batt,max}}$ = 10 Wh \\
    $A_\text{S}$ = 2.361 m\textsuperscript{2}    &  $Q_\mathrm{EM,max}$ = 102 mNm     & $E_{\mathrm{batt},0}$ = 5 Wh\\
    $I_\text{zz}$ = 201 kg m\textsuperscript{2} &  $P_\mathrm{EM,max}$ = 60 W          & $\eta_{\mathrm{DC/DC,batt}}$ = 95\% \\
     $m$ = 189 kg                               &  $P_\mathrm{G,max}$ = 50 W           & $\eta_{g}$ = 98\%  \\
    $A_\text{L1}$ = 0.42                        &  $a_\mathrm{c0}$ = 0.174 mg/s                 & $P_{\mathrm{aux}}$ = 1 W \\
    $S$ = 1.18                                  &  $a_\mathrm{c1}$ = 0.945 mg/J                & $\rho$ = 997 kg/m\textsuperscript{3} \\
\bottomrule
\end{tabular*}
\end{table}
\subsection{Captive measurements}
Figure \ref{fig:benchmark} compares predictions of the mathematical model parameterized according to Table \ref{tbl:ship_params} to captive test measurements. Test results for hull hydrodynamic forces and moments are available for two speeds: 1.53 (9.26) m/s and 0.93 (5.56) m/s in model (full) scale. Test results for the rudder are available only for the higher speed at a number of discrete points. 

The measurements indicate that the rudder stalls when the angle of attack reaches approximately 20°, developing less lift thereafter. Since the mathematical model does not capture the behaviour of a stalled rudder, the maximum angle of attack is limited to 20°, and only values up to this angle are reported.
\begin{figure*}
    \centering
    \includegraphics[width=0.9\textwidth]{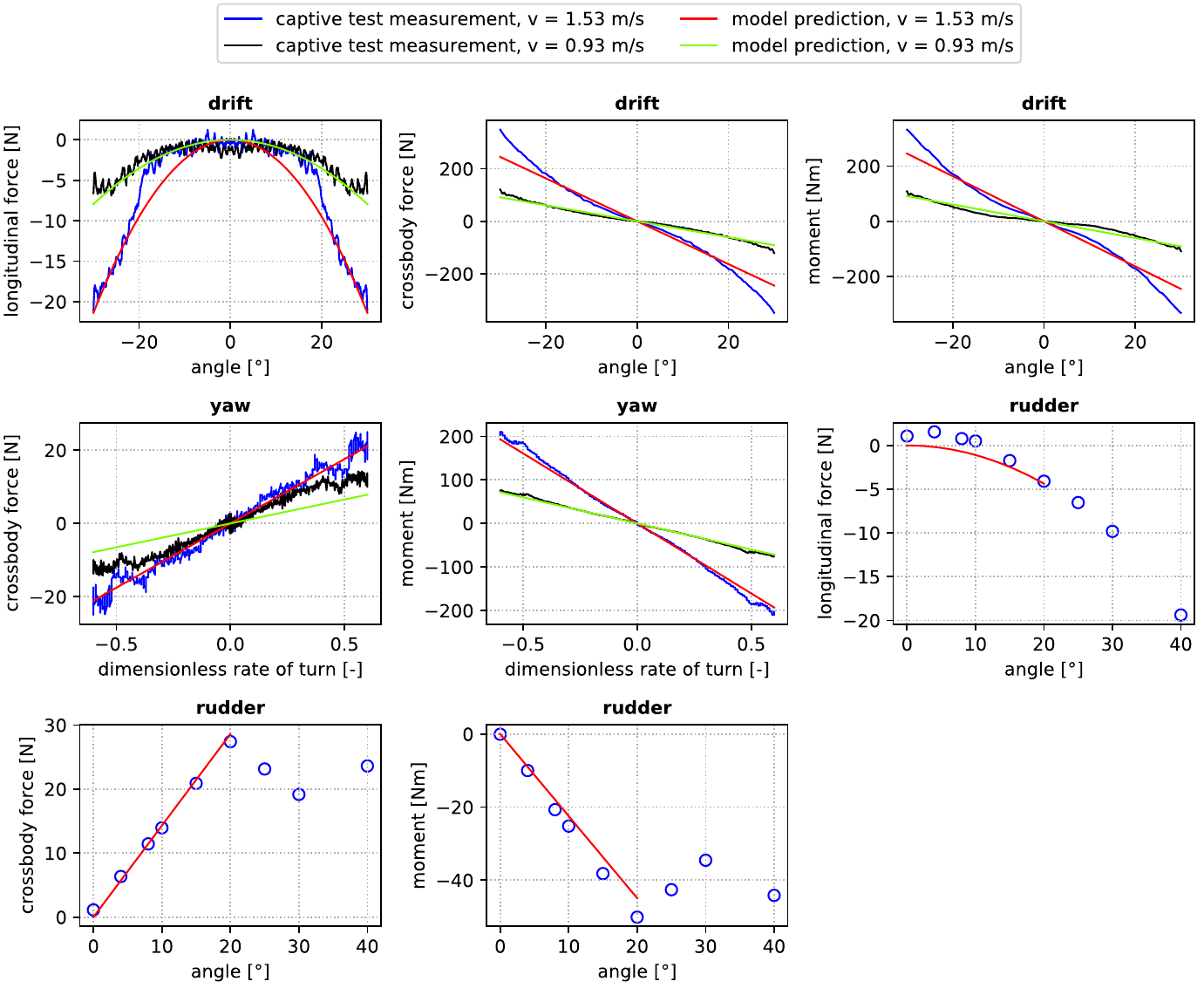}
    \caption{Comparison of captive test measurements and mathematical model predictions of hull and rudder hydrodynamic forces and moments.}
    \label{fig:benchmark}
\end{figure*}

The comparison shows good agreement of crossbody force and moment due to drift at low and moderate drift angles, while an underestimation of both force and moment are observed at a high angle and high speed. In contrast, the agreement of longitudinal force is good for high angles and poor for low angles. An underestimation of crossbody force is also observed in yawing at a low speed. Otherwise, the yaw force and moment are captured accurately by the mathematical model. 

The measured rudder force and moment exhibit near-linear response within the nominal operating range. Thus, the linear rudder model can be deemed accurate.
\subsection{Zig-zag maneuvering test}
The zig-zag maneuver is a typical maneuvering test that is performed to assess yaw response characteristics. The test begins at a steady speed with zero rudder angle. The rudder is then deflected to 20°. Once the vessel has turned 20°, the rudder is deflected to -20° and held until the vessel has turned to -20° with respect to the initial heading. The steps are repeated until steady oscillation is reached. 

The motion of the vessel, described by the convex hydrodynamic force elements, in the 20/20 zig-zag maneuver is predicted by numerical simulation and compared to tank test data of the same maneuver. In this case, the equations of motion (\ref{eq:time_dynamics}) in the time domain are numerically integrated using second-degree Runge-Kutta method with step size 0.01. 

The time histories of orientation and rate of turn show that the turning characteristics of the vessel are captured accurately by the convex mathematical model (Figure \ref{fig:benchmark}). However, the model underestimates drag due to drift, which is observed as a lag of the simulated vessel in the $x-y$-position subplot, and lower speed reduction.
\begin{figure*}
    \centering
    \includegraphics[width=0.8\textwidth]{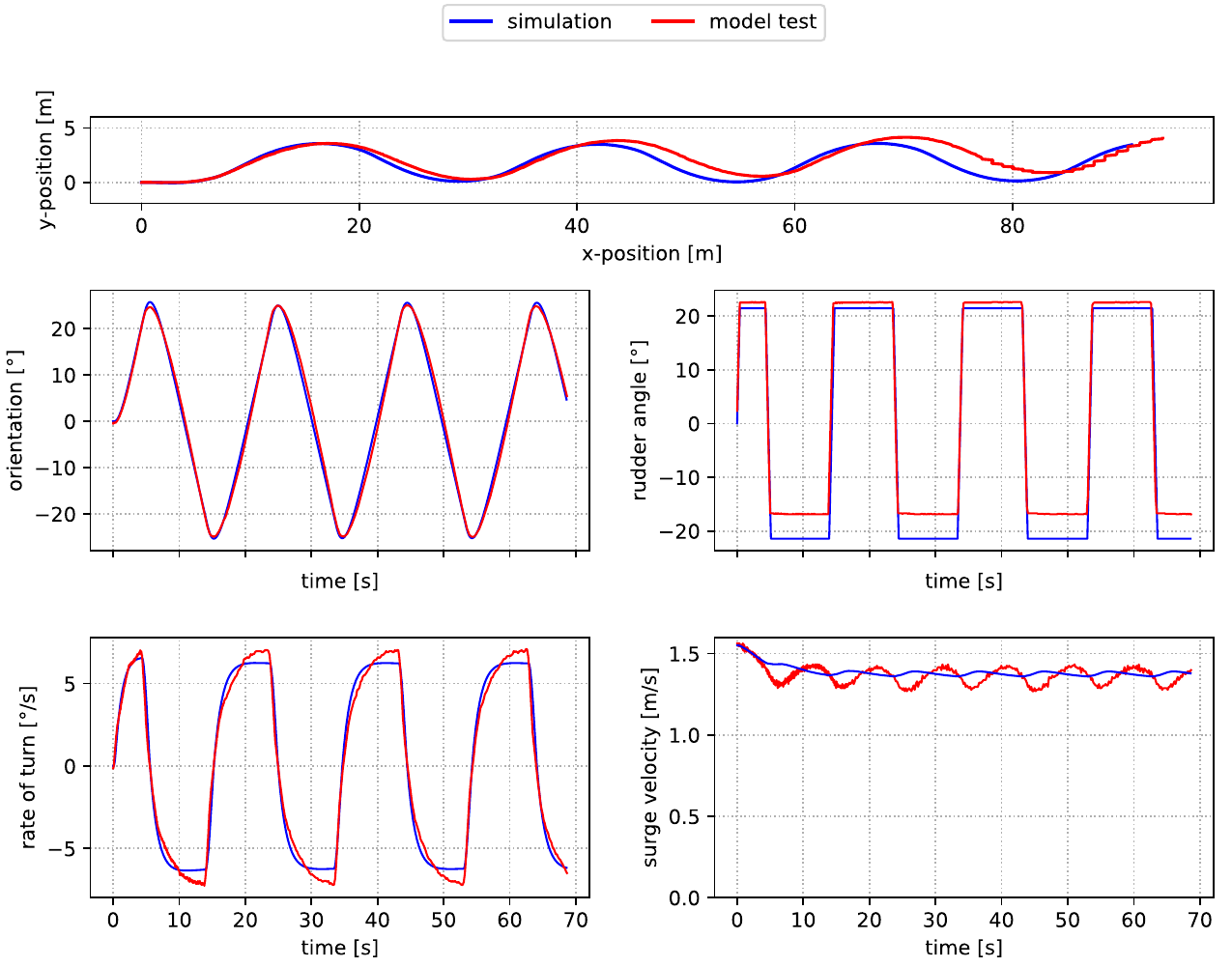}
    \caption{20/20 zig-zag maneuver comparison.}
    \label{fig:zz20_comparison}
\end{figure*}

\begin{figure*}
    \centering
    \includegraphics[width=\textwidth]{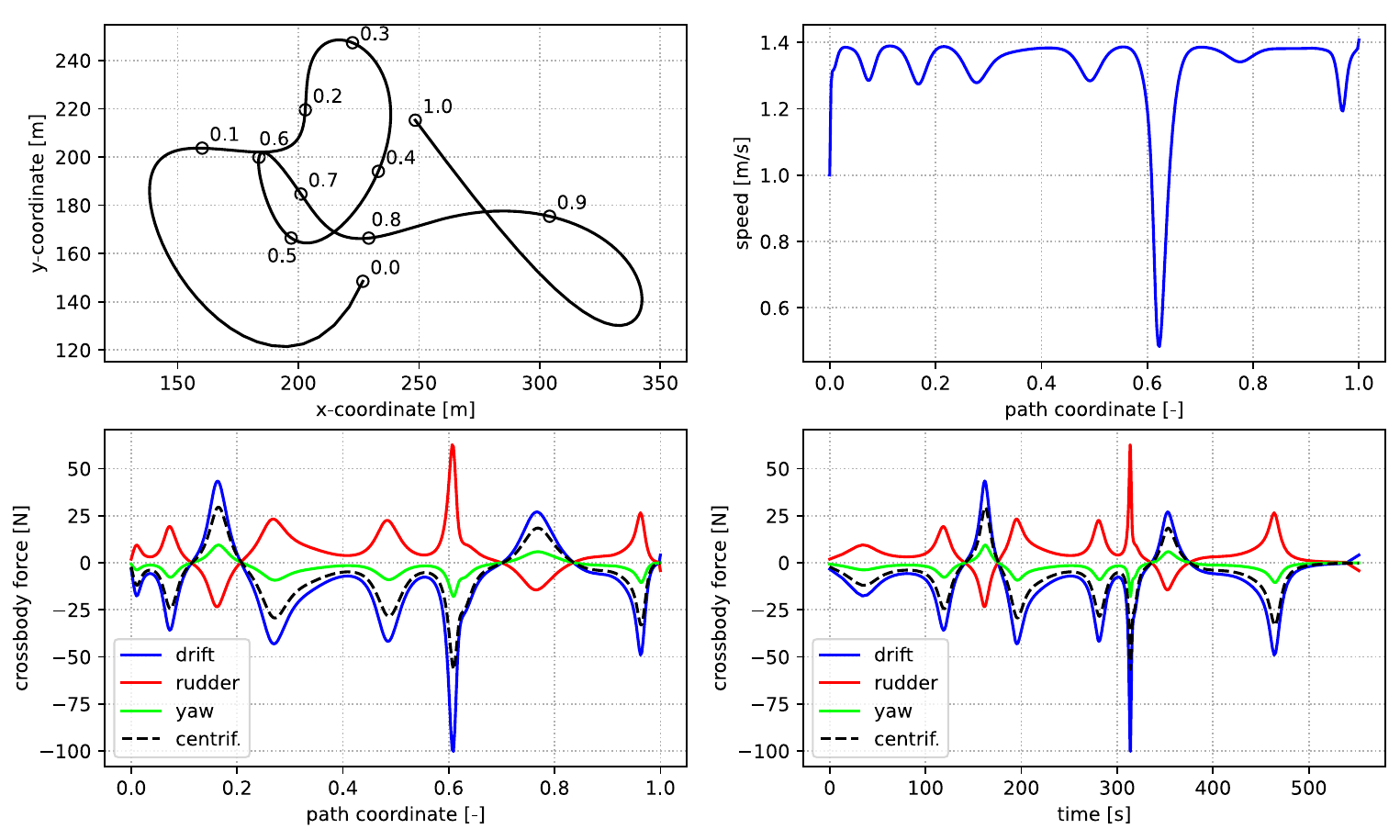}
    \caption{Randomly generated vessel path (top left), vessel speed (top right) and hull forces (bottom).}
    \label{fig:bezier_path_sol}
\end{figure*}

\begin{figure*}
    \centering
    \includegraphics[width=1.0\textwidth]{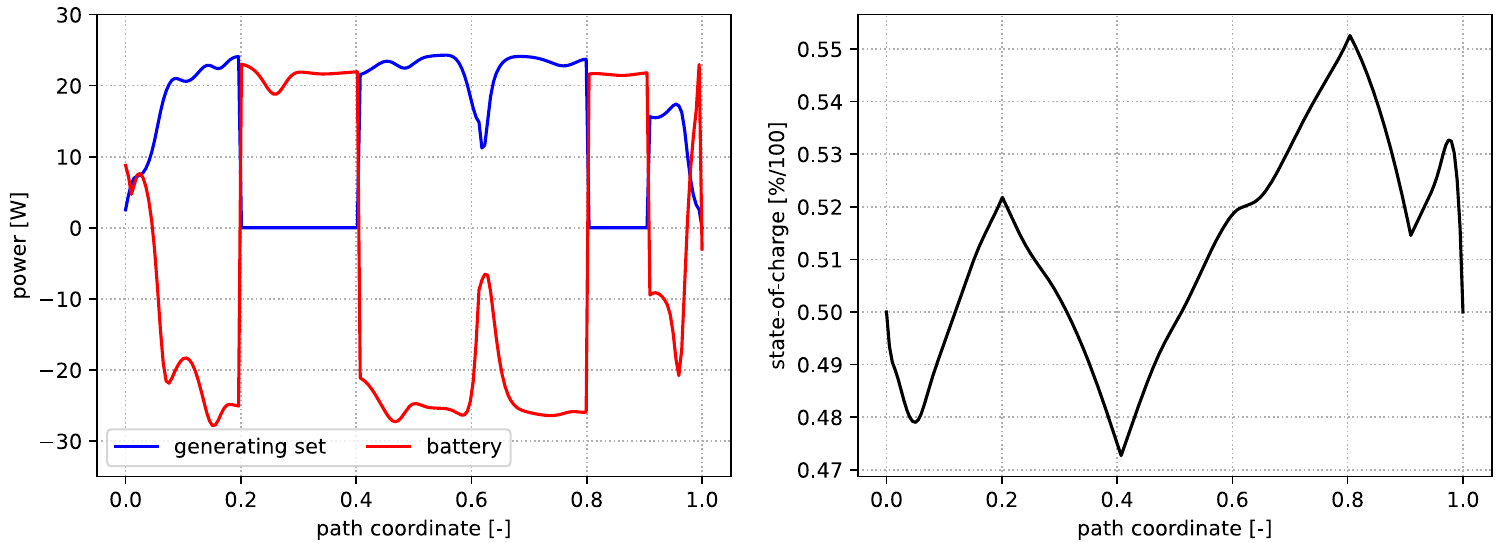}
    \caption{Solution of the problem instance with converter power limited sailing legs.}
    \label{fig:zem_sol}
\end{figure*}

\section{Numerical examples}
To make use of a numerical optimization algorithm, we create a finite dimensional problem by discretizing $\Vec{x}(\sigma)$ and $\glsuseri{u}(\sigma)$ at $n+1$ evenly spaced points such that
\begin{equation*}
    \sigma_{i+1} - \sigma_{i} = d\sigma
\end{equation*}
for all $i$. The finite difference approximation of the derivative is
\begin{equation*}
    \Vec{x}'(\sigma_i) \approx \frac{\Vec{x}(\sigma_{i+1})-\Vec{x}(\sigma_{i}) }{d\sigma}, \, i=1,...,n.
\end{equation*}
The path derivatives $s'(\sigma)$ and $s^{''}(\sigma)$ are evaluated at the discretization points according to (\ref{eq:bezier_derivative}).
The finite-dimensional problem is formulated with the CVXPY \cite{DB16} interface and solved with the primal-dual interior point algorithm implemented in the solver ECOS \cite{DCB13}. All the problem instances are solved in less than 0.5 s with $n=399$.

Figure \ref{fig:bezier_path_sol} (top left) depicts the polynomial curve that characterizes the vessel path in all problem instances. The polynomial is defined by 40 randomly generated control points. Although the discretization points of the path coordinate $\sigma$ are evenly spaced $(\sigma_{i+1} - \sigma_{i} = d\sigma)$, the points on the path $s(\sigma_i)$ need not be. This is evident in Figure \ref{fig:bezier_path_sol} that shows 11 path coordinate points along the path. Finally, Table \ref{tbl:ship_params} lists the fixed parameter values applied in all problem instances.

The optimal solution with initial speed 1.0 m/s is observed in Figure \ref{fig:bezier_path_sol}, which shows the time histories of vessel speed and forces acting on the hull and rudder.
The rest of this section investigates how the results vary for different requirements and input parameters.

\subsection{Tradeoff between voyage duration and fuel consumption}
The objective function of the hybrid vessel problem is a scalar objective that is a sum fuel consumption and voyage time objectives multiplied by the weight $\omega_T$. We aim to delimit a trade-off surface between these competing objectives, that is, the Pareto optimal points. Regardless of the prioritization (value of $\omega_T$) of each objective, if a given point is an optimal point for the problem, then the point is also Pareto optimal due to the convexity of the problem \cite{boyd_ConvexOptimization_2004}. This insight enables one to construct the Pareto surface by solving a sequence of problems with varying weights.
\begin{figure}
    \centering
    \includegraphics[width=\columnwidth]{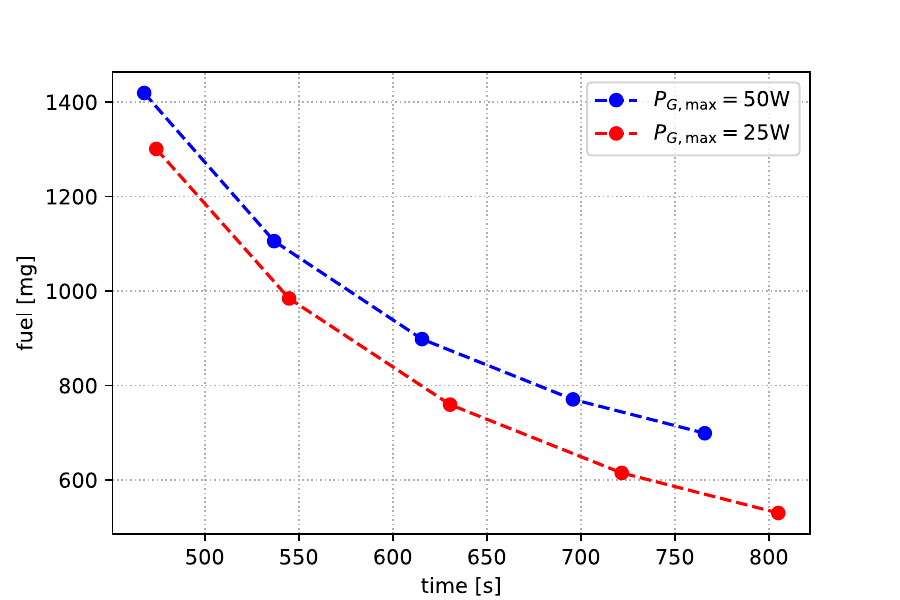}
    \caption{Trade-off between voyage total fuel consumption and sailing time}
    \label{fig:tradeoff}
\end{figure}

Figure \ref{fig:tradeoff} illustrates the trade-off curve for two vessels with converter configurations $P_\mathrm{G,max}=25$ and $P_\mathrm{G,max}=50$, but otherwise identical. The speed of modern convex program solvers allows generating such trade-off curves in only a few seconds.

The parameters $a_\mathrm{c0}$ and $a_\mathrm{c1}$ for the smaller converter were obtained by scaling the parameters of the larger converter such that both exhibit the same efficiency at the design point, i.e., the maximum power. Both vessels are equipped with two converters of the same size. The trade-off curves were generated by sampling the weight of the sailing time term in the objective from the set
\begin{equation*}
    \omega_T \in \{ 10, 3.33, 2, 1.25, 0.83\}.
\end{equation*}
Both trade-off curves exhibit increasing marginal fuel consumption for marginal reduction in sailing time. The low power configuration attains lower fuel consumption for any given voyage time, because the converters operate at higher efficiency closer to the design point. However, the fuel efficiency comes at the cost of maximum speed and minimum possible voyage time.

\subsection{Battery powered legs}
The hybrid power source of the vessel enables emission, noise and exhaust free operation on some subset of the voyage. This section implements two fully battery-powered voyage segments on path coordinate intervals $[0.2,0.4]$ and $[0.8,0.9]$. The requirement for battery charge sustainability for the complete voyage is preserved. Thus the investigation focuses on the optimal charging of the battery.

From Figure \ref{fig:zem_sol} it can be observed that the optimization algorithm ensures that the battery is charged with sufficient energy for the battery-powered legs while ensuring charge sustainability. In this case the battery discharging power limit does not impose a constraint on the speed of the battery-powered legs.

\section{Conclusion}
Modern marine transportation is moving in the direction of increased autonomy. Motivated by the need for fast and reliable decision making capability onboard, this work implements a physics-based convex optimization model for integrated hybrid power source supervisory control and dynamically feasible speed trajectory generation.
In the model formulation, the three-degrees-of-freedom vessel dynamics under a fixed pitch propeller model are convexified in the spatial domain by constraint relaxations and changes of variables.
The energy supply system consists of a fuel-to-electricity converter, e.g., a fuel cell or a generating set, and a battery to achieve local zero-emission and zero-noise operation. The whole energy system is represented by simplified loss models for each component.

The presented fixed pitch propeller model based on a physically reasonable second-order polynomial fitting of \gls{kt} and \gls{kq} is compared to other modeling variants, a linear fit and a \gls{poly3}.
The results indicate that the chosen approach yields sufficient accurate results around the most important operating points.
Reducing the error in the design point even further, e.g., by means of a Taylor expansion to second order about the propeller design point, could be explored in future work.
The analysis of 180 propellers of the Wageningen B-Series shows that the chosen model can be applied to a range of different propellers.

The velocity signal calculated by the planning algorithm is purely feedforward. A feedback controller needs to be implemented for tracking the signal. The optimal state and control signals respect the equations of motion of the vessel, i.e., they are dynamically feasible, which should leave only a minor tracking error for the feedback controller to clean up. Nevertheless, it is recommended to impose bounds on the control inputs which are lower than true allowable bounds. These conservative bounds account for modeling errors and prevent the feedback controller from becoming unstable \cite{versc_PracticalTimeOptimalTrajectory_2008}.

The generated trajectory - and the power demand prediction for the energy supply system - could be used in autonomous shipping as input for lower-level control.
This application requires the underlying optimization problem to be solvable reliably in real-time, which is achieved by the proposed convex optimization model.
In fact, it can be solved sufficiently fast to consider the whole voyage as the power demand prediction horizon if a specialized algorithm is implemented.
The special structure of the problem can be exploited to design custom algorithms \cite{lipp_MinimumtimeSpeedOptimisation_2014}.
Future research could focus on the implementation of this algorithm in combination with realizing the convex optimization model on a real-time platform to validate the simulations with data from a case study vessel.
Further developments of the presented optimization problem could cover a controllable pitch propeller model and the inclusion of additional resistance sources like shallow water, sea currents and the state of the sea depending on weather data.

\section*{Funding}
This research was funded by Business Finland’s Clean Propulsion Technologies project (ref. 38485/31/2020).

\section*{Acknowledgment}
The authors thank Dr. Janne Huotari for his contribution to the development of the  methodology presented in this manuscript.

\renewenvironment{theglossary}%
    {\begin{list}{}
        {%
        \setlength\labelwidth{0pt}%
        \setlength\itemindent{-\leftmargin}%
        \setlength\itemsep{0pt}%
        \setlength\parsep{4pt}%
        \let\makelabel\descriptionlabel}%
    }
    {\end{list}}
\printglossaries




\bibliographystyle{model1-num-names}

\bibliography{main.bib}

\begin{thebibliography}{70}
\expandafter\ifx\csname natexlab\endcsname\relax\def\natexlab#1{#1}\fi
\providecommand{\url}[1]{\texttt{#1}}
\providecommand{\href}[2]{#2}
\providecommand{\path}[1]{#1}
\providecommand{\DOIprefix}{doi:}
\providecommand{\ArXivprefix}{arXiv:}
\providecommand{\URLprefix}{URL: }
\providecommand{\Pubmedprefix}{pmid:}
\providecommand{\doi}[1]{\href{http://dx.doi.org/#1}{\path{#1}}}
\providecommand{\Pubmed}[1]{\href{pmid:#1}{\path{#1}}}
\providecommand{\bibinfo}[2]{#2}
\ifx\xfnm\relax \def\xfnm[#1]{\unskip,\space#1}\fi
\bibitem[{{United Nations Framework Convention on Climate
  Change}(2016)}]{unite_ParisAgreement_2016}
\bibinfo{author}{{United Nations Framework Convention on Climate Change}},
  \bibinfo{title}{Paris {Agreement}}, \bibinfo{year}{2016}.
\bibitem[{IMO(2018)}]{imo_NoteInternationalMaritime_2018}
\bibinfo{author}{IMO}, \bibinfo{title}{Note by the {International} {Maritime}
  {Organization} to the {UNFCCC} {Talanoa} {Dialogue} - {Adoption} of the
  initial {IMO} {Strategy} on reduction of {GHG} emissions from ships and
  existing {IMO} activity related to reducing {GHG} emissions in the shipping
  sector}, \bibinfo{type}{Technical Report}, IMO, \bibinfo{year}{2018}.
\bibitem[{Sofiev et~al.(2018)Sofiev, Winebrake, Johansson, Carr, Prank, Soares,
  Vira, Kouznetsov, Jalkanen, and Corbett}]{sofie_CleanerFuelsShips_2018}
\bibinfo{author}{M.~Sofiev}, \bibinfo{author}{J.~J. Winebrake},
  \bibinfo{author}{L.~Johansson}, \bibinfo{author}{E.~W. Carr},
  \bibinfo{author}{M.~Prank}, \bibinfo{author}{J.~Soares},
  \bibinfo{author}{J.~Vira}, \bibinfo{author}{R.~Kouznetsov},
  \bibinfo{author}{J.-P. Jalkanen}, \bibinfo{author}{J.~J. Corbett},
\newblock \bibinfo{title}{Cleaner fuels for ships provide public health
  benefits with climate tradeoffs},
\newblock \bibinfo{journal}{Nature Communications} \bibinfo{volume}{9}
  (\bibinfo{year}{2018}) \bibinfo{pages}{406}.
\bibitem[{Smith et~al.(2015)Smith, O'Keeffe, Aldous, Parker, Raucci, Traut,
  Corbett, Winebrake, Jalkanen, Johansson, Anderson, Agrawal, Ettinger, Ng,
  Hanayama, Faber, Nelissen, ´t Hoen, Lee, Chesworth, and
  Pandey}]{smith_ThirdIMOGHG_2015}
\bibinfo{author}{T.~Smith}, \bibinfo{author}{E.~O'Keeffe},
  \bibinfo{author}{L.~Aldous}, \bibinfo{author}{S.~Parker},
  \bibinfo{author}{C.~Raucci}, \bibinfo{author}{M.~Traut},
  \bibinfo{author}{J.~J. Corbett}, \bibinfo{author}{J.~J. Winebrake},
  \bibinfo{author}{J.-P. Jalkanen}, \bibinfo{author}{L.~Johansson},
  \bibinfo{author}{B.~Anderson}, \bibinfo{author}{A.~Agrawal},
  \bibinfo{author}{S.~Ettinger}, \bibinfo{author}{S.~Ng},
  \bibinfo{author}{S.~Hanayama}, \bibinfo{author}{J.~Faber},
  \bibinfo{author}{D.~Nelissen}, \bibinfo{author}{M.~´t Hoen},
  \bibinfo{author}{D.~Lee}, \bibinfo{author}{S.~Chesworth},
  \bibinfo{author}{A.~Pandey}, \bibinfo{title}{Third {IMO} {GHG} {Study}
  2014.}, \bibinfo{type}{Technical Report}, IMO, \bibinfo{address}{London},
  \bibinfo{year}{2015}.
\bibitem[{Ritari et~al.(2020)Ritari, Huotari, Halme, and
  Tammi}]{ritar_HybridElectricTopology_2020}
\bibinfo{author}{A.~Ritari}, \bibinfo{author}{J.~Huotari},
  \bibinfo{author}{J.~Halme}, \bibinfo{author}{K.~Tammi},
\newblock \bibinfo{title}{Hybrid electric topology for short sea ships with
  high auxiliary power availability requirement},
\newblock \bibinfo{journal}{Energy} \bibinfo{volume}{190}
  (\bibinfo{year}{2020}) \bibinfo{pages}{116359}.
\bibitem[{Geertsma et~al.(2017)Geertsma, Negenborn, Visser, Loonstijn, and
  Hopman}]{geert_PitchControlShips_2017}
\bibinfo{author}{R.~Geertsma}, \bibinfo{author}{R.~Negenborn},
  \bibinfo{author}{K.~Visser}, \bibinfo{author}{M.~Loonstijn},
  \bibinfo{author}{J.~Hopman},
\newblock \bibinfo{title}{Pitch control for ships with diesel mechanical and
  hybrid propulsion: {Modelling}, validation and performance quantification},
\newblock \bibinfo{journal}{Applied Energy} \bibinfo{volume}{206}
  (\bibinfo{year}{2017}) \bibinfo{pages}{1609--1631}.
\bibitem[{Baldi et~al.(2020)Baldi, Moret, Tammi, and Maréchal}]{BMT+20}
\bibinfo{author}{F.~Baldi}, \bibinfo{author}{S.~Moret},
  \bibinfo{author}{K.~Tammi}, \bibinfo{author}{F.~Maréchal},
\newblock \bibinfo{title}{The role of solid oxide fuel cells in future ship
  energy systems},
\newblock \bibinfo{journal}{Energy} \bibinfo{volume}{194}
  (\bibinfo{year}{2020}) \bibinfo{pages}{116811}.
\bibitem[{Tronstad et~al.(2017)Tronstad, Astrand, Haugom, and
  Langfeldt}]{emsa17}
\bibinfo{author}{T.~Tronstad}, \bibinfo{author}{H.~Astrand},
  \bibinfo{author}{P.~Haugom}, \bibinfo{author}{L.~Langfeldt},
  \bibinfo{title}{Study on the use of fuel cells in shipping},
  \bibinfo{type}{Technical Report}, {European Maritime Safety Agency},
  \bibinfo{year}{2017}.
\bibitem[{Geertsma et~al.(2017)Geertsma, Negenborn, Visser, and
  Hopman}]{geert_DesignControlHybrid_2017}
\bibinfo{author}{R.~Geertsma}, \bibinfo{author}{R.~Negenborn},
  \bibinfo{author}{K.~Visser}, \bibinfo{author}{J.~Hopman},
\newblock \bibinfo{title}{Design and control of hybrid power and propulsion
  systems for smart ships: {A} review of developments},
\newblock \bibinfo{journal}{Applied Energy} \bibinfo{volume}{194}
  (\bibinfo{year}{2017}) \bibinfo{pages}{30--54}.
\bibitem[{Kalikatzarakis et~al.(2018)Kalikatzarakis, Geertsma, Boonen, Visser,
  and Negenborn}]{kalik_ShipEnergyManagement_2018}
\bibinfo{author}{M.~Kalikatzarakis}, \bibinfo{author}{R.~Geertsma},
  \bibinfo{author}{E.~Boonen}, \bibinfo{author}{K.~Visser},
  \bibinfo{author}{R.~Negenborn},
\newblock \bibinfo{title}{Ship energy management for hybrid propulsion and
  power supply with shore charging},
\newblock \bibinfo{journal}{Control Engineering Practice} \bibinfo{volume}{76}
  (\bibinfo{year}{2018}) \bibinfo{pages}{133--154}.
\bibitem[{Liu et~al.(2016)Liu, Zhang, Yu, and Yuan}]{ZYX+16}
\bibinfo{author}{Z.~Liu}, \bibinfo{author}{Y.~Zhang}, \bibinfo{author}{X.~Yu},
  \bibinfo{author}{C.~Yuan},
\newblock \bibinfo{title}{Unmanned surface vehicles: An overview of
  developments and challenges},
\newblock \bibinfo{journal}{Annual Reviews in Control} \bibinfo{volume}{41}
  (\bibinfo{year}{2016}) \bibinfo{pages}{71--93}.
\bibitem[{IMO(2016)}]{imo_2016GuidelinesDevelopment_2016}
\bibinfo{author}{IMO}, \bibinfo{title}{2016 {Guidelines} for the development of
  a ship energy efficiency management plan ({SEEMP})}, \bibinfo{type}{Technical
  Report}, IMO, \bibinfo{year}{2016}.
\bibitem[{{DNV GL - Maritime}(2019)}]{dnvg_MARITIMEFORECAST2050_2019}
\bibinfo{author}{{DNV GL - Maritime}}, \bibinfo{title}{{Maritime} {forecast}
  {to} 2050. {Energy} {Transition} {Outlook} 2019}, \bibinfo{year}{2019}.
\bibitem[{Frangopoulos(2020)}]{frang_DevelopmentsTrendsChallenges_2020}
\bibinfo{author}{C.~A. Frangopoulos},
\newblock \bibinfo{title}{Developments, {Trends}, and {Challenges} in
  {Optimization} of {Ship} {Energy} {Systems}},
\newblock \bibinfo{journal}{Applied Sciences} \bibinfo{volume}{10}
  (\bibinfo{year}{2020}) \bibinfo{pages}{4639}.
\bibitem[{Böhme and Frank(2017)}]{bohme_HybridSystemsOptimal_2017}
\bibinfo{author}{T.~J. Böhme}, \bibinfo{author}{B.~Frank},
  \bibinfo{title}{Hybrid {Systems}, {Optimal} {Control} and {Hybrid}
  {Vehicles}: {Theory}, {Methods} and {Applications}}, Advances in {Industrial}
  {Control}, \bibinfo{publisher}{Springer International Publishing},
  \bibinfo{year}{2017}.
\bibitem[{Hoburg and Abbeel(2014)}]{hobur_GeometricProgrammingAircraft_2014}
\bibinfo{author}{W.~Hoburg}, \bibinfo{author}{P.~Abbeel},
\newblock \bibinfo{title}{Geometric {Programming} for {Aircraft} {Design}
  {Optimization}},
\newblock \bibinfo{journal}{AIAA Journal} \bibinfo{volume}{52}
  (\bibinfo{year}{2014}) \bibinfo{pages}{2414--2426}.
\bibitem[{Murgovski et~al.(2015)Murgovski, Johannesson, Hu, Egardt, and
  Sjoberg}]{murgo_ConvexRelaxationsOptimal_2015}
\bibinfo{author}{N.~Murgovski}, \bibinfo{author}{L.~Johannesson},
  \bibinfo{author}{X.~Hu}, \bibinfo{author}{B.~Egardt},
  \bibinfo{author}{J.~Sjoberg},
\newblock \bibinfo{title}{Convex relaxations in the optimal control of
  electrified vehicles},
\newblock in: \bibinfo{booktitle}{2015 {American} {Control} {Conference}
  ({ACC})}, \bibinfo{publisher}{IEEE}, \bibinfo{address}{Chicago, IL, USA},
  \bibinfo{year}{2015}, pp. \bibinfo{pages}{2292--2298}.
\bibitem[{Ebbesen et~al.(2018)Ebbesen, Salazar, Elbert, Bussi, and
  Onder}]{ebbes_TimeoptimalControlStrategies_2018}
\bibinfo{author}{S.~Ebbesen}, \bibinfo{author}{M.~Salazar},
  \bibinfo{author}{P.~Elbert}, \bibinfo{author}{C.~Bussi},
  \bibinfo{author}{C.~H. Onder},
\newblock \bibinfo{title}{Time-optimal {Control} {Strategies} for a {Hybrid}
  {Electric} {Race} {Car}},
\newblock \bibinfo{journal}{IEEE Transactions on Control Systems Technology}
  \bibinfo{volume}{26} (\bibinfo{year}{2018}) \bibinfo{pages}{233--247}.
\bibitem[{Açıkmeşe et~al.(2013)Açıkmeşe, Carson, and
  Blackmore}]{acikm_LosslessConvexificationNonconvex_2013}
\bibinfo{author}{B.~Açıkmeşe}, \bibinfo{author}{J.~M. Carson},
  \bibinfo{author}{L.~Blackmore},
\newblock \bibinfo{title}{Lossless {Convexification} of {Nonconvex} {Control}
  {Bound} and {Pointing} {Constraints} of the {Soft} {Landing} {Optimal}
  {Control} {Problem}},
\newblock \bibinfo{journal}{IEEE Transactions on Control Systems Technology}
  \bibinfo{volume}{21} (\bibinfo{year}{2013}) \bibinfo{pages}{2104--2113}.
\bibitem[{Boyd and Vandenberghe(2004)}]{boyd_ConvexOptimization_2004}
\bibinfo{author}{S.~P. Boyd}, \bibinfo{author}{L.~Vandenberghe},
  \bibinfo{title}{Convex optimization}, \bibinfo{publisher}{Cambridge
  University Press}, \bibinfo{address}{Cambridge, UK ; New York},
  \bibinfo{year}{2004}.
\bibitem[{Park et~al.(2015)Park, Sun, Pekarek, Stone, Opila, Meyer,
  Kolmanovsky, and DeCarlo}]{park_RealTimeModelPredictive_2015}
\bibinfo{author}{H.~Park}, \bibinfo{author}{J.~Sun},
  \bibinfo{author}{S.~Pekarek}, \bibinfo{author}{P.~Stone},
  \bibinfo{author}{D.~Opila}, \bibinfo{author}{R.~Meyer},
  \bibinfo{author}{I.~Kolmanovsky}, \bibinfo{author}{R.~DeCarlo},
\newblock \bibinfo{title}{Real-{Time} {Model} {Predictive} {Control} for
  {Shipboard} {Power} {Management} {Using} the {IPA}-{SQP} {Approach}},
\newblock \bibinfo{journal}{IEEE Transactions on Control Systems Technology}
  \bibinfo{volume}{23} (\bibinfo{year}{2015}) \bibinfo{pages}{2129--2143}.
\bibitem[{Makino et~al.(2017)Makino, Umeda, Ohtsuka, Ikejima, Sekiguchi,
  Tanizawa, Suzuki, and Fukazawa}]{makin_EnergySavingsShip_2017}
\bibinfo{author}{H.~Makino}, \bibinfo{author}{N.~Umeda},
  \bibinfo{author}{T.~Ohtsuka}, \bibinfo{author}{S.~Ikejima},
  \bibinfo{author}{H.~Sekiguchi}, \bibinfo{author}{K.~Tanizawa},
  \bibinfo{author}{J.~Suzuki}, \bibinfo{author}{M.~Fukazawa},
\newblock \bibinfo{title}{Energy savings for ship propulsion in waves based on
  real-time optimal control of propeller pitch and electric propulsion},
\newblock \bibinfo{journal}{Journal of Marine Science and Technology}
  \bibinfo{volume}{22} (\bibinfo{year}{2017}) \bibinfo{pages}{546--558}.
\bibitem[{Hou et~al.(2018)Hou, Sun, and
  Hofmann}]{hou_MitigatingPowerFluctuations_2018}
\bibinfo{author}{J.~Hou}, \bibinfo{author}{J.~Sun}, \bibinfo{author}{H.~F.
  Hofmann},
\newblock \bibinfo{title}{Mitigating {Power} {Fluctuations} in {Electric}
  {Ship} {Propulsion} {With} {Hybrid} {Energy} {Storage} {System}: {Design} and
  {Analysis}},
\newblock \bibinfo{journal}{IEEE Journal of Oceanic Engineering}
  \bibinfo{volume}{43} (\bibinfo{year}{2018}) \bibinfo{pages}{93--107}.
\bibitem[{Haseltalab(2019)}]{hasel_ModelPredictiveManeuvering_2019}
\bibinfo{author}{A.~Haseltalab},
\newblock \bibinfo{title}{Model predictive maneuvering control and energy
  management for all-electric autonomous ships},
\newblock \bibinfo{journal}{Applied Energy}  (\bibinfo{year}{2019})
  \bibinfo{pages}{27}.
\bibitem[{Huotari et~al.(2020)Huotari, Ritari, Vepsäläinen, and
  Tammi}]{huota_HybridShipUnit_2020}
\bibinfo{author}{J.~Huotari}, \bibinfo{author}{A.~Ritari},
  \bibinfo{author}{J.~Vepsäläinen}, \bibinfo{author}{K.~Tammi},
\newblock \bibinfo{title}{Hybrid {Ship} {Unit} {Commitment} with {Demand}
  {Prediction} and {Model} {Predictive} {Control}},
\newblock \bibinfo{journal}{Energies} \bibinfo{volume}{13}
  (\bibinfo{year}{2020}) \bibinfo{pages}{4748}.
\bibitem[{Wang et~al.(2018)Wang, Yan, Yuan, Jiang, Lin, and
  Negenborn}]{wang_DynamicOptimizationShip_2018}
\bibinfo{author}{K.~Wang}, \bibinfo{author}{X.~Yan}, \bibinfo{author}{Y.~Yuan},
  \bibinfo{author}{X.~Jiang}, \bibinfo{author}{X.~Lin}, \bibinfo{author}{R.~R.
  Negenborn},
\newblock \bibinfo{title}{Dynamic optimization of ship energy efficiency
  considering time-varying environmental factors},
\newblock \bibinfo{journal}{Transportation Research Part D: Transport and
  Environment} \bibinfo{volume}{62} (\bibinfo{year}{2018})
  \bibinfo{pages}{685--698}.
\bibitem[{Tzortzis and
  Frangopoulos(2019)}]{tzort_DynamicOptimizationSynthesis_2019}
\bibinfo{author}{G.~J. Tzortzis}, \bibinfo{author}{C.~A. Frangopoulos},
\newblock \bibinfo{title}{Dynamic optimization of synthesis, design and
  operation of marine energy systems},
\newblock \bibinfo{journal}{Proceedings of the Institution of Mechanical
  Engineers, Part M: Journal of Engineering for the Maritime Environment}
  \bibinfo{volume}{233} (\bibinfo{year}{2019}) \bibinfo{pages}{454--473}.
\bibitem[{Norstad et~al.(2011)Norstad, Fagerholt, and
  Laporte}]{norst_TrampShipRouting_2011}
\bibinfo{author}{I.~Norstad}, \bibinfo{author}{K.~Fagerholt},
  \bibinfo{author}{G.~Laporte},
\newblock \bibinfo{title}{Tramp ship routing and scheduling with speed
  optimization},
\newblock \bibinfo{journal}{Transportation Research Part C: Emerging
  Technologies} \bibinfo{volume}{19} (\bibinfo{year}{2011})
  \bibinfo{pages}{853--865}.
\bibitem[{Hvattum et~al.(2013)Hvattum, Norstad, Fagerholt, and
  Laporte}]{hvatt_AnalysisExactAlgorithm_2013}
\bibinfo{author}{L.~M. Hvattum}, \bibinfo{author}{I.~Norstad},
  \bibinfo{author}{K.~Fagerholt}, \bibinfo{author}{G.~Laporte},
\newblock \bibinfo{title}{Analysis of an exact algorithm for the vessel speed
  optimization problem},
\newblock \bibinfo{journal}{Networks} \bibinfo{volume}{62}
  (\bibinfo{year}{2013}) \bibinfo{pages}{132--135}.
\bibitem[{Zis et~al.(2015)Zis, North, Angeloudis, Ochieng, and
  Bell}]{zis_EnvironmentalBalanceShipping_2015}
\bibinfo{author}{T.~Zis}, \bibinfo{author}{R.~J. North},
  \bibinfo{author}{P.~Angeloudis}, \bibinfo{author}{W.~Y. Ochieng},
  \bibinfo{author}{M.~G.~H. Bell},
\newblock \bibinfo{title}{Environmental {Balance} of {Shipping} {Emissions}
  {Reduction} {Strategies}},
\newblock \bibinfo{journal}{Transportation Research Record: Journal of the
  Transportation Research Board} \bibinfo{volume}{2479} (\bibinfo{year}{2015})
  \bibinfo{pages}{25--33}.
\bibitem[{Gondzio(2012)}]{Gon12}
\bibinfo{author}{J.~Gondzio},
\newblock \bibinfo{title}{Interior point methods 25 years later},
\newblock \bibinfo{journal}{European Journal of Operational Research}
  \bibinfo{volume}{218} (\bibinfo{year}{2012}) \bibinfo{pages}{587--601}.
\bibitem[{Silvas et~al.(2016)Silvas, Hofman, Murgovski, Etman, and
  Steinbuch}]{SHM+16}
\bibinfo{author}{E.~Silvas}, \bibinfo{author}{T.~Hofman},
  \bibinfo{author}{N.~Murgovski}, \bibinfo{author}{L.~P. Etman},
  \bibinfo{author}{M.~Steinbuch},
\newblock \bibinfo{title}{Review of optimization strategies for system-level
  design in hybrid electric vehicles},
\newblock \bibinfo{journal}{IEEE Transactions on Vehicular Technology}
  \bibinfo{volume}{66} (\bibinfo{year}{2016}) \bibinfo{pages}{57--70}.
\bibitem[{Shi et~al.(2018)Shi, Xu, Tan, Kirschen, and
  Zhang}]{shi_OptimalBatteryControl_2018}
\bibinfo{author}{Y.~Shi}, \bibinfo{author}{B.~Xu}, \bibinfo{author}{Y.~Tan},
  \bibinfo{author}{D.~Kirschen}, \bibinfo{author}{B.~Zhang},
\newblock \bibinfo{title}{Optimal {Battery} {Control} {Under} {Cycle} {Aging}
  {Mechanisms} in {Pay} for {Performance} {Settings}},
\newblock \bibinfo{journal}{arXiv:1709.05715 [cs, math]}
  (\bibinfo{year}{2018}). \bibinfo{note}{ArXiv: 1709.05715}.
\bibitem[{Cai et~al.(2017)Cai, Zhang, Kim, Braun, and
  Hu}]{cai_ConvexOptimizationbasedControl_2017}
\bibinfo{author}{J.~Cai}, \bibinfo{author}{H.~Zhang}, \bibinfo{author}{D.~Kim},
  \bibinfo{author}{J.~E. Braun}, \bibinfo{author}{J.~Hu},
\newblock \bibinfo{title}{Convex optimization-based control of sustainable
  communities with on-site photovoltaic ({PV}) and batteries},
\newblock in: \bibinfo{booktitle}{2017 {IEEE} {Conference} on {Control}
  {Technology} and {Applications} ({CCTA})}, \bibinfo{publisher}{IEEE},
  \bibinfo{address}{Mauna Lani Resort, HI, USA}, \bibinfo{year}{2017}, pp.
  \bibinfo{pages}{1007--1012}.
\bibitem[{Verscheure et~al.(2008)Verscheure, Demeulenaere, Swevers, Schutter,
  and Diehl}]{versc_PracticalTimeOptimalTrajectory_2008}
\bibinfo{author}{D.~Verscheure}, \bibinfo{author}{B.~Demeulenaere},
  \bibinfo{author}{J.~Swevers}, \bibinfo{author}{J.~D. Schutter},
  \bibinfo{author}{M.~Diehl},
\newblock \bibinfo{title}{Practical {Time}-{Optimal} {Trajectory} {Planning}
  for {Robots}: a {Convex} {Optimization} {Approach}},
\newblock \bibinfo{journal}{IEEE Transactions on Automatic Control}
  (\bibinfo{year}{2008}) \bibinfo{pages}{10}.
\bibitem[{Sun et~al.(2020)Sun, Xu, Yuan, and
  Yang}]{sun_OptimalEnergyManagement_2020}
\bibinfo{author}{Y.~Sun}, \bibinfo{author}{Q.~Xu}, \bibinfo{author}{Y.~Yuan},
  \bibinfo{author}{B.~Yang},
\newblock \bibinfo{title}{Optimal {Energy} {Management} of {Fuel} {Cell}
  {Hybrid} {Electric} {Ships} {Considering} {Fuel} {Cell} {Aging} {Cost}},
\newblock in: \bibinfo{booktitle}{2020 {IEEE}/{IAS} {Industrial} and
  {Commercial} {Power} {System} {Asia} ({I} {CPS} {Asia})},
  \bibinfo{year}{2020}, pp. \bibinfo{pages}{240--245}.
\bibitem[{Abkowitz(1964)}]{Abk64}
\bibinfo{author}{M.~A. Abkowitz}, \bibinfo{title}{Lectures on ship
  hydrodynamics--Steering and manoeuvrability}, \bibinfo{type}{Technical
  Report}, Technical University of Denmark, \bibinfo{year}{1964}.
\bibitem[{{ABS}(2017)}]{Abs17}
\bibinfo{author}{{ABS}}, \bibinfo{title}{Guide for vessel maneuverability},
  \bibinfo{type}{Technical Report}, {American Bureau of Shipping},
  \bibinfo{year}{2017}.
\bibitem[{Lefeber et~al.(2003)Lefeber, Pettersen, and Nijmeijer}]{LPN03}
\bibinfo{author}{E.~Lefeber}, \bibinfo{author}{K.~Y. Pettersen},
  \bibinfo{author}{H.~Nijmeijer},
\newblock \bibinfo{title}{Tracking control of an underactuated ship},
\newblock \bibinfo{journal}{IEEE Transactions on Control Systems Technology}
  \bibinfo{volume}{11} (\bibinfo{year}{2003}) \bibinfo{pages}{52--61}.
\bibitem[{Nomoto et~al.(1957)Nomoto, Taguchi, Honda, and Hirano}]{NTH+57}
\bibinfo{author}{K.~Nomoto}, \bibinfo{author}{T.~Taguchi},
  \bibinfo{author}{K.~Honda}, \bibinfo{author}{S.~Hirano},
\newblock \bibinfo{title}{On the steering qualities of ships},
\newblock \bibinfo{journal}{International Shipbuilding Progress}
  \bibinfo{volume}{4} (\bibinfo{year}{1957}) \bibinfo{pages}{354--370}.
\bibitem[{Carlton(2007)}]{carlt_MarinePropellersPropulsion_2007}
\bibinfo{author}{J.~S. Carlton}, \bibinfo{title}{Marine {Propellers} and
  {Propulsion}}, \bibinfo{edition}{2} ed., \bibinfo{publisher}{Elsevier},
  \bibinfo{year}{2007}.
\bibitem[{van Lammeren et~al.(1969)van Lammeren, van Manen, and
  Oosterveld}]{vanl_WageningenBScrewSeries_1969}
\bibinfo{author}{W.~van Lammeren}, \bibinfo{author}{J.~van Manen},
  \bibinfo{author}{M.~W.~C. Oosterveld}, \bibinfo{title}{The {Wageningen}
  {B}-{Screw} {Series}}, \bibinfo{year}{1969}.
\bibitem[{Oosterveld and van
  Oossanen(1972)}]{ooste_RecentDevelopmentsMarine_1972}
\bibinfo{author}{M.~Oosterveld}, \bibinfo{author}{P.~van Oossanen},
\newblock \bibinfo{title}{Recent developments in marine propeller
  hydrodynamics},
\newblock in: \bibinfo{booktitle}{{NSMB} 40 {Years}, {International} {Jubilee}
  {Meeting} 1972/73}, \bibinfo{year}{1972}, pp. \bibinfo{pages}{52--98}.
\bibitem[{Oosterveld and van
  Oossanen(1975)}]{ooste_FurtherComputeranalyzedData_1975}
\bibinfo{author}{M.~Oosterveld}, \bibinfo{author}{P.~van Oossanen},
\newblock \bibinfo{title}{Further computer-analyzed data of the {Wageningen}
  {B}-screw series},
\newblock \bibinfo{journal}{International Shipbuilding Progress}
  \bibinfo{volume}{22} (\bibinfo{year}{1975}) \bibinfo{pages}{251--262}.
\bibitem[{Bertram(2012)}]{bertr_PracticalShipHydrodynamics_2012}
\bibinfo{author}{V.~Bertram}, \bibinfo{title}{Practical ship hydrodynamics},
  \bibinfo{edition}{2nd} ed.,
  \bibinfo{publisher}{Elsevier/Butterworth-Heinemann},
  \bibinfo{address}{Amsterdam; London}, \bibinfo{year}{2012}.
\bibitem[{Healey et~al.(1994)Healey, Rock, Cody, Miles, and
  Brown}]{heale_ImprovedUnderstandingThruster_1994}
\bibinfo{author}{A.~J. Healey}, \bibinfo{author}{M.~Rock},
  \bibinfo{author}{S.~Cody}, \bibinfo{author}{D.~Miles}, \bibinfo{author}{J.~P.
  Brown},
\newblock \bibinfo{title}{Toward an {Improved} {Understanding} of {Thruster}
  {Dynamics} for {Underwater} {Vehicles}},
\newblock in: \bibinfo{booktitle}{Proceedings of {IEEE} {Symposium} on
  {Autonomous} {Underwater} {Vehicle} {Technology} ({AUV}'94)},
  \bibinfo{address}{Cambridge, MA, USA}, \bibinfo{year}{1994}, pp.
  \bibinfo{pages}{340--352}.
\bibitem[{Kim and Chung(2006)}]{kim_AccuratePracticalThruster_2006}
\bibinfo{author}{J.~Kim}, \bibinfo{author}{W.~K. Chung},
\newblock \bibinfo{title}{Accurate and practical thruster modeling for
  underwater vehicles},
\newblock \bibinfo{journal}{Ocean Engineering} \bibinfo{volume}{33}
  (\bibinfo{year}{2006}) \bibinfo{pages}{566--586}.
\bibitem[{Smogeli(2006)}]{smoge_ControlMarinePropellers_2006}
\bibinfo{author}{O.~Smogeli}, \bibinfo{title}{Control of {Marine} {Propellers}:
  from {Normal} to {Extreme} {Conditions}}, \bibinfo{type}{{PhD} {Thesis}},
  Norwegian University of Science and Technology, \bibinfo{address}{Trondheim},
  \bibinfo{year}{2006}.
\bibitem[{Sørensen and Smogeli(2009)}]{soren_TorquePowerControl_2009}
\bibinfo{author}{A.~J. Sørensen}, \bibinfo{author}{O.~N. Smogeli},
\newblock \bibinfo{title}{Torque and power control of electrically driven
  marine propellers},
\newblock \bibinfo{journal}{Control Engineering Practice} \bibinfo{volume}{17}
  (\bibinfo{year}{2009}) \bibinfo{pages}{1053--1064}.
\bibitem[{Skulstad et~al.(2021)Skulstad, Li, Fossen, Vik, and
  Zhang}]{skuls_HybridApproachMotion_2021}
\bibinfo{author}{R.~Skulstad}, \bibinfo{author}{G.~Li}, \bibinfo{author}{T.~I.
  Fossen}, \bibinfo{author}{B.~Vik}, \bibinfo{author}{H.~Zhang},
\newblock \bibinfo{title}{A {Hybrid} {Approach} to {Motion} {Prediction} for
  {Ship} {Docking}—{Integration} of a {Neural} {Network} {Model} {Into} the
  {Ship} {Dynamic} {Model}},
\newblock \bibinfo{journal}{IEEE Transactions on Instrumentation and
  Measurement} \bibinfo{volume}{70} (\bibinfo{year}{2021})
  \bibinfo{pages}{1--11}.
\bibitem[{Blanke et~al.(2000)Blanke, Lindegaard, and
  Fossen}]{blank_DynamicModelThrust_2000}
\bibinfo{author}{M.~Blanke}, \bibinfo{author}{K.-P. Lindegaard},
  \bibinfo{author}{T.~I. Fossen},
\newblock \bibinfo{title}{Dynamic {Model} for {Thrust} {Generation} of {Marine}
  {Propellers}},
\newblock \bibinfo{journal}{IFAC Proceedings Volumes} \bibinfo{volume}{33}
  (\bibinfo{year}{2000}) \bibinfo{pages}{353--358}.
\bibitem[{Häusler et~al.(2013)Häusler, Saccon, Hauser, Pascoal, and
  Aguiar}]{hausl_FourQuadrantPropellerModeling_2013}
\bibinfo{author}{A.~J. Häusler}, \bibinfo{author}{A.~Saccon},
  \bibinfo{author}{J.~Hauser}, \bibinfo{author}{A.~M. Pascoal},
  \bibinfo{author}{A.~P. Aguiar},
\newblock \bibinfo{title}{Four-{Quadrant} {Propeller} {Modeling}: {A}
  {Low}-{Order} {Harmonic} {Approximation}},
\newblock \bibinfo{journal}{IFAC Proceedings Volumes} \bibinfo{volume}{46}
  (\bibinfo{year}{2013}).
\bibitem[{Hou et~al.(2018)Hou, Sun, and
  Hofmann}]{hou_AdaptiveModelPredictive_2018}
\bibinfo{author}{J.~Hou}, \bibinfo{author}{J.~Sun},
  \bibinfo{author}{H.~Hofmann},
\newblock \bibinfo{title}{Adaptive model predictive control with propulsion
  load estimation and prediction for all-electric ship energy management},
\newblock \bibinfo{journal}{Energy} \bibinfo{volume}{150}
  (\bibinfo{year}{2018}) \bibinfo{pages}{877--889}.
\bibitem[{Lipp and Boyd(2014)}]{lipp_MinimumtimeSpeedOptimisation_2014}
\bibinfo{author}{T.~Lipp}, \bibinfo{author}{S.~Boyd},
\newblock \bibinfo{title}{Minimum-time speed optimisation over a fixed path},
\newblock \bibinfo{journal}{International Journal of Control}
  \bibinfo{volume}{87} (\bibinfo{year}{2014}) \bibinfo{pages}{1297--1311}.
\bibitem[{Matusiak(2013)}]{Mat13}
\bibinfo{author}{J.~Matusiak}, \bibinfo{title}{{Dynamics of a Rigid Ship}},
  \bibinfo{type}{Technical Report}, Aalto University, \bibinfo{year}{2013}.
\bibitem[{Davidson and Schiff(1946)}]{DS46}
\bibinfo{author}{K.~S. Davidson}, \bibinfo{author}{L.~I. Schiff},
\newblock \bibinfo{title}{Turning and course keeping qualities},
\newblock \bibinfo{journal}{Stevens Institute of Technology, Experimental
  Towing Tank, Hoboken, New Yersey, Presented at The Society of Naval
  Architects and Marine Engineers, SNAME Transactions}  (\bibinfo{year}{1946}).
\bibitem[{{ITTC}(2011)}]{Itt11}
\bibinfo{author}{{ITTC}}, \bibinfo{title}{Recommended Procedures and Guidelines
  - Resistance Test}, \bibinfo{type}{Technical Report}
  \bibinfo{number}{7.5-02-02-01}, {ITTC}, \bibinfo{year}{2011}.
\bibitem[{Newman(2018)}]{New08}
\bibinfo{author}{J.~N. Newman}, \bibinfo{title}{Marine hydrodynamics},
  \bibinfo{publisher}{The MIT press}, \bibinfo{year}{2018}.
\bibitem[{Triantafyllou and Hover(2003)}]{TH03}
\bibinfo{author}{M.~Triantafyllou}, \bibinfo{author}{F.~Hover},
  \bibinfo{title}{Maneuvering and control of marine vehicles},
  \bibinfo{publisher}{Massachusetts Institute of Technology},
  \bibinfo{address}{Cambridge, Massachusetts USA}, \bibinfo{year}{2003}.
\bibitem[{Zhang et~al.(2014)Zhang, Li, Mcloone, and
  Yang}]{zhang_BatteryModellingMethods_2014}
\bibinfo{author}{C.~Zhang}, \bibinfo{author}{K.~Li},
  \bibinfo{author}{S.~Mcloone}, \bibinfo{author}{Z.~Yang},
\newblock \bibinfo{title}{Battery modelling methods for electric vehicles - {A}
  review},
\newblock in: \bibinfo{booktitle}{2014 {European} {Control} {Conference}
  ({ECC})}, \bibinfo{publisher}{IEEE}, \bibinfo{address}{Strasbourg, France},
  \bibinfo{year}{2014}, pp. \bibinfo{pages}{2673--2678}.
\bibitem[{Ma et~al.(2019)Ma, Murgovski, Egardt, and
  Cui}]{ma_ConvexModelingOptimal_2019}
\bibinfo{author}{Z.~Ma}, \bibinfo{author}{N.~Murgovski},
  \bibinfo{author}{B.~Egardt}, \bibinfo{author}{S.~Cui},
\newblock \bibinfo{title}{Convex modeling for optimal battery sizing and
  control of an electric variable transmission powertrain},
\newblock \bibinfo{journal}{Oil \& Gas Science and Technology – Revue d’IFP
  Energies nouvelles} \bibinfo{volume}{74} (\bibinfo{year}{2019})
  \bibinfo{pages}{25}.
\bibitem[{Bordin and Mo(2019)}]{bordi_IncludingPowerManagement_2019}
\bibinfo{author}{C.~Bordin}, \bibinfo{author}{O.~Mo},
\newblock \bibinfo{title}{Including power management strategies and load
  profiles in the mathematical optimization of energy storage sizing for fuel
  consumption reduction in maritime vessels},
\newblock \bibinfo{journal}{Journal of Energy Storage} \bibinfo{volume}{23}
  (\bibinfo{year}{2019}) \bibinfo{pages}{425--441}.
\bibitem[{Bryson(1975)}]{bryso_AppliedOptimalControl_2018}
\bibinfo{author}{A.~E. Bryson}, \bibinfo{title}{Applied optimal control:
  optimization, estimation and control}, \bibinfo{publisher}{CRC Press},
  \bibinfo{year}{1975}.
\bibitem[{Marcucci et~al.(2022)Marcucci, Petersen, von Wrangel, and
  Tedrake}]{MPW+22}
\bibinfo{author}{T.~Marcucci}, \bibinfo{author}{M.~Petersen},
  \bibinfo{author}{D.~von Wrangel}, \bibinfo{author}{R.~Tedrake},
  \bibinfo{title}{Motion planning around obstacles with convex optimization},
  \bibinfo{year}{2022}.
\bibitem[{Lau et~al.(2009)Lau, Sprunk, and Burgard}]{LSB09}
\bibinfo{author}{B.~Lau}, \bibinfo{author}{C.~Sprunk},
  \bibinfo{author}{W.~Burgard},
\newblock \bibinfo{title}{Kinodynamic motion planning for mobile robots using
  splines},
\newblock in: \bibinfo{booktitle}{2009 IEEE/RSJ International Conference on
  Intelligent Robots and Systems}, \bibinfo{organization}{IEEE},
  \bibinfo{year}{2009}, pp. \bibinfo{pages}{2427--2433}.
\bibitem[{Häusler et~al.(2015)Häusler, Saccon, Hauser, Pascoal, and
  Aguiar}]{hausl_NovelFourQuadrantPropeller_2015}
\bibinfo{author}{A.~J. Häusler}, \bibinfo{author}{A.~Saccon},
  \bibinfo{author}{J.~Hauser}, \bibinfo{author}{A.~M. Pascoal},
  \bibinfo{author}{A.~P. Aguiar},
\newblock \bibinfo{title}{A {Novel} {Four}-{Quadrant} {Propeller} {Model}},
\newblock in: \bibinfo{booktitle}{Proceedings of {Fourth} {International}
  {Symposium} on {Mariine} {Propulsors}}, \bibinfo{address}{Austin, Texas,
  USA}, \bibinfo{year}{2015}.
\bibitem[{{MARIN}(2014{\natexlab{a}})}]{Mar14a}
\bibinfo{author}{{MARIN}}, \bibinfo{title}{Model test specification - Appended
  Hull 5415 PMM tests in deep water}, \bibinfo{type}{Technical Report},
  {MARIN}, \bibinfo{year}{2014}{\natexlab{a}}.
\bibitem[{{MARIN}(2014{\natexlab{b}})}]{Mar14b}
\bibinfo{author}{{MARIN}}, \bibinfo{title}{Model test specification - 5415 free
  model tests in deep water}, \bibinfo{type}{Technical Report}, {MARIN},
  \bibinfo{year}{2014}{\natexlab{b}}.
\bibitem[{Diamond and Boyd(2016)}]{DB16}
\bibinfo{author}{S.~Diamond}, \bibinfo{author}{S.~Boyd},
\newblock \bibinfo{title}{{CVXPY}: {A} {P}ython-embedded modeling language for
  convex optimization},
\newblock \bibinfo{journal}{Journal of Machine Learning Research}
  \bibinfo{volume}{17} (\bibinfo{year}{2016}) \bibinfo{pages}{1--5}.
\bibitem[{Domahidi et~al.(2013)Domahidi, Chu, and Boyd}]{DCB13}
\bibinfo{author}{A.~Domahidi}, \bibinfo{author}{E.~Chu},
  \bibinfo{author}{S.~Boyd},
\newblock \bibinfo{title}{Ecos: An socp solver for embedded systems},
\newblock in: \bibinfo{booktitle}{2013 European Control Conference (ECC)},
  \bibinfo{organization}{IEEE}, \bibinfo{year}{2013}, pp.
  \bibinfo{pages}{3071--3076}.

\end{thebibliography}

\appendix

\section{Thrust and torque coefficient fitting}

\subsection{Coefficients of the propulsive thrust function} \label{app:T_p_coeffs}
\begin{gather*}
    \glsuseriii{a_T} = a_{T,2}  \gls{rho_sw}  D_p^2  (1 - f_\text{w})^2, \\
    \glsuserii{a_T} = a_{T,1}  \gls{rho_sw}  D_p^3  (1 - f_\text{w}), \\
    \glsuseri{a_T} = a_{T,0}  \gls{rho_sw}  D_p^4.
\end{gather*}

\subsection{Coefficients of the propulsive torque function} \label{app:Q_p_coeffs}
\begin{gather*}
    \glsuseriii{a_Q} = a_{Q,2}  \gls{rho_sw}  D_p^3  (1 - f_\text{w})^2, \\
    \glsuserii{a_Q} = a_{Q,1}  \gls{rho_sw}  D_p^4  (1 - f_\text{w}), \\
    \glsuseri{a_Q} = a_{Q,0}  \gls{rho_sw}  D_p^5.
\end{gather*}

\subsection{Coefficients of the change in propulsive energy function} \label{app:coeff_dE_p}
\begin{gather*}
    \glsuseriii{k_dEp} = 2\pi  \gls{rho_sw}  D_p^3  a_{Q,2}  (1-f_\text{w})^2,\\
    \glsuserii{k_dEp} = 2\pi  \gls{rho_sw}  D_p^4  a_{Q,1}  (1-f_\text{w}),\\
    \glsuseri{k_dEp} = 2\pi  \gls{rho_sw}  D_p^5  a_{Q,0}.
\end{gather*}

\section{Convexity of reduced propulsive power}
\label{app:power_convexity}
Let $a_1=2\pi  a_\text{Q,0}  \gls{D_p}[^5]$, $a_2=2\pi  \sqrt{\gls{s}'_{12}}  a_{Q,2}  \gls{D_p}[^3] $, $x_1 = \gls{n_p}$ and $x_2 = \gls{b}$. Using this notation, we define the function $f$ as
\begin{equation*}
    f(x_1, x_2)=a_1  \frac{x_1^2}{\sqrt{x_1  x_2}} - a_2\sqrt{x_1  x_2}, \quad (a_1,a_2) \in \mathbb{R}_{++}.
\end{equation*}
Convexity of $f$ for $(x_1,x_2) \in \mathbb{R}_{++}$ is shown as follows.
Let $y_1 \in \mathbb{R}, y_2 >0$ and $g(y_1,y_2)=y_1^2/y_2$.
The function $g$ is a standard quadratic over linear convex function that is increasing in $y_1$ for $y_1\geq0$ and decreasing in $y_2$. Let $h_1(x_1)=x_1$ and $h_2(x_1,x_2)=\sqrt{x_1  x_2}$. The function $h_1$ is linear, i.e., both convex and concave, and $h_2$ is a standard geometric mean concave function.
The function $f$ can be expressed as 
\begin{equation*}
f(x_1,x_2)=a_1  g\left(h_1(x_1),h_2(x_1,x_2)\right) - a_2  h_2(x_1,x_2).
\end{equation*}
According to the general composition theorem, $g\left(h_1(),h_2()\right)$ is a convexity preserving operation when $g$ is convex, $g$ is increasing in the first argument, $h_1$ convex and $g$ is decreasing in the second argument, and $h_2$ is concave \cite{boyd_ConvexOptimization_2004}.
The negative of a concave function is convex, so the composition $-h_2$ is convex.
Positive scalar multiplication and addition of convex functions are convexity-preserving operations.
Therefore, $f$ is convex.

\section{Augmented Hamiltonian}
\label{app:hamiltonian}
The augmented Hamiltonian of the problem is given by
\begin{equation}
\begin{split}
    H() = &\left( k_{\mathrm{c}}  \glsuseri{a_c} + \gls{w_t} \right)  \gls{y_t} + k_{\mathrm{c}}  \glsuserii{a_c}  \gls{F_c} +\\
    & \begin{bmatrix} \psi_x \\ \psi_y \\ \psi_\theta \end{bmatrix}^\top 2 \textbf{diag}(s')^{-1} \left( M^{-1}R \begin{bmatrix}
    \gls{T_p} - \gls{F_D} \\
    \gls{F_H} + \gls{F_P} - \gls{F_R} \\
    \gls{F_H} + \gls{F_P} - \gls{F_R} \\
    \end{bmatrix}  -s''b \right) \\
    &\lMultIneq{\gls{F_H}_{,1}}  \left( \gls{F_H} - \frac{\gls{rho_sw}}{2}  S  C_\mathrm{L,max}  \gls{s}'_{12}  \gls{b} \right) -\\
    &\lMultIneq{\gls{F_H}_{,2}}  \left( \gls{F_H} + \frac{\gls{rho_sw}}{2}  S  C_\mathrm{L,max}  \gls{s}'_{12}  \gls{b} \right) +\\
    &\lMultIneq{\gls{F_D}}  \left( \frac{\gls{rho_sw}}{2}  (C_{\mathrm{F}} + C_{\mathrm{R}})  A_{\mathrm{s}}  \gls{s}'_{12}  \gls{b} + \right.\\
    &\quad \left. \frac{2  \gls{F_H}^2  A_{\mathrm{s}}}{\gls{rho_sw}  \pi  \Omega  S^2  \gls{s}'_{12}  \gls{b}} - \gls{F_D} \right) +\\
    &\lMultIneq{\gls{F_R}_{,1}}  \left(\gls{F_R} - \frac{\gls{rho_sw}}{2}  C_\text{K}\left(\omega_\text{max} \right)  A_\text{R}  k_\text{tm}^2  \right.\\
    &\quad \left. \left( (1 - \gls{f_w})^2  \gls{s}'_{12}  \gls{b} + \frac{\gls{T_p}}{\frac{\gls{rho_sw}}{2}  \frac{\pi}{4}  \gls{D_p}[^2]} \right) \right)-\\
    &\lMultIneq{\gls{F_R}_{,2}}  \left(\gls{F_R} + \frac{\gls{rho_sw}}{2}  C_\text{K}\left(\omega_\text{max} \right)  A_\text{R}  k_\text{tm}^2  \right.\\
    &\quad \left. \left( (1 - \gls{f_w})^2  \gls{s}'_{12}  \gls{b} + \frac{\gls{T_p}}{\frac{\gls{rho_sw}}{2}  \frac{\pi}{4}  \gls{D_p}[^2]} \right) \right)+\\
    &\lMultIneq{\gls{z}}  \left( \gls{z} - \sqrt{\glsuseriv{s}_{12}^2  \gls{b}  \gls{n_p}} \right) +\\
    &\lMultIneq{ESS}  \left( k_\text{p}  \frac{\gls{F_dE_p}}{\tilde{\eta}_{\mathrm{EM}}} + \gls{P_aux}  \gls{y_t} + \gls{F_batd} - \right.\\
    &\quad \left. \gls{k_c}  \gls{F_c} - \gls{F_bat}  \eta_{\mathrm{DC/DC,batt}}\right)+\\
    &\lMultIneq{n_\text{EM,max}}  \left( \gls{n_p} - \left( \frac{n_\text{EM,max}}{i_\text{g}} \right)^2 \right) +\\
    &\lMultIneq{\gls{Q_p}}  \left( \gls{Q_p} - Q_{\mathrm{EM,max}}  \eta_\text{g}  i_\text{g} \right) +\\
    &\lMultIneq{F_\text{EM,max}}  \left( \gls{F_em} - F_\text{EM,max} \right) +\\
    &\lMultIneq{\gls{F_c}_\text{,min}}  \left( -\gls{F_c} \right) +\\
    &\lMultIneq{\gls{F_c}_\text{,max}}  \left(  \gls{F_c} - \frac{P_\text{c,max}  \sqrt{\gls{s}'_{12}}}{2  \gls{v_s}_\text{,r}}  \left( 3 - \frac{\gls{s}'_{12}  \gls{b}}{\gls{v_s}[^2]_\text{,r}}\right) \right) +\\
    &\lMultIneq{\gls{y_t}}  \left( \frac{1}{\sqrt{\gls{b}}} - \gls{y_t} \right) +\\
    &\lMultDiff{\gls{dE_bat}}  \left( -\gls{F_bat} \right) +\\
    &\lMultIneq{\gls{dE_bat}_\text{min}}  \left( \gls{dE_bat}_{\mathrm{,min}} - \gls{dE_bat} \right) +\\
    &\lMultIneq{\gls{dE_bat}_\text{max}}  \left( \gls{dE_bat} - \gls{dE_bat}_{\mathrm{,max}} \right) +\\
    &\lMultIneq{F_\text{cha, max}}  \left( -F_\text{cha,max} - \gls{F_bat} \right) +\\
    &\lMultIneq{F_\text{dis, max}}  \left( \gls{F_bat} - F_\text{dis,max} \right).
\end{split}
\end{equation}

\section{Derivation of dynamics in the spatial domain} \label{app:dynamics_spatial}
First, we rewrite the second time derivative of the configuration vector $\Ddot{\gls{q}}$ in (\ref{eq:time_dynamics}) using the path coordinate \gls{sigma} and the fixed path $s$. The first time derivative is
\begin{equation}
    \dot{\gls{q}}(t) = \frac{dq(t)}{dt}=\frac{ds(\gls{sigma}(t))}{dt}=\frac{ds(\gls{sigma}(t))}{d\gls{sigma}(t)} \frac{d \gls{sigma}(t)}{dt}
\end{equation}
where the last expression is obtained by applying the chain rule from calculus.
Using the shorthand notation $'$ to represent derivatives with respect to \gls{sigma}, the result above can be expressed concisely as
\begin{equation} \label{eq:q_dot_aux}
    \dot{\gls{q}}(t)=s'(\gls{sigma}(t))\dot{\gls{sigma}}(t)\,.
\end{equation}

The second derivative is (we drop the time dependency of $\gls{sigma}(t)$ and path dependency of $\gls{s}(\gls{sigma})$ for clarity)
\begin{equation}  \label{eq:2nd_derivative}
\begin{aligned}
    \ddot{\gls{q}}(t) & = \frac{d}{dt}\left( \frac{ds}{d\gls{sigma}} \frac{d \gls{sigma}}{dt} \right) = \frac{d}{dt}\frac{ds}{d\gls{sigma}}\frac{d\gls{sigma}}{dt}+\frac{ds}{d\gls{sigma}}\frac{d^2\gls{sigma}}{dt^2} \\
    & = \frac{d}{d\gls{sigma}}\left( \frac{ds}{d\gls{sigma}}\frac{d\gls{sigma}}{dt} \right) \frac{d\gls{sigma}}{dt} + \frac{ds}{d\gls{sigma}}\frac{d^2\gls{sigma}}{dt^2} \\
    & = \frac{d^2s}{d\gls{sigma}^2}\left( \frac{d\gls{sigma}}{dt} \right)^2 + \frac{ds}{d\gls{sigma}}\frac{d^2\gls{sigma}}{dt^2} \\ 
    &  = \gls{s}''(\gls{sigma}(t))\dot{\gls{sigma}}(t)^2 + \gls{s}'(\gls{sigma}(t))\ddot{\gls{sigma}}(t)
\end{aligned}
\end{equation}
where the third expression is obtained by applying product rule form calculus. The result \eqref{eq:2nd_derivative} is the same as the one derived in \cite{lipp_MinimumtimeSpeedOptimisation_2014}.

By applying (\ref{eq:2nd_derivative}) and (\ref{eq:path_definition}) to the dynamics (\ref{eq:time_dynamics}), we obtain a representation with respect to \gls{sigma}:
\begin{equation} \label{eq:dynamics_sigma}
    R(\gls{s}(\gls{sigma}))u = M(\gls{s}''(\gls{sigma})\dot{\gls{sigma}}^2 + \gls{s}'(\gls{sigma})\ddot{\gls{sigma}}).
\end{equation}
Here, $\dot{\gls{sigma}}^2$ is the square of the speed of the vessel along the path, and $\ddot{\gls{sigma}}$ is the acceleration along the path. Note that the terms $\gls{s}''(\gls{sigma})$ and $\gls{s}'(\gls{sigma})$ are obtained directly from the definition of the path. 

The nonlinear squared speed term in \eqref{eq:dynamics_sigma} renders the equality constraints non-convex. We will now introduce a new function that transforms (\ref{eq:dynamics_sigma}) to a convex form in a lossless manner.
Let
\begin{equation} \label{eq:b_defition}
    b(\gls{sigma}) = \dot{\gls{sigma}}^2.
\end{equation}
We observe that 
\begin{equation} \label{eq:dotb1}
    \dot{b}(\gls{sigma}) = \frac{db(\gls{sigma})}{dt}=\frac{db(\gls{sigma})}{d\gls{sigma}}\frac{d\gls{sigma}}{dt}=b'(\gls{sigma})\dot{\gls{sigma}}.
\end{equation}
Also directly from the definition \eqref{eq:b_defition} follows that
\begin{equation} \label{eq:dotb2}
    \dot{b}(\gls{sigma}) = \frac{d(\dot{\gls{sigma}})^2}{dt}=\frac{d}{dt} \left( \frac{d\gls{sigma}}{dt} \frac{d\gls{sigma}}{dt} \right) = \frac{d^2\gls{sigma}}{dt^2}\frac{d\gls{sigma}}{dt} + \frac{d^2\gls{sigma}}{dt^2}\frac{d\gls{sigma}}{dt} = 2 \ddot{\gls{sigma}}\dot{\gls{sigma}}.
\end{equation}
From the equivalency of \eqref{eq:dotb1} and \eqref{eq:dotb2} follows the relation
\begin{equation*}
    b'=2\ddot{\gls{sigma}}.
\end{equation*}
Since the derivative is a linear operator, the dynamics constraint expressed as
\begin{equation*}
    2R(\gls{s}(\gls{sigma}))u = 2M \left( \gls{s}''(\gls{sigma})b(\gls{sigma}) + \gls{s}'(\gls{sigma})b'(\gls{sigma}) \right)
\end{equation*}
is affine in the decision variables $b$ and $u$.

\printcredits



\end{document}